\documentclass[12pt,a4paper,reqno]{amsart}
\usepackage{amsmath}
\usepackage{geometry}
\usepackage{amssymb, latexsym}
\usepackage{verbatim}
\usepackage{enumerate}
\usepackage{mathtools}%                  http://www.ctan.org/pkg/mathtools
\usepackage[tableposition=top]{caption}% http://www.ctan.org/pkg/caption
\usepackage{booktabs,dcolumn}%           http://www.ctan.org/pkg/dcolumn + http://www.ctan.org/pkg/booktabs

\newtheorem{theorem}{Theorem}[section]
\newtheorem{proposition}[theorem]{Proposition}

\newtheorem{lemma}[theorem]{Lemma}
\newtheorem{corollary}[theorem]{Corollary}
\newtheorem{conjecture}[theorem]{Conjecture}

\newtheorem{remark}[theorem]{Remark}

\newcommand\E{\mathbb{E}}
\newcommand\R{\mathbb{R}}
\newcommand\Z{\mathbb{Z}}
\newcommand\N{\mathbb{N}}
\newcommand\C{\mathbb{C}}

\newcommand\T{\mathbb{T}}

\newcommand\eps{\varepsilon}

\newcommand\record{\sigma}  % the exponent that is a bit less than 2/7 that our methods reach
\newcommand\recordexplicit{{\frac{8}{33}=0.2424\dots}}

\renewcommand{\phi}{\varphi}
\begin{document}

%\begin{frontmatter}

\title[Correlations of von Mangoldt and divisor functions]{Correlations of the von Mangoldt and higher divisor functions I.  Long shift ranges}

\author{Kaisa Matom\"aki}
\address{Department of Mathematics and Statistics \\
University of Turku, 20014 Turku\\
Finland}
\email{ksmato@utu.fi}

\author{Maksym Radziwi{\l}{\l}}
\address{ Department of Mathematics \\
  McGill University \\
  Burnside Hall \\ Room 1005 \\
  805 Sherbrooke Street West \\
  Montreal \\ Quebec \\
  Canada\\
  H3A 0B9 }
\email{maksym.radziwill@gmail.com}

\author{Terence Tao}
\address{Department of Mathematics, UCLA\\
405 Hilgard Ave\\
Los Angeles CA 90095\\
USA}
\email{tao@math.ucla.edu}
\begin{abstract}  We study asymptotics of sums of the form $\sum_{X < n \leq 2X} \Lambda(n) \Lambda(n+h)$, $\sum_{X < n \leq 2X} d_k(n) d_l(n+h)$, $\sum_{X < n \leq 2X} \Lambda(n) d_k(n+h)$, and $\sum_n \Lambda(n) \Lambda(N-n)$, where $\Lambda$ is the von Mangoldt function, $d_k$ is the $k^{\operatorname{th}}$ divisor function, and $N,X$ are large.  Our main result is that the expected asymptotic for the first three sums holds for almost all $h \in [-H,H]$, provided that $X^{\record+\eps} \leq H \leq X^{1-\eps}$ for some $\eps>0$, where $\record \coloneqq \recordexplicit$, with an error term saving on average an arbitrary power of the logarithm over the trivial bound.  This improves upon results of Mikawa, Perelli-Pintz, and Baier-Browning-Marasingha-Zhao, who obtained statements of this form with $\record$ replaced by $\frac{1}{3}$.  We obtain an analogous result for the fourth sum for most $N$ in an interval of the form $[X, X + H]$ with $X^{\record+\eps} \leq H \leq X^{1-\eps}$.  

Our method starts with a variant of an argument from a paper of Zhan, using the circle method and some oscillatory integral estimates to reduce matters to establishing some mean-value estimates for certain Dirichlet polynomials associated to ``Type $d_3$'' and ``Type $d_4$'' sums (as well as some other sums that are easier to treat).  After applying H\"older's inequality to the Type $d_3$ sum, one is left with two expressions, one of which we can control using a short interval mean value theorem of Jutila, and the other we can control using exponential sum estimates of Robert and Sargos.  The Type $d_4$ sum is treated similarly using the classical $L^2$ mean value theorem and the classical van der Corput exponential sum estimates.  

In a sequel to this paper we will obtain related results for the correlations involving $d_k(n)$ for much smaller values of $H$ but with weaker bounds. 
\end{abstract}
\maketitle

%\end{frontmatter}
%%%%%%%%%%%%%%%%%%%%%%%%%

\section{Introduction}
\label{se:intro}
This paper (as well as the sequel \cite{mrt-corr2}) will be concerned with the asymptotic estimation of correlations of the form
\begin{equation}\label{co}
 \sum_{X < n \leq 2X} f(n) \overline{g(n+h)} 
\end{equation}
for various functions $f,g \colon \Z \to \C$ and large $X$, and for ``most'' integers $h$ in the range $|h| \leq H$ for some $H = H(X)$ growing in $X$ at a moderate rate; in this paper we will mostly be concerned with the regime where $H = X^\theta$ for some fixed $0 < \theta < 1$.  We will focus our attention on the particularly well studied correlations
\begin{align}
\sum_{X < n \leq 2X} \Lambda(n) &\Lambda(n+h) \label{lambda} \\
\sum_{X < n \leq 2X} d_k(n) &d_l(n+h) \label{d3} \\
\sum_{X < n \leq 2X} \Lambda(n) &d_k(n+h) \label{titchmarsh} \\
\sum_n \Lambda(n) &\Lambda(X-n) \label{goldbach} 
\end{align}
for fixed $k,l \geq 2$, where $\Lambda$ is the von Mangoldt function and 
$$d_k(n) \coloneqq  \sum_{n_1 \dotsm n_k = n} 1$$
is the $k^{\operatorname{th}}$ divisor function, adopting the convention that $\Lambda(n) = d_k(n)= 0$ for $n \leq 0$.  Of course, to interpret \eqref{goldbach} properly one needs to take $X$ to be an integer, and then one can split this expression by symmetry into what is essentially twice a sum of the form \eqref{co} with $X$ replaced by $X/2$, $f(n) \coloneqq  \Lambda(n)$, $g(n) \coloneqq  \Lambda(-n)$, and $h \coloneqq  -X$.  One can also work with the range $1 \leq n \leq X$ rather than $X < n \leq 2X$ for \eqref{lambda}, \eqref{d3}, \eqref{titchmarsh} with only minor changes to the arguments below.  As is well known, the von Mangoldt function $\Lambda$ behaves similarly in many ways to the divisor functions $d_k$ for $k$ moderately large, with identities such as the Linnik identity \cite{linnik} and the Heath-Brown identity \cite{hb-ident} providing an explicit connection between the two functions.  Because of this, we will be able to treat both $\Lambda$ and $d_k$ in a largely unified fashion.

In the regime when $h$ is fixed and non-zero, and $X$ goes to infinity, we have well established conjectures for the asymptotic values of each of the above expressions:

\begin{conjecture}\label{bigconj}  Let $h$ be a fixed non-zero integer, and let $k,l \geq 2$ be fixed natural numbers.
\begin{itemize}
\item[(i)] (Hardy-Littlewood prime tuples conjecture \cite{hl})  We have\footnote{See Section \ref{notation-sec} for the asymptotic notation used in this paper.}
\begin{equation}\label{hl-conj}
\sum_{X < n \leq 2X} \Lambda(n) \Lambda(n+h) = {\mathfrak S}(h) X + O(X^{1/2+o(1)})
\end{equation}
as $X \to \infty$, where the \emph{singular series} ${\mathfrak S}(h)$ vanishes if $h$ is odd, and is equal to
\begin{equation}\label{sing-def}
 {\mathfrak S}(h) \coloneqq  2 \Pi_2 \prod_{p|h: p>2} \frac{p-1}{p-2}
\end{equation}
when $h$ is even, where $\Pi_2 \coloneqq  \prod_{p>2} (1 - \frac{1}{(p-1)^2})$ is the twin prime constant.  
\item[(ii)]  (Divisor correlation conjecture \cite{vinogradov}, \cite{ivic}, \cite[Conjecture 3]{conrey})  We have\footnote{In \cite{vinogradov} it is conjectured (in the $k=l$ case) that the error term is only bounded by $O( x^{1-1/k+o(1)} )$, and in \cite{ivic} it is in fact conjectured that the error term is not better than this; see also \cite{iw} for further discussion.  Interestingly, in the function field case (replacing $\Z$ by $F_q[t]$) the error term was bounded by $O(q^{-1/2})$ times the main term in the large $q$ limit in \cite{abr}, but this only gives square root cancellation in the degree $1$ case $n=1$ and so does not seem to give strong guidance as to the size of the error term in the large $n$ limit.}
\begin{equation}\label{conrey-gonek}
\sum_{X < n \leq 2X} d_k(n) d_l(n+h) = P_{k,l,h}(\log X) X + O(X^{1/2+o(1)})
\end{equation}
as $X \to \infty$, for some polynomial $P_{k,l,h}$ of degree $k+l-2$.
\item[(iii)]  (Higher order Titchmarsh divisor problem)  We have
\begin{equation}\label{titch}
\sum_{X < n \leq 2X} \Lambda(n) d_k(n+h) = Q_{k,h}(\log X) X + O(X^{1/2+o(1)})
\end{equation}
as $X \to \infty$, for some polynomial $Q_{k,h}$ of degree $k-1$. 
\item[(iv)]  (Quantitative Goldbach conjecture, see e.g. \cite[Ch. 19]{ik})  We have
\begin{equation}\label{goldbach-conj}
\sum_{n} \Lambda(n) \Lambda(X-n) = {\mathfrak S}(X) X + O(X^{1/2+o(1)})
\end{equation}
as $X \to \infty$, where ${\mathfrak S}(X)$ was defined in \eqref{sing-def} and $X$ is restricted to be integer.
\end{itemize}
\end{conjecture}

\begin{remark}
The polynomials $P_{k,l,h}$ are in principle computable (see \cite{conrey} for an explicit formula), but they become quite messy in their lower order terms.  For instance, a classical result of Ingham \cite{ingham} shows that the leading term in the quadratic polynomial $P_{2,2,h}(t)$ is $(\frac{6}{\pi^2} \sum_{d|h} \frac{1}{d}) t^2$, but the lower order terms of this polynomial, computed in \cite{estermann} (with the sum $\sum_{X < n \leq 2X}$ replaced with the closely related sum $\sum_{n \leq X}$), are significantly more complicated.  A similar situation occurs for $Q_{k,h}$; see for instance \cite{fiorilli} for an explicit formula for $Q_{2,h}$.  The top degree terms of $P_{k,l,h}, Q_{k,h}$ are however easy to predict from standard probablistic heuristics: one should have
\begin{equation}\label{pkl}
 P_{k,l,h}(t) = \frac{t^{k-1}}{(k-1)!} \frac{t^{l-1}}{(l-1)!} \left(\prod_p {\mathfrak S}_{k,l,p}(h)\right) + O_{k,l,h}(t^{k+l-3})
\end{equation}
and
$$ Q_{k,h}(t) = \frac{t^{k-1}}{(k-1)!} \left(\prod_p {\mathfrak S}_{k,p}(h)\right) + O_{k,h}(t^{k-2})$$
where the local factors ${\mathfrak S}_{k,l,p}(h), {\mathfrak S}_{k,p}(h)$ are defined by the formulae\footnote{One can simplify these formulae slightly by observing that $\E d_{k,p}({\mathbf n}) = (1-\frac{1}{p})^{1-k}$ and $\E \Lambda_p({\bf n}) = 1$.}
$$ {\mathfrak S}_{k,l,p}(h) \coloneqq  \frac{ \E d_{k,p}({\mathbf n}) d_{l,p}({\mathbf n}+h) }{ \E d_{k,p}({\mathbf n}) \E d_{l,p}({\mathbf n}) }$$
and
$$ {\mathfrak S}_{k,p}(h) \coloneqq  \frac{ \E d_{k,p}({\mathbf n}) \Lambda_{p}({\mathbf n}+h) }{ \E d_{k,p}({\mathbf n}) \E \Lambda_{p}({\mathbf n}) }$$
where ${\mathbf n}$ is a random variable drawn from the profinite integers $\hat \Z$ with uniform Haar probability measure, $d_{k,p}({\mathbf n}) \coloneqq  \binom{v_p({\mathbf n})+k-1}{k-1}$ is the local component of $d_k$ at $p$ (with the $p$-valuation $v_p({\mathbf n})$ being the supremum of all $j$ such that $p^j$ divides ${\mathbf n}$), and $\Lambda_p({\mathbf n}) \coloneqq  \frac{p}{p-1} 1_{p \nmid {\mathbf n}}$ is the local component of $\Lambda$.  See \cite{nt} for an explanation of these heuristics and a verification of the asymptotic \eqref{pkl}, as well as an explicit formula for the local factor ${\mathfrak S}_{k,l,p}(h)$.  For comparison, it is easy to see that
$$ {\mathfrak S}(h) = \prod_p \frac{\E \Lambda_p({\mathbf n}) \Lambda_p({\mathbf n}+h)}{\E \Lambda_p({\mathbf n}) \E \Lambda_p({\mathbf n})}$$
for all non-zero integers $h$, and similarly
$$ {\mathfrak S}(X) = \prod_p \frac{\E \Lambda_p({\mathbf n}) \Lambda_p(X-{\mathbf n})}{\E \Lambda_p({\mathbf n}) \E \Lambda_p({\mathbf n})}$$
for all non-zero integers $X$.
\end{remark}

Conjecture \ref{bigconj} is considered to be quite difficult, particularly when $k$ and $l$ are large, even if one allows the error term to be larger than $X^{1/2+o(1)}$ (but still smaller than the main term). For instance it is a notorious open problem to obtain an asymptotic for the divisor correlations in the case $k = l = 3$.  The objective of this paper is to obtain a weaker version of Conjecture \ref{bigconj} in which one has less control on the error terms, and one is content with obtaining the asymptotics for \emph{most} $h$ in a given range $[h_0-H,h_0+H]$, rather than for \emph{all} $h$.  This is in analogy with our recent work on Chowla and Elliott type conjectures for bounded multiplicative functions \cite{mrt}, although our methods here are different\footnote{In particular, the arguments in \cite{mrt} rely heavily on multiplicativity in small primes, which is absent in the case of the von Mangoldt function, and in the case of the divisor functions $d_k$ would not be strong enough to give error terms of size $O_A(\log^{-A} x)$ times the main term.  In any event, the arguments in this paper certainly cannot work for $H$ slower than $\log X$ even if one assumes conjectures such as the Generalized Lindel\"of Hypothesis, the Generalized Riemann Hypothesis or the Elliott-Halberstam conjecture, as the $h=0$ term would dominate all of the averages considered here.}.  Our ranges of $h$ will be shorter than those in previous literature on Conjecture \ref{bigconj}, although they cannot be made arbitrarily slowly growing with $X$ as was the case for bounded multiplicative functions in \cite{mrt}.  In particular, the methods in this paper will certainly be unable to unconditionally handle intervals of length $X^{1/6-\eps}$ or shorter for any $\eps>0$, since it is not even known\footnote{See \cite{zacc} for the best known result in this direction.} currently if the prime number theorem is valid in most intervals of the form $[X,X+X^{1/6-\eps}]$, and such a result would easily follow from an averaging argument (using a well-known calculation of Gallagher \cite{gal-hl}) if we knew the prime tuples conjecture \eqref{hl-conj} for most $h = O(X^{1/6-\eps})$.  However, one can do much better than this if one assumes powerful conjectures such as the Generalized Lindel\"of Hypothesis (GLH), the Generalized Riemann Hypothesis (GRH), or the Elliott-Halberstam conjecture (EH). We plan to discuss some of these conditional results in more detail on another occasion. 

In the case of the divisor correlation conjecture \eqref{conrey-gonek} and the higher order Titchmarsh divisor problem~(\ref{titch}), we can obtain much smaller values of $H$ (but with a much weaker error term) by a different method related to \cite{mr} and \cite{mrt}. We will address this question in the sequel \cite{mrt-corr2} to this paper.

\subsection{Prior results}

We now discuss some partial progress on each of the four parts to Conjecture \ref{bigconj}, starting with the prime tuples conjecture \eqref{hl-conj}.
The conjecture \eqref{hl-conj} is trivial for odd $h$, so we now restrict attention to even $h$.  In this case, even the weaker estimate
\begin{equation}\label{hl-weak}
\sum_{X < n \leq 2X} \Lambda(n) \Lambda(n+h) = {\mathfrak S}(h) X + o(X)
\end{equation}
is not known to hold for any single choice of $h$; for instance, the case $h=2$ would imply the twin prime conjecture, which remains open.  One can of course still use sieve theoretic methods (see e.g. \cite[Corollary 3.14]{mv}) to obtain the upper bound
\begin{equation*}
\sum_{X < n \leq 2X} \Lambda(n) \Lambda(n+h) \ll {\mathfrak S}(h) X 
\end{equation*}
uniformly for $|h| \leq X$ (say). 

There are a number of results \cite{vdc}, \cite{lavrik}, \cite{balog}, \cite{wolke}, \cite{mikawa}, \cite{pp}, \cite{kawada} that show that \eqref{hl-conj} holds for ``most'' $h$ with $|h| \leq H$, as long as $H$ grows moderately quickly with $X$.  The best known result in the literature (with respect to the range of $H$) is by Mikawa \cite{mikawa} and Perelli-Pintz \cite{pp}, who showed (in our notation) that if $X^{1/3+\eps} \leq H \leq X^{1-\eps}$ for some\footnote{One can also handle the range $X^{1-\eps} \leq H \leq X$ by the same methods; see \cite{mikawa} or \cite{pp}. However, we restrict $H$ to be slightly smaller than $X$ here in order to avoid some minor technicalities arising from the fact that $n+h$ might have a slightly different magnitude than $n$. This becomes relevant when dealing with the $d_k$ functions, whose average value depends on the magnitude of the argument.} fixed $\eps>0$, then the estimate \eqref{hl-weak} holds for all but $O_{A,\eps}( H \log^{-A} X )$ values of $h$ with $|h| \leq H$, for any fixed $A$; in fact the $o(X)$ error term in \eqref{hl-weak} can also be taken to be of the form $O_{A,\eps}( X \log^{-A} X )$. 

Now we turn to the divisor correlation conjecture \eqref{conrey-gonek}.  These correlations have been studied by many authors \cite{ingham-0}, \cite{ingham}, \cite{estermann}, \cite{linnik}, \cite{hb}, \cite{moto0}, \cite{moto}, \cite{moto2}, \cite{moto3}, \cite{kuz}, \cite{desh}, \cite{di}, \cite{top}, \cite{ft}, \cite{ivic}, \cite{ivic2}, \cite{conrey}, \cite{iw}, \cite{meurman}, \cite{bv}, \cite{drappeau}, \cite{nt}.  When $k=2$, the conjecture is known to be true with a somewhat worse error term. For instance in the case $k = l = 2$ the current record is 
\begin{align*}
\sum_{X < n \leq 2X} d_2(n) d_2(n+h) &= P_{2,2,h}(\log X) X + O(X^{2/3+o(1)}) \\
%\sum_{X < n \leq 2X} d_2(n) d_3(n+h) &= P_{2,3,h}(\log X) X + O(X^{1-\delta+o(1)}) \\
%\sum_{X < n \leq 2X} d_2(n) d_l(n+h) &= P_{2,l,h}(\log X) X + O_l(X \exp( - c_l \sqrt{\log X} ))
\end{align*}
as $X \to \infty$. This result is due to Deshouillers-Iwaniec \cite{di}. 
% and $l \geq 3$, for some $\delta >0$ and $c_l > 0$; these results are due to Deshouillers and Iwaniec \cite{di}, Deshouillers \cite{desh} and Topacogullari \cite{top}, and Fouvry and Tenenbaum \cite{ft} respectively, with such tools used as the Linnik dispersion method and estimates on Kloosterman sums.
In the cases $l \geq 3$, a power savings
$$ \sum_{X < n \leq 2X} d_2(n) d_l(n+h) = P_{2,l,h}(\log X) X + O(X^{1-\delta_{l}+o(1)})$$
for exponents $\delta_l > 0$, is known \cite{bv}, \cite{drappeau}, \cite{top2}.  See \cite{drappeau}, \cite{nt} for further references and surveys of the problem.  Finally, we remark that a function field analogue of \eqref{conrey-gonek} has been established in \cite{abr}, but with an error term that is only bounded by $O_{k,l}(q^{-1/2})$ times the main term (so the result pertains to the ``large $q$ limit'' rather than the ``large $n$ limit'').

When $k,l \geq 3$, no unconditional proof of even the weaker asymptotic
$$
\sum_{X < n \leq 2X} d_k(n) d_l(n+h) = P_{k,l,h}(\log X) X + o(X \log^{k+l-2} X)
$$ 
is known.
However, upper and lower bounds of the correct order of magnitude are available; see for example, \cite{henriot, henriot2, matt, matt2, Nair, nt}. 

In the case $k=l=3$, the analogue of Mikawa's and Perelli-Pintz's results (now with a power savings in error terms) were recently established by
Baier, Browning, Marasingha, and Zhao \cite{bbmz}, who were able to obtain the asymptotic
$$ \sum_{X < n \leq 2X} d_3(n) d_3(n+h) = P_{3,3,h}(\log X) X + O( X^{1-\delta} )$$
for all but $O_\eps(H X^{-\delta})$ choices of $h$ with $|h| \leq H$, provided that $X^{1/3+\eps} \leq H \leq X^{1-\eps}$ for some fixed $\eps>0$, and $\delta>0$ is a small exponent depending only on $\eps$.

Next, we turn to the (higher order) Titchmarsh divisor problem \eqref{titch}.  This problem is often expressed in terms of computing an asymptotic for $\sum_{p \leq X} d_k(p+h)$ rather than $\sum_{X < n \leq 2X} \Lambda(n) d_k(n+h)$, but the two sums can be related to each other via summation by parts up to negligible error terms, so it is fairly easy to translate results about one sum to the other.  The $k=2$ case of \eqref{titch} with qualitative error term was established by Linnik \cite{linnik}.  This result was improved by Fouvry \cite{fouvry} and Bombieri-Friedlander-Iwaniec \cite{bfi}, who in our notation showed that
$$\sum_{X < n \leq 2X} \Lambda(n) d_2(n+h) = Q_{2,h}(\log X) X + O_A( X \log^{-A} X )$$
for any $A>0$.  Recently, Drappeau \cite{drappeau} showed that the error term could be improved to $O(X \exp( - c \sqrt{\log X} ) )$ for some $c>0$ provided that one added a correction term in the case of a Siegel zero; under the assumption of GRH, the error term could be improved further to  $O(X^{1-\delta})$ for some absolute constant $\delta>0$.  Fiorilli \cite{fiorilli} also established some uniformity of the error term in the parameter $h$.  A function field analog of \eqref{titch} was proven (for arbitrary $k$) in \cite{abr}, but with an error term that is $O_k(q^{-1/2})$ times the main term.

When $k \geq 3$ even the weaker estimate
$$\sum_{X < n \leq 2X} \Lambda(n) d_k(n+h) = Q_{k,h}(\log X) X + o( X \log^{k-1} X )$$
remains open; sieve theoretic methods would only give this asymptotic assuming a level of distribution of $\Lambda$ that is greater than $1-1/k$, which would follow from EH but is not known unconditionally for any $k \geq 3$, even after the recent breakthrough of Zhang \cite{zhang} (see also \cite{polymath8a}).

In analogy with the results of Baier, Browning, Marasingha, and Zhao \cite{bbmz}, it is likely that the method of Mikawa \cite{mikawa} or Perelli-Pintz \cite{pp} can be extended to give an asymptotic of the form
$$\sum_{X < n \leq 2X} \Lambda(n) d_3(n+h) = Q_{3,h}(\log X) X + O_{A,\eps}( X \log^{-A} X )$$
for all but $O_{A,\eps}( H \log^{-A} X )$ values of $h$ with $|h| \leq H$, for any fixed $A$, if $X^{1/3+\eps} \leq H \leq X^{1-\eps}$ for some fixed $\eps>0$; however to our knowledge this result has not been explicitly proven in the literature.

Finally, we discuss some known results on the Goldbach conjecture \eqref{goldbach-conj}.  As with the prime tuples conjecture, standard sieve methods (e.g. \cite[Theorem 3.13]{mv}) will give the upper bound
$$ \sum_{n} \Lambda(n) \Lambda(X-n) \ll {\mathfrak S}(X) X$$
uniformly in $X$.  There are a number of results \cite{tch}, \cite{vdc}, \cite{est}, \cite{mv-gold}, \cite{chen-pan}, \cite{li}, \cite{lu} establishing that the left-hand side of \eqref{goldbach-conj} is positive for ``most'' large even integers $X$; for instance, in \cite{lu} it was shown that this was the case for all but $O(X_0^{0.879})$ of even integers $X \leq X_0$, for any large $X_0$.  There are analogous results in shorter intervals \cite{peneva}, \cite{ly}, \cite{yao}, \cite{jia}, \cite{harman}, \cite{matomaki}; for instance in \cite{matomaki} it was shown that for any $1/5 < \theta \leq 1$ the left-hand side of \eqref{goldbach-conj} is positive for all but $O( X_0^{\theta - \delta} )$ even integers $X \in [X_0, X_0 + X_0^\theta]$, for some $\delta>0$ depending on $\theta$, while in \cite[Chapter 10]{harman} it is shown that for $\frac{11}{180} \leq \theta \leq 1$ and $A>0$, the left-hand side of \eqref{goldbach-conj} is positive for all but $O_A( X_0 \log^{-A} X_0)$ even integers $X \in [X_0, X_0 + X_0^\delta]$.  On the other hand, if one wants the left-hand side of \eqref{goldbach-conj} to not just be positive, but be close to the main term ${\mathfrak S}(X) X$ on the right-hand side, the state of the art requires larger intervals.  For instance, in \cite[Proposition 19.5]{ik} it is shown that \eqref{goldbach-conj} holds (with $O_A( X_0 \log^{-A} X_0)$ error term) for all but $O_A(X_0 \log^{-A} X_0)$ even integers $X$ in $[1,X_0]$.  In \cite{pp}, Perelli and Pintz obtained a similar result for the intervals $[X_0,X_0 + X_0^{\frac{1}{3}+\eps}]$ for any $\eps>0$.  In \cite{halupczok}, Halupczok obtains variants of the result of Perelli-Pintz with the additional requirement that one of the prime in $n = p_1 + p_2$ is constrained to a short interval or an arithmetic progression with large moduli. 

\subsection{New results}

Our main result is as follows: for all four correlations (i)-(iv) in Conjecture \ref{bigconj}, we can improve upon the results of Mikawa, Perelli-Pintz, and Baier-Browning-Marasingha-Zhao by improving the exponent $\frac{1}{3}$ to the quantity
\begin{equation}\label{sigmadef}
\record \coloneqq \recordexplicit;
\end{equation}
for future reference we observe that $\record$ lies in the range
\begin{equation}\label{sigma-range}
\frac{1}{5} < \frac{11}{48} <  \frac{7}{30} < \record < \frac{1}{4}.
\end{equation}
(The significance of the other fractions in \eqref{sigma-range} will become more apparent later in the paper.)  More precisely, we have

\begin{theorem}[Averaged correlations]\label{unav-corr}  Let $A>0$, $0 < \eps < 1/2$ and $k,l \geq 2$ be fixed, and suppose that $X^{\record+\eps} \leq H \leq X^{1-\eps}$ for some $X \geq 2$, where $\record$ is defined by \eqref{sigmadef}.  Let $0 \leq h_0 \leq X^{1-\eps}$.  
\begin{itemize}
\item[(i)]  (Averaged Hardy-Littlewood conjecture) One has
$$ \sum_{X < n \leq 2X} \Lambda(n) \Lambda(n+h) = {\mathfrak S}(h) X + O_{A,\eps}( X \log^{-A} X )$$
for all but $O_{A,\eps}( H \log^{-A} X )$ values of $h$ with $|h-h_0| \leq H$.
\item[(ii)]  (Averaged divisor correlation conjecture) One has
$$ \sum_{X < n \leq 2X} d_k(n) d_l(n+h) = P_{k,l,h}(\log X) X + O_{A, \eps,k,l}( X \log^{-A} X )$$
for all but $O_{A,\eps,k,l}( H \log^{-A} X )$ values of $h$ with $|h-h_0| \leq H$.
\item[(iii)]  (Averaged higher order Titchmarsh divisor problem)  One has
$$ \sum_{X < n \leq 2X} \Lambda(n) d_k(n+h) = Q_{k,h}(\log X) X + O_{A,\eps,k}( X \log^{-A} X )$$
for all but $O_{A,\eps,k}( H \log^{-A} X )$ values of $h$ with $|h-h_0| \leq H$.
\item[(iv)]  (Averaged Goldbach conjecture)  One has
$$ \sum_n \Lambda(n) \Lambda(N-n) = {\mathfrak S}(N) N + O_{A,\eps}(X \log^{-A} X )$$
for all but $O_{A,\eps}(H \log^{-A} X)$ integers $N$ in the interval $[X,X+H]$.
\end{itemize}
\end{theorem}

In the case of correlations of the divisor functions, our method can be modified to obtain power-savings in the error terms. However, since we cannot obtain power-savings in the case of correlations of the von Mangoldt function, in order to keep our choice of parameters uniform accross the four cases stated in Theorem \ref{unav-corr} we have decided to state the result for the divisor function with weaker error terms (see Remark \ref{power} below for more details).

As mentioned previously, the cases $H \geq X^{\frac{1}{3}+\eps}$ of the above theorem are essentially in the literature, either being contained in the papers of Mikawa \cite{mikawa} Perelli-Pintz \cite{pp} and Baier et al. \cite{bbmz}, or following from a modification of their methods.  We give a slightly different proof on these cases in this paper. Still another, but related, proof of the $H \geq X^{\frac{1}{3}+\eps}$ cases could be obtained by adapting arguments used previously for studying the Goldbach problem in short intervals, see e.g. \cite[Chapter 10]{harman} for such arguments. In the range $H > X^{\frac{1}{3} + \varepsilon}$ our argument relies only on standard mean-value theorems and on a simple bound for the fourth moment of Dirichlet $L$-functions (which follows from the mean-value theorem and Poisson summation formula). In contrast, the approaches in \cite{mikawa, bbmz, pp} depend in the range $H > X^{1/3 + \varepsilon}$ on some non-trivial input, such as either bounds for the sixth moment of the Riemann zeta-function off the half-line (in \cite{bbmz}), zero-density estimates (in \cite{pp}) or estimates for Kloosterman sums (in \cite{mikawa}). In fact our approach is entirely independent of results on Kloosterman sums, even for smaller $H$ (see Remark \ref{kloosterman} below for more details). 

Before we embark on a discussion of the proof, we note that our results do not appear to have new consequences for moments of the Riemann-zeta function. For instance for the problem of estimating the sixth moment of the Riemann zeta-function one needs an estimate for
\begin{equation} \label{correlationD}
\sum_{h \leq H} \sum_{n \leq X} d_3(n) d_3(n + h)
\end{equation}
in the range $H = X^{1/3}$. To obtain an improvement over the best-known estimate
$$
\int_{T}^{2T} |\zeta(\tfrac 12 + it)|^6 dt \ll T^{5/4+\eps}
$$
one would need to show that in the range $H = X^{1/3}$ the error term in \eqref{correlationD} is $\ll H X^{5/6 - \varepsilon}$. A naive application of our result, gives a bound of 
$\ll_{A} H X (\log X)^{-A}$ for the error term. As we pointed out earlier in the case of the divisor function, it is possible to improve the $(\log X)^{-A}$ to $X^{-\delta}$ for some $\delta > 0$, however since our method is optimized for dealing with smaller $H$, rather than with $H = X^{1/3}$ we doubt that there will be new results in this range. 

%\marginpar{(And to get new results in the 8th moment case we need $(X/H)^{1/2} = X^{1/4}$ that is $H = X^{1/2}$ and a saving that exceeds $H (X/H)^{3/2} = X^{3/4  + 1/2} = X^{5/4}$, so can we get )}

%We note here that unlike in earlier works these estimates do not depend on estimates for Kloosterman sums, although the results are of comparable quality (see the discussion at the end of the introduction for more on this). 
% (but note that these estimates are proven in turn using Kloosterman sum estimates, so ultimately the same inputs are being used in both proofs).
%For part (iii) of this theorem it is likely that one can in fact obtain power savings in the error terms (as in \cite{bbmz}); see Remark \ref{power}.

We now briefly summarize the arguments used to prove Theorem \ref{unav-corr}.  To follow the many changes of variable of summation (or integration) in the argument, it is convenient to refer to the following diagram:
$$
\left.
\begin{array}{ccc}
\text{Additive frequency } \alpha & & \text{Multiplicative frequency } t \\
\Updownarrow & & \Updownarrow \\
\text{Position } n & \Leftrightarrow & \text{Logarithmic position } u
\end{array}
\right.
$$
Initially, the correlations studied in Theorem \ref{unav-corr} are expressed in terms of the position variable $n$ (an integer comparable to $X$), which we have placed in the bottom left of the above diagram.  The first step in analyzing these correlations, which is standard, is to apply the Hardy-Littlewood circle method (i.e., the Fourier transform), which expresses correlations  such as \eqref{co} as an integral
$$ \int_{\T} S_f(\alpha) \overline{S_g(\alpha)} e(\alpha h)\ d\alpha$$
over the unit circle $\T \coloneqq \R/\Z$, where $S_f, S_g$ are the exponential sums
\begin{align*}
S_f(\alpha) &\coloneqq  \sum_{X < n \leq 2X} f(n) e(n \alpha) \\
S_g(\alpha) &\coloneqq  \sum_{X < n \leq 2X} g(n) e(n\alpha).
\end{align*}
The additive frequency $\alpha$, which is the Fourier-analytic dual to the position variable $n$, is depicted on the top left of the above diagram.  In our applications, $f$ will be of the form $\Lambda 1_{(X,2X]}$ or $d_k 1_{(X,2X]}$, and similarly for $g$. We then divide $\T$ into the \emph{major arcs}, in which $|\alpha - \frac{a}{q}| \leq \frac{\log^{B'} X}{X}$ for some $q \leq \log^B X$, and the \emph{minor arcs}, which consist of all other $\alpha$.  Here $B' > B > 0$ are suitable large constants (depending on the parameters $A,k,l$).  

The major arcs contribute the main terms ${\mathfrak S}(h) X$, $P_{k,l,h}(\log X) X$, $Q_{k,h}(\log X) X$, ${\mathfrak S}(N) N$ to Theorem \ref{unav-corr}, and the estimation of their contribution is standard; we do this in Section \ref{major-sec}.   The main novelty in our arguments lies in the treatment of the minor arc contribution, which we wish to show is negligible on the average.  After an application of the Cauchy-Schwarz inequality, the main task becomes that of estimating the integral
\begin{equation}\label{dol}
 \int_{\beta - 1/H}^{\beta + 1/H} |S_f(\alpha)|^2\ d\alpha
\end{equation}
for various ``minor arc'' $\beta$.  To do this, we follow a strategy from a paper of Zhan \cite{zhan} and estimate this type of integral in terms of the Dirichlet series
$$ {\mathcal D}[f](\frac{1}{2}+it) \coloneqq \sum_n \frac{f(n)}{n^{\frac{1}{2}+it}}$$
for various ``multiplicative frequencies'' $t$.  Actually for technical reasons we will have to twist these Dirichlet series by a Dirichlet character $\chi$ of small conductor, but we ignore this complication for this informal discussion.  The variable $t$ is depicted on the top right of the above diagram, and so we will have to return to the position variable $n$ and then go through the logarithmic position variable $u$, which we will introduce shortly.

Applying the Fourier transform (as was done by Gallagher in \cite{gallagher}), we can control the expression \eqref{dol} in terms of an expression of the form
$$ \int_\R |\sum_{x \leq n \leq x+H} f(n) e(\beta n)|^2\ dx.$$
Actually, it is convenient to smooth the summation appearing here, but we ignore this technicality for this informal discussion.  This returns one to the bottom left of the above diagram.  Next, one makes the logarithmic change of variables $u = \log n - \log X$, or equivalently $n = X e^u$.  This transforms the main variable of interest to a bounded real number $u=O(1)$, and the phase $e(\beta n)$ that appears in the above expression now takes the form $e(\beta X e^u)$.  We are now at the bottom right of the diagram.

Finally, one takes the Fourier transform to convert the expression involving $u$ to an expression involving $t$, which (up to a harmless factor of $2\pi$, as well as a phase modulation) is the Fourier dual of $u$.  Because the $u$ derivative of the phase $\beta X e^u$ is comparable in magnitude to $|\beta| X$, one would expect the main contributions in the integration over $t$ to come from the region where $t$ is comparable to $|\beta| X$.  This intuition can be made rigorous using Fourier-analytic tools such as Littlewood-Paley projections and the method of stationary phase.

At this point, after all the harmonic analytic transformations, we come to the arithmetic heart of the problem. A precise statement of the estimates needed can be found in Proposition \ref{mve}; a model problem is to obtain an upper bound on the quantity
$$ \int_{|t| \asymp \lambda X} \left(\int_{t-\lambda H}^{t+\lambda H} \left|{\mathcal D}[f](\frac{1}{2}+it')\right|\ dt'\right)^2\ dt$$
for $\frac{1}{H} \ll \lambda \ll \log^{-B} X$ that improves (by a large power of $\log X$) upon the trivial bound of $O_k(\lambda^2 H^2 X \log^{O_k(1)} X)$ that one can obtain from the Cauchy-Schwarz inequality
$$ \left(\int_{t-\lambda H}^{t+\lambda H} |{\mathcal D}[f](\frac{1}{2}+it')|\ dt'\right)^2 \ll \lambda H 
\int_{t-\lambda H}^{t+\lambda H} |{\mathcal D}[f](\frac{1}{2}+it')|^2\ dt' $$
Fubini's theorem, and the standard $L^2$ mean value theorem for Dirichlet polynomials.
The most difficult case occurs when $\lambda$ is large (e.g. $\lambda = \log^{-B} X$); indeed, the case $\lambda \leq X^{-\frac{1}{6}-\eps}$ of small $\lambda$ is analogous to the prime number theorem in most short intervals of the form $[X, X+X^{\frac{1}{6}+\eps}]$, and (following \cite{harman}) can be treated by such methods as the Huxley large values estimate and mean value theorems for Dirichlet polynomials.  This is done in Appendix \ref{harman-sec}.  (In the case $f = d_3 1_{(X,2X]}$, these bounds are essentially contained (in somewhat disguised form) in \cite[Theorem 1.1]{bbmz}.)

For sake of argument let us focus now on the case $f = \Lambda 1_{(X,2X]}$.  We proceed via the usual technique of decomposing $\Lambda$ using the Heath-Brown identity \cite{hb-ident} and further dyadic decompositions.  Because $\sigma$ lies in the range \eqref{sigma-range}, this leaves us with ``Type II'' sums where $f$ is replaced by a Dirichlet convolution $\alpha \ast \beta$ with $\alpha$ supported on $[X^{\eps^2}, X^{-\eps^2} H]$, as well as ``Type $d_1$'', ``Type $d_2$'', ``Type $d_3$'', and ``Type $d_4$'' sums where (roughly speaking) $f$ is replaced by a Dirichlet convolution that resembles one of the first four divisor functions $d_1,d_2,d_3,d_4$ respectively.  (See Proposition \ref{types} for a precise statement of the estimates needed.) 

The contribution of the Type II sums can be easily handled by an application of the Cauchy-Schwarz inequality and $L^2$ mean value theorems for Dirichlet polynomials.  The Type $d_1$ and Type $d_2$ sums can be treated by $L^4$ moment theorems \cite{ramachandra}, \cite{bhp} for the Riemann zeta function and Dirichlet $L$-functions. These arguments are already enough to recover the results in \cite{mikawa}, \cite{pp}, \cite{bbmz}, which treated the case $H \geq X^{1/3+\eps}$; our methods are slightly different from those in \cite{mikawa}, \cite{pp}, \cite{bbmz} due to our heavier reliance on Dirichlet polynomials. To break the $X^{1/3}$ barrier we need to control Type $d_3$ sums, and to go below $X^{1/4}$ one must also consider Type $d_4$ sums.  The standard unconditional moment estimates on the Riemann zeta function and Dirichlet $L$-functions are inadequate for treating the $d_3$ sums.   Instead, after applying the Cauchy-Schwarz inequality and subdividing the range $\{ t: t \asymp \lambda X\}$ into intervals of length $\sqrt{\lambda X}$, the problem reduces to obtaining two bounds on Dirichlet polynomials in ``typical'' short or medium intervals.   A model for these problems would be to establish the bounds
\begin{equation}\label{t1}
 \int_{t_j - \sqrt{\lambda X}}^{t_j+\sqrt{\lambda X}} \left|{\mathcal D}[1_{(X^{1/3},2X^{1/3}]}]\left(\frac{1}{2}+it\right)\right|^4\ dt \ll_\eps X^{\eps^2} \sqrt{\lambda X}
\end{equation}
and
\begin{equation}\label{t2}
 \int_{t_j - H}^{t_j+H} \left|{\mathcal D}[1_{(X^{1/3},2X^{1/3}]}]\left(\frac{1}{2}+it\right)\right|^2\ dt \ll_\eps X^{\eps^2} H
\end{equation}
for ``typical'' $j=1,\dotsc,r$, where $t_1,\dotsc,t_r$ is a maximal $\sqrt{\lambda X}$-separated subset of $[\lambda X, 2\lambda X]$.  (These are oversimplifications; see Proposition \ref{qwe} and Proposition \ref{pq} for more precise statements of the bounds needed.)

The first estimate \eqref{t1} turns out to follow readily from a fourth moment estimate of Jutila \cite{jutila} for Dirichlet $L$-functions in medium-sized intervals on average.  As for \eqref{t2}, one can use the Fourier transform to bound the left-hand side by something that is roughly of the form
\begin{equation}\label{hx3}
 \frac{H}{X^{1/3}} \sum_{\ell = O( X^{1/3} / H )} \left|\sum_{m \asymp X^{1/3}} e\left( \frac{t_j}{2\pi} \log \frac{m+\ell}{m-\ell} \right)\right|.
\end{equation}
The diagonal term $\ell=0$ is easy to treat, so we focus on the non-zero values of $\ell$.
By Taylor expansion, the phase $\frac{t_j}{2\pi} \log \frac{m+\ell}{m-\ell}$ is approximately equal to the monomial $\frac{t_j}{\pi} \frac{\ell}{m}$.  If one were to actually replace $e( \frac{t_j}{2\pi} \log \frac{m+\ell}{m-\ell} )$ by $e( \frac{t_j}{\pi} \frac{\ell}{m})$, then it turns out that one can obtain a very favorable estimate by using the fourth moment bounds of Robert and Sargos \cite{rs} for exponential sums with monomial phases.  Unfortunately, the Taylor expansion does contain an additional lower order term of $\frac{t_j}{3\pi} \frac{\ell^3}{m^3}$ which complicates the analysis, but it turns out that (at the cost of some inefficiency) one can still apply the bounds of Robert and Sargos to obtain a satisfactory estimate for the indicated value \eqref{sigmadef} of $\sigma$.

In the range \eqref{sigma-range} one must also treat the Type $d_4$ sums.  Here we use a cruder version of the Type $d_3$ analysis.  The analogue of Jutila's estimate (which would now require control of sixth moments) is not known unconditionally, so we use the classical $L^2$ mean value theorem in its place.  The estimates of Robert and Sargos are now unfavorable, so we instead estimate the analogue of \eqref{hx3} using the classical van der Corput exponent pair $(1/14,2/7)$, which turns out to work even for $\sigma$ as small as $7/30$ (see \eqref{sigma-range}). Hence $d_4$ sums turn out to be easier than $d_3$ in our range of $H$. However we are not able to estimate $d_4$ sums in the full range $X^{1/5 + \varepsilon} < H < X^{1/4 - \varepsilon}$. Therefore there is no advantage in considering $d_5$ sums which would appear if we wanted to take $H$ below $X^{1/5 - \varepsilon}$ (we note that we can cover a tiny region of the $d_5$ sums by proceeding in the same manner as we do with $d_4$ sums). 
 
%It is possible that refinements of this step of the argument could lead to improvements in the final exponent $\record$ in Theorem \ref{unav-corr}; conversely, by using more classical methods (e.g. replacing the recent results in \cite{robert} with the classical van der Corput exponent pairs) one can recover Theorem \ref{unav-corr} with slightly larger values of $\record$.

\begin{remark} \label{kloosterman} It is interesting to note that our work does not depend at all on estimates for Kloosterman sums. While the work of Mikawa for $H > X^{1/3 + \varepsilon}$ depends on the Weil bound for Kloosterman sums, our result in the same range only uses a bound for the fourth moment of Dirichlet $L$-functions The latter follows from the approximate functional equation and a mean-value theorem. In the smaller ranges of $H$ we use in addition estimates for short moments of Dirichlet $L$-functions (due to Jutila, see Proposition \ref{jutila-prop} below and also Corollary \ref{jutila-cor}) that are of the same strength as those that one obtains from using Kloosterman sums (due to Iwaniec, see \cite{iwaniec}) and yet whose proof is independent of input from algebraic geometry or spectral theory.   On the other hand we note that the arguments of Perelli-Pintz \cite{pp} for $H > X^{1/3 + \varepsilon}$ do not depend on Kloosterman sums but instead of zero-density estimates. 
  \end{remark}

\begin{remark}\label{power}  As usual, the results involving $\Lambda$ will have the implied constant depend in an ineffective fashion on the parameter $A$, due to our reliance on Siegel's theorem.  It may be possible to eliminate this ineffectivity (possibly after excluding some ``bad'' scales $X \asymp X_0$) by introducing a separate argument (in the spirit of \cite{hb-twins}) to handle the case of a Siegel zero, but we do not pursue this matter here.  In the proof of Theorem \ref{unav-corr}(ii), we do not need to invoke Siegel's theorem, and it is likely that (as in \cite{bbmz}) we can improve the logarithmic savings $\log^{-A} X$ to a power savings $X^{-\frac{c\eps}{k+l}}$ for some absolute constant $c>0$ (and with effective constants) by a refinement of the argument.  However, we do not do this here in order to be able to treat all four estimates in a unified fashion.
\end{remark}

\subsection{Acknowledgments}

KM was supported by Academy of Finland grant no. 285894. MR was supported by a NSERC Discovery Grant, the CRC program and a Sloan fellowship. TT was supported by a Simons Investigator grant, the James and Carol Collins Chair, the Mathematical Analysis \&
Application Research Fund Endowment, and by NSF grant DMS-1266164.  We are indebted to Yuta Suzuki for a reference and for pointing out a gap in the proof of Proposition \ref{spe} in an earlier version of the paper. We also thank Sary Drappeau and Karin Halupczok for comments on the introduction and the referee for a careful reading of the paper. 

Part of this paper was written while the authors were in residence at MSRI in Spring 2017, which is supported by NSF grant DMS-1440140.

\section{Notation and preliminaries}\label{notation-sec}

All sums and products will be over integers unless otherwise specified, with the exception of sums and products over the variable $p$ (or $p_1$, $p_2$, $p'$, etc.) which will be over primes.  To accommodate this convention, we adopt the further convention that all functions on the natural numbers are automatically extended by zero to the rest of the integers, e.g. $\Lambda(n) = 0$ for $n \leq 0$.

We use $A = O(B)$, $A \ll B$, or $B \gg A$ to denote the bound $|A| \leq C B$ for some constant $C$.  If we permit $C$ to depend on additional parameters then we will indicate this by subscripts, thus for instance $A = O_{k,\eps}(B)$ or $A \ll_{k,\eps} B$ denotes the bound $|A| \leq C_{k,\eps} B$ for some $C_{k,\eps}$ depending on $k,\eps$.  If $A,B$ both depend on some large parameter $X$, we say that $A = o(B)$ as $X \to \infty$ if one has $|A| \leq c(X) B$ for some function $c(X)$ of $X$ (as well as further ``fixed'' parameters not depending on $X$), which goes to zero as $X \to \infty$ (holding all ``fixed'' parameters constant). We also write $A \asymp B$ for $A \ll B \ll A$, with the same subscripting conventions as before.

We use $ \T \coloneqq  \R/\Z$ to denote the unit circle, and $e: \T \to \C$ to denote the fundamental character
$$ e(x) \coloneqq  e^{2\pi i x}.$$

We use $1_E$ to denote the indicator of a set $E$, thus $1_E(n) = 1$ when $n \in E$ and $1_E(n) = 0$ otherwise. Similarly, if $S$ is a statement, we let $1_S$ denote the number $1$ when $S$ is true and $0$ when $S$ is false, thus for instance $1_E(n) = 1_{n \in E}$.  If $E$ is a finite set, we use $\# E$ to denote its cardinality.

We use $(a,b)$ and $[a,b]$ for the greatest common divisor and least common multiple of natural numbers $a,b$ respectively, and write $a|b$ if $a$ divides $b$.  We also write $a = b\ (q)$ if $a$ and $b$ have the same residue modulo $q$.

Given a sequence $f: X \to \C$ on a set $X$, we define the $\ell^p$ norm $\|f\|_{\ell^p}$ of $f$ for any $1 \leq p < \infty$ as
$$ \|f\|_{\ell^p} \coloneqq  \left(\sum_{n \in X} |f(n)|^p\right)^{1/p}$$
and similarly define the $\ell^\infty$ norm 
$$ \|f\|_{\ell^\infty} \coloneqq  \sup_{n \in X} |f(n)|.$$

Given two arithmetic functions $f, g: \N \to \C$, the Dirichlet convolution $f \ast g$ is defined by
$$ f \ast g(n) \coloneqq  \sum_{d|n} f(d) g\left(\frac{n}{d}\right).$$

\subsection{Summation by parts and exponential sums}

If one has an asymptotic of the form $\sum_{X \leq n \leq X''} g(n) \approx \int_X^{X''} h(x)\ dx$ for all $X \leq X'' \leq X'$, then one can use summation by parts to then obtain approximations of the form $\sum_{X \leq n \leq X'} f(n) g(n) \approx \int_X^{X'} f(x) h(x)\ dx$ for sufficiently ``slowly varying'' amplitude functions $f: [X,X'] \to \C$.  The following lemma formalizes this intuition:

\begin{lemma}[Summation by parts]\label{sbp}  Let $X \leq X'$, and let $f: [X,X'] \to \C$ be a smooth function.  Then for any function $g: \N \to \C$ and absolutely integrable $h: [X,X'] \to \C$, we have
$$ 
\sum_{X \leq n \leq X'} f(n) g(n) - \int_X^{X'} f(x) h(x)\ dx \leq |f(X')| E(X') + \int_{X}^{X'} |f'(X'')| E(X'')\ dX''$$
where $f'$ is the derivative of $f$ and $E(X'')$ is the quantity
$$ E(X'') \coloneqq  \left|\sum_{X \leq n \leq X''} g(n) - \int_X^{X''} h(x)\ dx\right|.$$
\end{lemma}

\begin{proof}  From the fundamental theorem of calculus we have
\begin{equation}\label{sb-ident}
 \sum_{X \leq n \leq X'} f(n) g(n) = f(X') \sum_{X \leq n \leq X'} g(n) - \int_X^{X'} \left(\sum_{X \leq n \leq X''} g(n)\right) f'(X'')\ dX''
\end{equation}
and similarly
$$ \int_X^{X'} f(x) h(x)\ dx = f(X') \int_X^{X'} h(x)\ dx - \int_X^{X'} \left(\int_X^{X''} h(x)\ dx\right) f'(X'')\ dX''.$$
Subtracting the two identities and applying the triangle inequality and Minkowski's integral inequality, we obtain the claim.
\end{proof}

The following variant of Lemma \ref{sbp} will also be useful.  Following Robert and Sargos \cite{rs}, define the maximal sum $|\sum_{X \leq n \leq X'} g(n)|^*$ to be the expression
\begin{equation}\label{maxexp}
 \left|\sum_{X \leq n \leq X'} g(n)\right|^* \coloneqq \sup_{X \leq X_1 \leq X_2 \leq X'} \left|\sum_{X_1 \leq n \leq X_2} g(n)\right|.
\end{equation}

\begin{lemma}[Summation by parts, II]\label{sbp-2}  Let $X \leq X'$, let $f: [X,X'] \to \C$ be smooth, and let $g: \N \to \C$ be a sequence.  Then
$$ \left|\sum_{X \leq n \leq X'} f(n) g(n)\right|^* \leq \left|\sum_{X \leq n \leq X'} g(n)\right|^* \left( \sup_{X \leq x \leq X'} |f(x)| + (X'-X) \sup_{X \leq x \leq X'} |f'(x)|\right).$$
\end{lemma}

\begin{proof}  Our task is to show that
$$ \left|\sum_{X_1 \leq n \leq X_2} f(n) g(n)\right| \leq \left|\sum_{X \leq n \leq X'} g(n)\right|^* \left( \sup_{X \leq x \leq X'} |f(x)| + (X'-X) \sup_{X \leq x \leq X'} |f'(x)|\right).$$
for all $X \leq X_1 \leq X_2 \leq X'$.  The claim then follows from \eqref{sb-ident} (replacing $X,X'$ by $X_1,X_2$) and the triangle inequality and Minkowski's integral inequality.
\end{proof}

To estimate maximal exponential sums, we will use the following estimates, contained in the work of Robert and Sargos \cite{rs}:

\begin{lemma}\label{rs1} Let $M \geq 2$ be a natural number, and let $X \geq 2$ be a real number.
\begin{itemize}
\item[(i)]  Let $\phi(1),\dotsc,\phi(M)$ be real numbers, let $a_1,\dotsc,a_M$ be complex numbers of modulus at most one, and let $2 \leq Y \leq X$.  Then
$$ \int_0^X \left(\left|\sum_{m=1}^M a_m e( t \phi(m) )\right|^*\right)^4\ dt \ll \frac{X \log^4 X}{Y}  \int_0^Y \left(\left|\sum_{m=1}^M e( t \phi(m) )\right|\right)^4\ dt.$$
\item[(ii)] Let $\theta \neq 0,1$ be a real number, let $\eps>0$, and let $a_M,\dotsc,a_{2M}$ be complex numbers of modulus at most one.  Then
$$ \int_0^X \left(\left|\sum_{m=M}^{2M} a_m e\left( t \left(\frac{m}{M}\right)^\theta \right)\right|^*\right)^4\ dt \ll_{\theta,\eps} (X+M)^\eps (M^4 + M^2 X).$$
\item[(iii)] Suppose that $M \ll X \ll M^2$.  Let $\phi: \R \to \R$ be a smooth function obeying the derivative estimates $|\phi^{(j)}(x)| \asymp X / M^j$ for $j=1,2,3,4$ and $x \asymp M$.  Then
$$ \left| \sum_{m=M}^{2M} e( \phi(m) )\right|^* \ll \frac{M}{X^{1/2}} \left| \sum_{\epsilon \ell \asymp L} e(\phi^*(\ell))\right|^* + M^{1/2}$$
for some $L \asymp \frac{X}{M}$, where $\phi^*(t) \coloneqq \phi(u(t)) - t u(t)$ is the (negative) Legendre transform of $\phi$, $u$ is the inverse of the function $\phi'$, and $\epsilon = \pm 1$ denotes the sign of $\phi'(x)$ in the range $x \asymp M$.
\end{itemize}
\end{lemma}

%{\bf one might also put the main theorem of Robert \cite{robert} in this lemma.  First we need to locate a copy of the paper!}

\begin{proof}  Part (i) follows from the $p=2$ case of \cite[Lemma 3]{rs}.  Part (ii) follows from \cite[Lemma 7]{rs} when $X \leq M^2$, and the remaining case $X > M^2$ then follows from part (i). Finally, part (iii) follows from applying the van der Corput $B$-process (and Lemma \ref{sbp-2}), see e.g. \cite[Lemma 3.6]{graham} or \cite[Theorem 8.16]{ik}, replacing $\phi$ with $-\phi$ if necessary to normalize the second derivative $\phi''$ to be positive.
\end{proof}

\subsection{Divisor-bounded arithmetic functions}\label{div-bounded-functions-sec}

Let us call an arithmetic function $\alpha: \N \to \C$ \emph{$k$-divisor-bounded} for some $k \geq 0$ if one has the pointwise bound
$$ \alpha(n) \ll_k d_2^k(n) \log^k(2+n)$$
for all $n$.  From the elementary mean value estimate
\begin{equation}\label{divisor-crude}
 \sum_{1 \leq n \leq x} d_l(n)^k \ll_{k,l} x \log^{l^k-1}(2+x),
\end{equation}
valid for any $k \geq 0$, $l \geq 2$, and $x \geq 1$ (see e.g., \cite[formula (1.80)]{ik}), we see that a $k$-divisor-bounded function obeys the $\ell^2$ bounds
\begin{equation}\label{alpha-2}
\sum_{n \leq x} \alpha(n)^2 \ll_k x \log^{O_k(1)}(2+x)
\end{equation}
for any $x \geq 1$.  Applying \eqref{alpha-2} with $\alpha$ replaced by a large power of $\alpha$, we conclude in particular the $\ell^\infty$ bound
\begin{equation}\label{alpha-infty}
\sup_{n \leq x} \alpha(n) \ll_{k,\eps} x^\eps
\end{equation}
for any $\eps > 0$.

\subsection{Dirichlet polynomials}\label{dp-sec}

Given any function $f: \N \to \C$ supported on a finite set, we may form the Dirichlet polynomial
\begin{equation}\label{dir}
 {\mathcal D}[f](s) \coloneqq  \sum_n \frac{f(n)}{n^s}
\end{equation}
for any complex $s$; if $f$ has infinite support but is bounded, we can still define ${\mathcal D}[f]$ in the region $\mathrm{Re} s > 1$.  We will use a normalization in which we mostly evaluate Dirichlet polynomials on the critical line $\{\frac{1}{2}+it: t \in \R\}$, but one could easily run the argument using other normalizations, for instance by evaluating all Dirichlet polynomials on the line $\{1+it: t \in \R\}$ instead.

We have the following standard estimate:

\begin{lemma}[Truncated Perron formula]\label{tpf-lem}  Let $f: \N \to \C$, let $T, X \geq 2$, and let $1 \leq x \leq X$.
\begin{itemize}
\item[(i)]  If $f$ is $k$-divisor-bounded for some $k\geq 0$, and $T \leq X^{1-\eps}$, then for any $0 \leq \sigma < 1 - 2\eps$, one has
$$ \sum_{n \leq x} \frac{f(n)}{n^\sigma} - \frac{1}{2\pi} \int_{-T}^T {\mathcal D}[f](1 + \frac{1}{\log X} +it) \frac{x^{1-\sigma+\frac{1}{\log X}+it}}{1-\sigma+\frac{1}{\log X}+it}\ dt + O_{k,\sigma,\eps}\left( \frac{X^{1-\sigma} \log^{O_k(1)}(TX)}{T} \right).$$
\item[(ii)]  If $f: \N \to \C$ is supported on $[X/C, CX]$ for some $C>1$, then
\begin{equation}\label{tpf}
 \sum_{n \leq x} f(n) = \frac{1}{2\pi} \int_{-T}^T {\mathcal D}[f](\frac{1}{2}+it) \frac{x^{\frac{1}{2}+it}}{\frac{1}{2}+it}\ dt + O_C\left( \sum_n |f(n)| \min\left( 1, \frac{X}{T|x-n|}\right) \right).
\end{equation}
In particular, if we estimate $f(n)$ pointwise by $\|f\|_{\ell^\infty}$, we have
\begin{equation}\label{tpf-2}
 \sum_{n \leq x} f(n) = \frac{1}{2\pi} \int_{-T}^T {\mathcal D}[f](\frac{1}{2}+it) \frac{x^{\frac{1}{2}+it}}{\frac{1}{2}+it}\ dt + O_C\left( \|f\|_{\ell^\infty} \frac{X \log(2 + T)}{T}  \right).
\end{equation}
\end{itemize}
\end{lemma}

\begin{proof}  For (i), apply \cite[Corollary 5.3]{mv} with $a_n \coloneqq \frac{f(n)}{n^\sigma}$ and $\sigma_0 \coloneqq 1 - \sigma + \frac{1}{\log X}$, as well as \eqref{divisor-crude}.  For (ii), apply \cite[Corollary 5.3]{mv} instead with $a_n \coloneqq f(n)$ and $\sigma_0 \coloneqq \frac{1}{2}$.
\end{proof}

As one technical consequence of this lemma, we can estimate the effect of truncating an arithmetic function $f$ on its Dirichlet series:

\begin{corollary}[Truncating a Dirichlet series]\label{trunc-dir}  Suppose that $f: \N \to \C$ is supported on $[X/C, CX]$ for some $X \geq 1$ and $C>1$.  Let $T \geq 1$.  Then for any interval $[X_1,X_2]$ and any $t \in \R$, we have the pointwise bound
$$ {\mathcal D}[f 1_{[X_1,X_2]}](\frac{1}{2}+it) \ll_C \int_{-T}^T |{\mathcal D}[f](\frac{1}{2}+it+iu)| \frac{du}{1+|u|} + \|f\|_{\ell^\infty} \frac{X^{1/2} \log(2+T)}{T}.$$
\end{corollary}

Because the weight $\frac{1}{1+|u|}$ integrates to $O(\log(2+T))$ on $[-T,T]$, this corollary is morally asserting that the Dirichlet polynomial of $f 1_{[X_1,X_2]}$ is controlled by that of $f$ up to logarithmic factors.  As such factors will be harmless in our applications, this corollary effectively allows one to dispose of truncations such as $1_{[X_1,X_2]}$ appearing in a Dirichlet polynomial whenever desired.

\begin{proof}  Applying Lemma \ref{tpf-lem}(ii) with $f$ replaced by $n \mapsto f(n) / n^{it}$, we have for any $x$ that
$$  \sum_{n \leq x} \frac{f(n)}{n^{it}} = \frac{1}{2\pi} \int_{-T}^T {\mathcal D}[f](\frac{1}{2}+it+iu) \frac{x^{\frac{1}{2}+iu}}{\frac{1}{2}+iu}\ du + O_C\left( \|f\|_{\ell^\infty} \frac{X\log(2 + T) }{T} \right)$$
and hence by the triangle inequality
$$ \sum_{n \leq x} \frac{f(n)}{n^{it}} \ll_C x^{1/2} \int_{-T}^T |{\mathcal D}[f](\frac{1}{2}+it+iu)| \frac{du}{1+|u|}+ \|f\|_{\ell^\infty} \frac{X \log(2 + T)}{T} .$$
The claim now follows from Lemma \ref{sbp} (with $h=0$, $g(n)$ replaced by $f(n)/n^{it}$, and $f(x)$ replaced by $x^{-1/2}$).
\end{proof}

\subsection{Arithmetic functions with good cancellation}\label{sec:good-cancel}
Let $\alpha \colon \mathbb{N} \to \mathbb{C}$ be a $k$-divisor-bounded function. From \eqref{alpha-2} and Cauchy-Schwarz, we see that
\begin{equation}\label{dirt}
 \sum_{n \leq x: n = a\ (q)} \frac{\alpha(n)}{n^{\frac{1}{2}+it}} \ll_{k}  x^{1/2} \log^{O_k(1)} x
\end{equation}
for any $t \in \R$, $q \geq 1$, and $a \in \Z$.  We will say that a $k$-divisor-bounded function $\alpha$ has \emph{good cancellation} if one has the improved bound
\begin{equation}\label{alb}
 \sum_{n \leq x: n = a\ (q)} \frac{\alpha(n)}{n^{\frac{1}{2}+it}} \ll_{k,A,B,B'}  x^{1/2} \log^{-A} x
\end{equation}
for any $A,B,B'>0$, $x \geq 2$, $q \leq \log^B x$, $a \in \Z$, and $t \in \R$ with $\log^{B'} x \leq |t| \leq x^{B'}$, provided that $B'$ is sufficiently large depending on $A,B,k$.  

It is clear that if $\alpha$ is a $k$-divisor-bounded function with good cancellation, then so is its restriction $\alpha 1_{[X_1,X_2]}$ to any interval $[X_1,X_2]$.  The property of being $k$-divisor-bounded with good cancellation is also basically preserved under Dirichlet convolution:

\begin{lemma}\label{good-cancel}  Let $\alpha, \beta$ be $k$-divisor-bounded functions.  Then $\alpha \ast \beta$ is a $(2k+1)$-divisor-bounded function.  Furthemore, if $\alpha$ and $\beta$ both have good cancellation, then so does $\alpha \ast \beta$.

If, in addition, there is an $N$ for which $\alpha$ is supported on $[N^2,+\infty]$ and $\beta$ is supported on $[1,N]$, then one can omit the hypothesis that $\beta$ has good cancellation in the previous claim.
\end{lemma}

\begin{proof}  Using the elementary inequality $d_2^{k_1} \ast d_2^{k_2} \leq d_2^{k_1+k_2+1}$, we see $\alpha \ast \beta$ is $2k+1$-divisor-bounded.  Next, suppose that $\alpha$ and $\beta$ have good cancellation, and let $A,B,B' > 0$, $x \geq 2$, $q \leq \log^B X$, $a \in \Z$, and $t \in \R$ with $\log^{B'} x \leq |t| \leq x^{B'}$, with $B'$ is sufficiently large depending on $A,B,k$.  To show that $\alpha \ast \beta$ has good cancellation, it suffices by dyadic decomposition to show that
\begin{equation}\label{dance}
 \sum_{x < n \leq 2x: n = a\ (q)} \frac{\alpha \ast \beta(n)}{n^{\frac{1}{2}+it}} \ll_{k,A,B,B'}  x^{1/2} \log^{-A} x.
\end{equation}
By decomposing $\alpha$ into $\alpha 1_{[1, \sqrt{x}]}$ and $\alpha 1_{(\sqrt{x},+\infty)}$, and similarly for $\beta$, we may assume from the triangle inequality that at least one of $\alpha,\beta$ is supported on $(\sqrt{x},+\infty)$; by symmetry we may assume that $\alpha$ is so supported.  The left-hand side of \eqref{dance} may thus be written as
$$
\sum_{a = b c\  (q)} \sum_{m \ll \sqrt{x}: m = c\ (q)} \frac{\beta(m)}{m^{\frac{1}{2}+it}} \sum_{x/m < n \leq 2x/m: n = b\ (q)} \frac{\alpha(n)}{n^{\frac{1}{2}+it}}.$$
Let $A'>0$ be a quantity depending on $A,B,k$ to be chosen later.
As $\alpha$ has good cancellation, we may bound this (for $B'$ sufficiently large depending on $k,A',B$) using the triangle inequality by
$$ \ll_{k,A',B,B'} \sum_{a = bc\  (q)} \sum_{m \ll \sqrt{x}: m = c\ (q)} \frac{|\beta(m)|}{m^{\frac{1}{2}}} (x/m)^{1/2} \log^{-A'} (x/m);$$
as $\beta$ is $k$-divisor-bounded and $q \leq \log^B x$, we may apply \eqref{divisor-crude} and bound this by
$$ \ll_{k,A',B,B'} x^{\frac 12} \log^{-A' + B + O_k(1)} x.$$
Choosing $A'$ sufficiently large depending on $A,B,k$, we obtain \eqref{dance}.  The final claim of the lemma is proven similarly, noting that the left-hand side of \eqref{dance} vanishes unless $x \gg N^2$, and hence from the support of $\beta$ we may already restrict $\alpha$ to the region $(c\sqrt{x},+\infty)$ for some absolute constant $c>0$ without invoking symmetry.
\end{proof} 

We have three basic examples of functions with good cancellation:

\begin{lemma}\label{point}  The constant function $1$, the logarithm function $L: n \mapsto \log n$ and the M\"obius function $\mu$ are $1$-divisor-bounded with good cancellation.
\end{lemma}

From this lemma and Lemma \ref{good-cancel}, we also see that $\Lambda$ and $d_k$ have good cancellation for any fixed $k$.

\begin{proof}  For the functions $1,L$ this follows from standard van der Corput exponential sum estimates for $|\sum_{n \leq x} e(-\frac{t}{2\pi} \log (qn+a))|^*$ (e.g. \cite[Lemma 8.10]{ik}) and Lemma \ref{sbp-2}, normalizing $a$ to be in the range $0 \leq a < q$.  Now we consider the function $\mu$.  By using multiplicativity (and increasing $A$ as necessary) we may assume that $a$ is coprime to $q$.  By decomposition into Dirichlet characters (and again increasing $A$ as necessary) it suffices to show that
\begin{equation}\label{much1}
 \sum_{n \leq x} \frac{\mu(n) \chi(n)}{n^{\frac{1}{2}+it}} \ll_{k,A,B,B'}  x^{1/2} \log^{-A} x
\end{equation}
for any Dirichlet character $\chi$ of period $q$.  

This estimate is certainly known to the experts, but as we did not find it in this form in the literature, we prove it here. The Vinogradov-Korobov zero-free region \cite[\S 9.5]{mont} implies that $L(s,\chi)$ has no zeroes in the region
$$ \left\{ \sigma+it': 0 < |t'| \ll |t| + x^2; \sigma \geq 1 - \frac{c_{B,B'}}{\log^{2/3} |x| (\log\log |x|)^{1/3}} \right\}$$
for some $c_{B,B'}>0$ depending only on $B,B'$.  Applying the estimates in \cite[\S 16]{davenport}, and shrinking $c_{B,B'}$ if necessary, we obtain the crude upper bounds
$$ \frac{L'(s,\chi)}{L(s,\chi)} \ll_{B,B'} \log^2 |x|$$
in this region.
Applying Perron's formula as in \cite[Lemma 2]{mr-p} (see also \cite[Lemma 1.5]{harman}), one then has the bound
$$
 \sum_{n \leq y} \Lambda(n) \chi(n) n^{-it} \ll_{A',B,B'}  y \log^{-A'} x
$$
for any $A'>0$ and any $y$ with $\exp(\log^{3/4} x) \leq y \leq x^2$ (in fact one may replace $3/4$ here by any constant larger than $2/3$). 

To pass from $\Lambda$ to $\mu$ we use a variant of the arguments used to prove Lemma \ref{good-cancel} (one could also work more directly, using upper bound for $1/L(s, \chi)$ but we could not find the exact upper bound we need from the literature).  We begin with the trivial bound
\begin{equation}\label{mu-triv}
 \sum_{n \leq y} \mu(n) \chi(n) n^{-it} \ll y
\end{equation}
for any $y > 0$.   Writing $\mu(n) \log(n) \chi(n) n^{-it}$ as the Dirichlet convolution of $\Lambda(n) \chi(n) n^{-it}$ and $-\mu(n) \chi(n) n^{it}$ and using the Dirichlet hyperbola method, we conclude that
$$
 \sum_{n \leq y} \mu(n) \log n \chi(n) n^{-it} \ll_{B,B'} y \log^{3/4} x 
$$
for any $y = x^{1+o(1)}$, which by Lemma \ref{sbp-2} implies that
$$
 \sum_{n \leq y} \mu(n) \chi(n) n^{-it} \ll_{B,B'} y \log^{-1/4} x 
$$
for $y = x^{1+o(1)}$.  Applying the Dirichlet hyperbola method again (using the above bound to replace the trivial bound \eqref{mu-triv} for $y = x^{1+o(1)}$) we conclude that
$$
 \sum_{n \leq y} \mu(n) \chi(n) n^{-it} \ll_{B,B'} y \log^{-2/4} x 
$$
for $y = x^{1+o(1)}$.  Iterating this argument $O(A)$ times, we eventually conclude that
$$
 \sum_{n \leq y} \mu(n) \chi(n) n^{-it} \ll_{A,B,B'} y \log^{-A} x 
$$
for $y=x^{1+o(1)}$, and the claim \eqref{much1} then follows from Lemma \ref{sbp-2}.
\end{proof}

\subsection{Mean value theorems}

In view of Lemma \ref{tpf-lem}, it becomes natural to seek upper bounds on the quantity $|{\mathcal D}[f](\frac{1}{2}+it)|$ for various functions $f$ supported on $[X/C,CX]$.  We will primarily be interested in functions $f$ which are $k$-divisor-bounded for some bounded $k$.  In such a case, we see from \eqref{alpha-2} that
$$ \|f\|_{\ell^2}^2 \ll_k X \log^{O_k(1)} X.$$
The heuristic of \emph{square root cancellation} then suggests that the quantity $|{\mathcal D}[f](\frac{1}{2}+it)|$ should be of size $O( X^{o(1)})$ for all values of $t$ that are of interest (except possibly for the case $t=O(1)$ in which there might not be sufficient oscillation).  Such square root cancellation is not obtainable unconditionally with current techniques; for instance, square root cancellation for $f = 1_{(X,2X]}$ is equivalent to the Lindel\"of hypothesis, while square root cancellation for $f = \Lambda 1_{(X,2X]}$ is equivalent to the Riemann hypothesis.  However, we will be able to use a number of results that obtain something resembling square root cancellation on the average.  The most basic instance of these results is the classical $L^2$ mean value theorem:

\begin{lemma}[Mean value theorem]\label{mvt-lem}  Suppose that $f \colon \N \to \C$ is supported on $[X/2,4X]$ for some $X \geq 2$.  Then one has
$$ \int_{T_0}^{T_0+T} \left|{\mathcal D}[f](\frac{1}{2}+it)\right|^2\ dt \ll \frac{T + X}{X} \|f\|_{\ell^2}^2$$
for all $T > 0$ and $T_0 \in \R$.  In particular (from \eqref{alpha-2}), if $f$ is $k$-divisor-bounded, then
$$ \int_{T_0}^{T_0+T} \left|{\mathcal D}[f](\frac{1}{2}+it)\right|^2\ dt \ll_k (T + X) \log^{O_k(1)} X.$$
\end{lemma}

\begin{proof}  See \cite[Theorem 9.1]{ik}.
\end{proof}

We will need to twist Dirichlet series by Dirichlet characters.  With $f$ as above, and any Dirichlet character $\chi \colon \Z \to \C$, we can define
$$ {\mathcal D}[f](s,\chi) \coloneqq  \sum_n \frac{f(n)}{n^s} \chi(n)$$
and more generally
$$ {\mathcal D}[f](s,\chi,q_0) \coloneqq  \sum_n \frac{f(q_0 n)}{n^s} \chi(n)$$
for any complex $s$ and any natural number $q_0$.  These Dirichlet series naturally appear when estimating Dirichlet series with a Fourier weight $e\left(\frac{an}{q}\right)$, as the following simple lemma shows:

\begin{lemma}[Expansion into Dirichlet characters]\label{edc}  Let $f \colon \N \to \C$ be a function supported on a finite set, let $q$ be a natural number, and let $a$ be coprime to $q$.  Let $e\left(\frac{a \cdot}{q}\right)$ denote the function $n \mapsto e\left(\frac{an}{q}\right)$.  Then we have the pointwise bound
\begin{equation}\label{c-split}
 \left|{\mathcal D}\left[f e\left(\frac{a \cdot}{q}\right)\right](s)\right| \leq \frac{d_2(q)}{\sqrt{q}} \sum_{q = q_0 q_1} \sum_{\chi\ (q_1)} |{\mathcal D}[f](s,\chi,q_0)|
\end{equation}
for all complex numbers $s$ with $\mathrm{Re}(s) = \frac{1}{2}$,
where in this paper the sum $\sum_{\chi\ (q_1)}$ denotes a summation over all characters $\chi$ (including the principal character) of period $q_1$, and $q_0,q_1$ are understood to be natural numbers.  
\end{lemma}

\begin{proof}  Let $s$ be such that $\mathrm{Re}(s) = \frac{1}{2}$.  By definition we have
$$ {\mathcal D}\left[f e\left(\frac{a \cdot}{q}\right)\right](s) = \sum_n \frac{f(n)}{n^s} e\left(\frac{an}{q}\right).$$
We now decompose the $n$ summation in terms of the greatest common divisor $q_0 \coloneqq (n,q)$ of $n$ and $q$, obtaining (after writing $q = q_0 q_1$ and $n = q_0 n_1$)
$$ {\mathcal D}\left[f e\left(\frac{a \cdot}{q}\right)\right](s) = \sum_{q = q_0 q_1} \frac{1}{q_0^s} \sum_{n_1: (n_1,q_1)=1} \frac{f(q_0 n_1)}{n_1^s} e\left(\frac{an_1}{q_1}\right)$$
and thus by the triangle inequality
$$ \left|{\mathcal D}\left[f e\left(\frac{a \cdot}{q}\right)\right](s)\right| \leq \sum_{q = q_0 q_1} \frac{1}{\sqrt{q_0}} \left|\sum_{n_1: (n_1,q_1)=1} \frac{f(q_0 n_1)}{n_1^s} e\left(\frac{an_1}{q_1}\right)\right|.$$
Next, we perform the usual Dirichlet expansion
\begin{equation}\label{usual}
 e\left(\frac{an_1}{q_1}\right) 1_{(n_1,q_1)=1} = \frac{1}{\phi(q_1)} \sum_{\chi\ (q_1)} \chi(a) \chi(n_1) \tau(\overline{\chi})
\end{equation}
where $\tau(\overline{\chi})$ is the Gauss sum
\begin{equation}\label{taug}
 \tau(\overline{\chi}) \coloneqq \sum_{l=1}^{q_1} e\left(\frac{l}{q_1}\right) \overline{\chi(l)}
\end{equation}
As is well known, we have
$$ |\tau(\overline{\chi})| \leq \sqrt{q_1}$$
(as can be seen for instance from making the substitution $l \mapsto al$ to \eqref{taug} for $(a,q)=1$ and then applying the Parseval identity in $a$).  Using the crude bound
$$  \frac{1}{\sqrt{q_0}} \frac{1}{\phi(q_1)} \sqrt{q_1} \ll  \frac{1}{\sqrt{q_0}} \frac{d_2(q_1)}{q_1} \sqrt{q_1} \leq \frac{d_2(q)}{\sqrt{q}}$$
and the triangle inequality, we obtain \eqref{c-split}.  
\end{proof}

It thus becomes of interest to have upper bounds, on average at least, on the quantity $|{\mathcal D}[f](\frac{1}{2}+it,\chi,q_0)|$.  
We first recall a variant of Lemma \ref{mvt-lem}, which can save a factor of $q_1$ or so compared to that lemma when summing over characters $\chi$:

\begin{lemma}[Mean value theorem with characters]\label{mvt-lem-ch} Suppose that $f \colon \N \to \C$ is supported on $[X/2,4X]$ for some $X \geq 2$.  Then one has
$$ \sum_{\chi\ (q_1)} \int_{T_0}^{T_0+T} \left|{\mathcal D}[f]\left(\frac{1}{2}+it, \chi \right)\right|^2 \ll \frac{q_1 T + X}{X} \|f\|_{\ell^2}^2 \log^3(q_1 T X)$$
for all $T \geq 2$, $T_0 \in \R$, and natural numbers $q_1$, where $\chi$ is summed over all Dirichlet characters of period $q_1$.  In particular, if $f$ is $k$-divisor-bounded, then (from \eqref{alpha-2}) we have
$$ \sum_{\chi\ (q_1)} \int_{T_0}^{T_0+T} \left|{\mathcal D}[f]\left(\frac{1}{2}+it, \chi \right)\right|^2 \ll_k (q_1 T + X) \log^{O_k(1)}(q_1 T X)$$
\end{lemma}

\begin{proof}  This is a special case of \cite[Theorem 9.12]{ik}.  
\end{proof}

In the case when $f$ is an indicator function $f = 1_{[1,X]}$ we have a fourth moment estimate:

\begin{lemma}[Fourth moment estimate]\label{fourth}  Let $X \geq 2$, $q_1 \geq 1$, and $T \geq 1$.  Let ${\mathcal S}$ be a finite set of pairs $(\chi,t)$ with $\chi$ a character of period $q$, and $t \in [-T,T]$.  Suppose that ${\mathcal S}$ is $1$-separated in the sense that for any two distinct pairs $(\chi,t), (\chi',t') \in {\mathcal S}$, one either has $\chi \neq \chi'$ or $|t-t'| \geq 1$.  Then one has
\begin{align*}
\sum_{(\chi,t) \in {\mathcal S}} \left|{\mathcal D}[1_{[1,X]}]\left(\frac{1}{2}+it,\chi\right)\right|^4 &\ll q_1 T \log^{O(1)} X + |{\mathcal S}| \log^4 X \left( \frac{q_1^2}{T^2} + \frac{X^2}{T^4} \right)\\
&\quad  + X^2 \sum_{(\chi,t) \in {\mathcal S}} \delta_\chi (1+|t|)^{-4}
\end{align*}
where $\delta_\chi$ is equal to $1$ when $\chi$ is principal, and equal to zero otherwise.
\end{lemma}

\begin{proof}  See \cite[Lemma 9]{bhp}.  We remark that this estimate is proven using fourth moment estimates \cite{ramachandra} for Dirichlet $L$-functions.
\end{proof}

It will be more convenient to use a (slightly weaker) integral form of this estimate:

\begin{corollary}[Fourth moment estimate, integral form]\label{fourth-integral}  Let $X \geq 2$, $q_1 \geq 1$, and $T \geq 1$.  Then
\begin{equation}\label{iant}
\sum_{\chi\ (q_1)} \int_{T/2 \leq |t| \leq T}  \left|{\mathcal D}[1_{[1,X]}](\frac{1}{2}+it,\chi)\right|^4\ dt \ll q_1 T \left(1 + \frac{q_1^2}{T^2} + \frac{X^2}{T^4} \right) \log^{O(1)} X.
\end{equation}
A similar bound holds with $1_{[1,X]}$ replaced by $L 1_{[1,X]}$.
\end{corollary}

\begin{proof}  For each $\chi$, we cover the region $T/2 \leq |t| \leq T$ by unit intervals $I$, and for each such $I$ we find a point $t \in I$ that maximizes $|{\mathcal D}[1_{[1,X]}](\frac{1}{2}+it,\chi)|$, then add $(\chi,t)$ to ${\mathcal S}$.  Then $|{\mathcal S}| \ll q T$; it is not necessarily $1$-separated, but one can easily separate it into $O(1)$ $1$-separated sets.  Applying Lemma \ref{fourth} (bounding $\delta_\chi (1+|t|)^{-4}$ by $O(1/T^4)$), we obtain \eqref{iant}.  Finally, to handle $L 1_{[1,X]}$, one can use the integration by parts identity
\begin{equation}\label{lix}
 L 1_{[1,X]} (y) = \log X 1_{[1,X]}(y) - \int_1^X 1_{[1,X']}(y) \frac{dX'}{X'}
\end{equation}
and the triangle inequality (cf. Lemmas \ref{sbp}, \ref{sbp-2}).
\end{proof}

We will also need the following variant of the fourth moment estimate due to Jutila \cite{jutila}.

\begin{proposition}[Jutila]\label{jutila-prop} Let $q, T \geq 1$ and $\eps > 0$.  Let $T^{1/2 + \varepsilon} \ll T_0 \ll T^{2/3}$ and $T < t_1 < \ldots < t_r < 2T$ with $t_{i + 1} - t_{i} > T_0$. Then we have
$$
\sum_{\chi \ (q)} \sum_{i = 1}^{r} \int_{t_i}^{t_i + T_0} |L(\tfrac 12 + it, \chi)|^4 dt \ll_\eps q (r T_0 + (r T)^{2/3}) (q T)^{\varepsilon}.
$$
\end{proposition}

\begin{proof}
See \cite[Theorem 3]{jutila}. This estimate is a variant of Iwaniec's result \cite{iwaniec} on the fourth moment of $\zeta$ in short intervals; it is however proven using a completely different and more elementary method. 
\end{proof}

Using a variant of Corollary \ref{trunc-dir} we may truncate the Dirichlet $L$-function to conclude

\begin{corollary}\label{jutila-cor}  Let the hypotheses be as in Proposition \ref{jutila-prop}.  Then for any $1 \leq X \ll T^2$ and any Dirichlet character $\chi$ of period $q$, one has
$$
\sum_{i = 1}^{r} \int_{t_i}^{t_i + T_0} |{\mathcal D}[1_{[1,X]}](\tfrac 12 + it, \chi)|^4 dt \ll_\eps q^{O(1)} (r T_0 + (r T)^{2/3}) T^{\varepsilon}.
$$
Similarly with $1_{[1,X]}$ replaced by $L 1_{[1,X]}$.
\end{corollary}

One can be more efficient here with respect to the dependence of the right-hand side on $q$, but we will not need to do so in our application, as we will only use Corollary \ref{jutila-cor} for quite small values of $q$.

\begin{proof}  In view of \eqref{lix} and the triangle inequality, followed by dyadic decomposition, it suffices to show that
$$
\sum_{i = 1}^{r} \int_{t_i}^{t_i + T_0} |{\mathcal D}[1_{[X,2X]}](\tfrac 12 + it, \chi)|^4 dt \ll_\eps q^{O(1)} (r T_0 + (r T)^{2/3}) T^{\varepsilon}.
$$
From the fundamental theorem of calculus we have
$$ {\mathcal D}[1_{[X,2X]}](\tfrac 12 + it, \chi) = \frac{1}{(2X)^{1/2}} {\mathcal D}[1_{[X,2X]}](it, \chi) + \int_X^{2X}
{\mathcal D}[1_{[X,X']}](it, \chi) \frac{dX'}{2(X')^{3/2}} $$
so by the triangle inequality again, it suffices to show that
\begin{equation}\label{spak}
\sum_{i = 1}^{r} \int_{t_i}^{t_i + T_0} |{\mathcal D}[1_{[1,X]}](it, \chi)|^4 dt \ll_\eps q^{O(1)} X^2 (r T_0 + (r T)^{2/3}) T^{\varepsilon}.
\end{equation}
Let $t$ lie in the range $[T, 3T]$.  From Lemma \ref{tpf-lem}(i) with $f(n), \sigma, T$ replaced by $\frac{\chi(n)}{n^{it}}$, $0$, $T^2$ respectively, we see that
$$
{\mathcal D}[1_{[1,X]}](it, \chi) \ll \left|\int_{-T^2}^{T^2} L\left( 1 + \frac{1}{\log X} + i(t+t'), \chi\right) \frac{X^{1+\frac{1}{\log X} + it'}}{1+\frac{1}{\log X}+it'}\ dt'\right| + \frac{X \log^{O(1)} T}{T^2}$$
where $L(s,\chi)$ is the Dirichlet $L$-function. Note that $X / T^2 \ll 1$ by assumption.
Shifting the contour and using the crude convexity bound $L(\sigma+it,\chi) \ll q^{O(1)} (1+t)^{1/2}$ for $1/2 \leq \sigma \leq 2$, all $\varepsilon > 0$ and $|\sigma+it-1| \gg 1$, and also noting that the residue of $L(s,\chi)$ at $1$ (if it exists) is $O(1)$, we obtain the estimate
$$
{\mathcal D}[1_{[1,X]}](it, \chi) \ll q^{O(1)} \left( \left|\int_{-T^{2}}^{T^{2}} L\left( \frac{1}{2} + i(t+t'), \chi\right) \frac{X^{\frac{1}{2}+it'}}{\frac{1}{2}+it'}\ dt'\right| + \frac{X}{T} + \log^{O(1)} T \right)
$$
(say).  Since $X \ll T^2$, we can bound the error term inside brackets by $X^{1/2} \cdot \log^{O(1)} T$.
We have the crude $L^2$ mean value estimate
$$ \int_{-T'}^{T'} \left|L\left(\frac{1}{2}+i( t+ t'), \chi\right)\right|^2\ dt' \ll q^{O(1)} T' (\log (T' + 2))^{O(1)}$$
for any $T' > T/10$ (which can be established for instance from Lemma \ref{mvt-lem} and the approximate functional equation).
From this, Cauchy-Schwarz, and dyadic decomposition,
we see that
$$
\left|\int_{-T^{2}}^{T^{2}} L\left( \frac{1}{2} + i(t+t'), \chi\right) \frac{X^{\frac{1}{2}+it'}}{\frac{1}{2}+it'}\ dt'\right| \ll X^{1/2} \Big ( \int_{-T/2}^{T/2} \frac{ | L ( \frac{1}{2} + i(t + t'), \chi ) |}{1 + |t'|} d t' + (\log (T + 2))^{O(1)} \Big ) 
$$
We conclude that,
$$
{\mathcal D}[1_{[1,X]}](it, \chi) \ll q^{O(1)} X^{1/2} (\log(T+2))^{O(1)} \left( 1 +  \int_{-T/2}^{T/2} \frac{\left|L( \frac{1}{2} + i(t+t'), \chi)\right|}{1+|t'|}\ dt' \right).
$$
By H\"older's inequality, we then have
$$
|{\mathcal D}[1_{[1,X]}](it, \chi)|^4 \ll q^{O(1)} X^2 (\log(T+2))^{O(1)} \left( 1 + \int_{-T/2}^{T/2} \frac{\left|L\left( \frac{1}{2} + i(t+t'), \chi\right)\right|^4}{1+|t'|}\ dt' \right).
$$
From shifting the $t_j$ by $t'$, we see from Proposition \ref{jutila-prop} that
$$
\sum_{i = 1}^{r} \int_{t_i}^{t_i + T_0} |L(\tfrac 12 + i(t+t'), \chi)|^4 dt \ll_\eps q (r T_0 + (r T)^{2/3}) (q T)^{\varepsilon}.
$$
whenever $-T/2 \leq t' \leq T/2$.  The claim \eqref{spak} now follows from Fubini's theorem.
\end{proof}

\subsection{Combinatorial decompositions}

We will treat the functions $\Lambda 1_{(X,2X]}$ and $d_k 1_{(X,2X]}$ in a unified fashion, decomposing both of these functions as certain (truncated) Dirichlet convolutions of various types, which we will call ``Type $d_j$ sums'' for some small $j=1,2,\dots$ and ``Type II sums'' respectively.  More precisely, we have

\begin{lemma}[Combinatorial decomposition]\label{comb-decomp}  Let $k, m \geq 1$ and $0 < \eps < \frac{1}{m}$ be fixed.  Let $X \geq 2$, and let $H_0$ be such that $X^{\frac{1}{m} + \eps} \leq H_0 \leq X$.  Let $f \colon \N \to \C$ be either the function $f \coloneqq  \Lambda 1_{(X,2X]}$ or $f \coloneqq  d_k 1_{(X,2X]}$.  Then one can decompose $f$ as the sum of $O_{k,m,\eps}( \log^{O_{k,m,\eps}(1)} X )$ components $\tilde f$, each of which is of one of the following types:
\begin{itemize}
\item[(Type $d_j$)]  A function of the form
\begin{equation}\label{fst}
 \tilde f = (\alpha \ast \beta_1 \ast \dotsb \ast \beta_j) 1_{(X,2X]}
\end{equation}
for some arithmetic functions $\alpha,\beta_1,\dotsc,\beta_j \colon \N \to \C$, where $1 \leq j < m$, $\alpha$ is $O_{k,m,\eps}(1)$-divisor-bounded and supported on $[N,2N]$, and each $\beta_i$, $i=1,\dotsc,j$ is either equal to $1_{(M_i,2M_i]}$ or $L 1_{(M_i,2M_i]}$ for some $N, M_1,\dotsc,M_j$ obeying the bounds
\begin{equation}\label{n-small}
1 \ll N \ll_{k,m,\eps} X^\eps,
\end{equation}
\begin{equation}\label{nmx}
N M_1 \dots M_j \asymp_{k,m,\eps} X
\end{equation}
and 
$$ H_0 \ll M_1 \ll \dotsb \ll M_j \ll X.$$
\item[(Type II sum)]  A function of the form
$$ f = (\alpha \ast \beta) 1_{(X,2X]}$$
for some $O_{k,m,\eps}(1)$-divisor-bounded arithmetic functions $\alpha,\beta \colon \N \to \C$ of good cancellation supported on $[N,2N]$ and $[M,2M]$ respectively, for some $N,M$ obeying the bounds
\begin{equation}\label{n-medium}
X^\eps \ll_{k,m,\eps} N \ll_{k,m,\eps} H_0
\end{equation}
and
\begin{equation}\label{nmx-2}
NM \asymp_{k,m,\eps} X.
\end{equation}
\end{itemize}
\end{lemma}

As the name suggests, Type $d_j$ sums behave similarly to the $j^{\operatorname{th}}$ divisor function $d_j$ (but with all factors in the Dirichlet convolution constrained to be supported on moderately large natural numbers).  In our applications we will take $m$ to be at most $5$, so that the only sums that appear are Type $d_1$, Type $d_2$, Type $d_3$, Type $d_4$, and Type II sums, and the dependence in the above asymptotic notation on $m$ can be ignored.  The contributions of Type $d_1$, Type $d_2$, and Type II sums were essentially treated by previous literature; our main innovations lie in our estimation of the contributions of the Type $d_3$ and Type $d_4$ sums.  

\begin{proof}  We first claim a preliminary decomposition: $f$ can be expressed as a linear combination (with coefficients of size $O_{k,m,\eps}(1)$) of $O_{k,m,\eps}(\log^{O_{k,m,\eps}(1)} X)$ terms $\tilde f$ that are each of the form
\begin{equation}\label{pre}
 \tilde f = (\gamma_1 \ast \dotsb \ast \gamma_r) 1_{(X,2X]}
\end{equation}
for some $r = O_{k,m,\eps}(1)$, where each $\gamma_i \colon \N \to \C$ is supported on $[N_i,2N_i]$ for some $N_1,\dotsc,N_r \gg 1$ and are $1$-divisor-bounded with good cancellation.  Furthermore, for each $i$, one either has $\gamma_i = 1_{(N_i,2N_i]}$, $\gamma_i = L 1_{(N_i,2N_i]}$, or $N_i \ll X^\eps$.

We first perform this decomposition in the case $f = d_k 1_{(X,2X]}$.  On the interval $(X,2X]$, we clearly have
$$ d_k = 1_{[1,2X]} \ast \dotsb \ast 1_{[1,2X]}$$
where the term $1_{[1,2X]}$ appears $k$ times.  We can dyadically decompose $1_{[1,2X]}$ as the sum of $O( \log X )$ terms, each of which is of the form $1_{(N,2N]}$ for some $1 \ll N \ll X$.  This decomposes $d_k 1_{(X,2X]}$ as the sum of $O_k( \log^k X )$ terms of the form
$$ (1_{(N_1,2N_1]} \ast \dotsb \ast 1_{(N_k,2N_k]}) 1_{(X,2X]}$$
and this is clearly of the required form \eqref{pre} thanks to Lemma \ref{point}.

Now suppose that $f = \Lambda 1_{(X,2X]}$.  Here we use the well-known Heath-Brown identity \cite[Lemma 1]{hb-ident}.  Let $K$ be the first natural number such that $K \geq \frac{1}{\eps} \geq m$, thus $K = O_{m,\eps}(1)$.  The Heath-Brown identity then gives
\begin{equation}\label{lao}
 \Lambda = \sum_{j=1}^K (-1)^{j+1} \binom{K}{j} L \ast 1^{\ast (j-1)} * (\mu 1_{[1,(2X)^{1/K}]})^{\ast j}
\end{equation}
on the interval $(X,2X]$, where $f^{\ast j}$ denotes the Dirichlet convolution of $j$ copies of $f$.  Clearly we may replace $L$ and $1$ by $L 1_{[1,2X]}$ and $1_{[1,2X]}$ respectively without affecting this identity on $(X,2X]$.  As before, we can decompose $1_{[1,2X]}$ into $O(\log X)$ terms of the form $1_{(N,2N]}$ for some $1 \ll N \ll X$; one similarly decomposes $L 1_{[1,2X]}$ into $O(\log X)$ terms of the form $1_{(N,2N]}$ for $1 \ll N \ll X$, and $\mu 1_{[1,(2X)^{1/K}]}$ into $O(\log X)$ terms of the form $\mu 1_{(N,2N]}$ for $1 \ll N \leq (2X)^{1/K} \ll X^\eps$.  Inserting all these decompositions into \eqref{lao} and using Lemma \ref{point}, we obtain the desired expansion of $f$ into $O_{k,m,\eps}(\log^{O_{k,m,\eps}(1)} X)$ terms of the form \eqref{pre}.

In view of the above decomposition, it suffices to show that each individual term of the form \eqref{pre} can be expressed as the sum of $O_{k,m,\eps}(1)$ terms, each of which are either a Type $d_j$ sum for some $1 \leq j < m$ or as a Type II sum (note that the coefficients of the linear combination can be absorbed into the $\alpha$ factor for both the Type $d_j$ and the Type II sums).  First note that we may assume that
\begin{equation}\label{nrx}
 N_1 \dotsm N_r \asymp_{k,m,\eps} X
\end{equation}
otherwise the expression in \eqref{pre} vanishes.  By symmetry we may also assume that $N_1 \leq \dotsb \leq N_r$.  We may also assume that $X$ is sufficiently large depending on $k,m,\eps$ as the claim is trivial otherwise (every arithmetic function of interest would be a Type II sum, for instance, setting $\alpha$ to be the Kronecker delta function at one).

Let $0 \leq s \leq r$ denote the largest integer for which
\begin{equation}\label{nas}
 N_1 \dotsm N_s \leq X^\eps.
\end{equation}
From \eqref{nrx} we have $s < r$ (if $X$ is large enough).  We divide into two cases, depending on whether $N_1 \dotsm N_{s+1} \leq 2H_0$ or not.  First suppose that $N_1 \dotsm N_{s+1} \leq 2H_0$, then by construction we have
$$ X^\eps \leq N_1 \dotsm N_{s+1} \leq 2H_0.$$
One can then almost express \eqref{pre} as a Type II sum by setting 
$$\alpha \coloneqq  \gamma_1 \ast \dotsb \ast \gamma_{s+1} $$
and
$$\beta \coloneqq  \gamma_{s+2} \ast \dotsb \ast \gamma_r$$
and using Lemma \ref{good-cancel}.  The only difficulty is that $\alpha$ is not quite supported on an interval of the form $[N,2N]$, instead being supported on $[N_1 \dotsm N_{s+1}, 2^{s+1} N_1 \dotsm N_{s+1}]$, and similarly for $\beta$; but this is easily rectified by decomposing both $\alpha$ and $\beta$ dyadically into $O_{k,m,\eps}(1)$ pieces, each of which are supported in an interval of the form $[N,2N]$.

Finally we consider the case when $N_1 \dotsm N_{s+1} > 2H_0$.  Since $H_0 \geq X^{\frac{1}{m}+\eps}$, we conclude from \eqref{nas} that 
$$ N_r \geq \dotsb \geq N_{s+2} \geq N_{s+1} > 2 X^{\frac{1}{m}} \geq (2X)^{\frac{1}{m}}.$$
In particular, if $r-s \geq m$, then $N_{s+1} \dotsm N_r > 2X$ and \eqref{pre} vanishes.  Thus we may assume that $s = r-j$ for some $1 \leq j < m$.  Also, as $N_r,\dotsc,N_{s+1}$ are significantly larger than $X^\eps$, the $\gamma_j$ for $j=s+1,\dotsc,r$ must be of the form $1_{(N_j,2N_j]}$ or $L 1_{(N_j,2N_j]}$.  One can then almost express \eqref{pre} as a Type $d_j$ sum by setting
$$ \alpha \coloneqq  \gamma_1 \ast \dotsb \ast \gamma_s$$
and
$$ \beta_i \coloneqq  \gamma_{s+i}$$
for $i=1,\dotsc,j$ and using Lemma \ref{good-cancel}.  The support of $\alpha$ is again slightly too large, but this can be rectified as before by a dyadic decomposition.
\end{proof}

For technical reasons (arising from the terms in Lemma \ref{edc} when $q_0>1$), we will need a more complicated variant of this proposition, in which one decomposes the function $n \mapsto f(q_0 n)$ rather than $f$ itself.   This introduces some additional ``small'' sums which are not Dirichlet convolutions, but which are quite small in $\ell^2$ norm and so can be easily managed using crude estimates such as Lemma \ref{mvt-lem}.

\begin{lemma}[Combinatorial decomposition, II]\label{comb-decomp-2}  Let $k, m, B \geq 1$ and $0 < \eps < \frac{1}{m}$ be fixed.  Let $X \geq 2$, and let $H_0$ be such that $X^{\frac{1}{m} + \eps} \leq H_0 \leq X$.  Let $q_0$ be a natural number with $q_0 \leq \log^B X$.  Let $f \colon \N \to \C$ be either the function $f \coloneqq  \Lambda 1_{(X,2X]}$ or $f \coloneqq  d_k 1_{(X,2X]}$.  Then one can decompose the function $f(q_0 \cdot): n \mapsto f(q_0 n)$ as a linear combination (with coefficients of size $O_k (d_2(q_0)^{O_k(1)})$) of $O_{k,m,\eps}( \log^{O_{k,m,\eps}(1)} X )$ components $\tilde f$, each of which is of one of the following types:
\begin{itemize}
\item[(Type $d_j$ sum)]  A function of the form
\begin{equation}\label{fst2}
 \tilde f = (\alpha \ast \beta_1 \ast \dotsb \ast \beta_j) 1_{(X/q_0,2X/q_0]}
\end{equation}
for some arithmetic functions $\alpha,\beta_1,\dotsc,\beta_j \colon \N \to \C$, where $1 \leq j < m$, $\alpha$ is $O_{k,m,\eps}(1)$-divisor-bounded  and supported on $[N,2N]$, and each $\beta_i$, $i=1,\dotsc,j$ is either of the form $\beta_i = 1_{(M_i,2M_i]}$ or $\beta_i = L 1_{(M_i,2M_i]}$ for some $N, M_1,\dotsc,M_j$ obeying the bounds \eqref{n-small},
\begin{equation}\label{nmx-q}
N M_1 \dotsm M_j \asymp_{k,m,\eps} X/q_0
\end{equation}
and 
$$ H_0 \ll M_1 \ll \dotsb \ll M_j \ll X/q_0.$$ 
\item[(Type II sum)]  A function of the form
$$ \tilde f = (\alpha \ast \beta) 1_{(X/q_0,2X/q_0]}$$
for some $O_{k,m,\eps}(1)$-divisor-bounded arithmetic functions $\alpha,\beta \colon \N \to \C$ with good cancellation supported on $[N,2N]$ and $[M,2M]$ respectively, for some $N,M$ obeying the bounds \eqref{n-medium} and
\begin{equation}\label{nmx-2q}
NM \asymp_{k,m,\eps} X / q_0.
\end{equation}
The good cancellation bounds \eqref{alb} are permitted to depend on the parameter $B$ in this lemma (in particular, $B'$ can be assumed to be large depending on this parameter).
\item[(Small sum)]  A function $\tilde f$ supported on $(X/q_0,2X/q_0]$ obeying the bound
$$ \|\tilde f \|_{\ell^2}^2 \ll_{k,m,\eps,B} X^{1-\eps/8}.$$
\end{itemize}
\end{lemma}

\begin{proof}  If $q_0=1$ then the claim follows from Lemma \ref{comb-decomp}, so suppose $q_0 > 1$.  We first dispose of the case when $f = \Lambda 1_{(X,2X]}$.  From the support of the von Mangoldt function we see that the function $f(q_0 \cdot): n \mapsto f(q_0 n)$ vanishes unless $q_0$ is a prime power $p_0^{k_0}$ and in this case is supported on powers of $p_0$. Thus
\begin{align*}
\|f(q_0 \cdot)\|_{\ell^2}^2 &\leq \sum_{k \geq 0} 1_{(X,2X]}(p_0^{k_0 + k}) \Lambda(p_0^{k_0 + k})^2 \ll (\log X)^3.
\end{align*}
Thus $f(q_0 \cdot)$ is already a small sum, and we are done in this case.

It remains to consider the case when $f = d_k 1_{(X,2X]}$.  Let $d_k(q_0 \cdot)$ denote the function $n \mapsto d_k(q_0 n)$.  The function $d_k(q_0 \cdot)/d_k(q_0)$ is multiplicative, hence by M\"obius inversion we may factor
$$ d_k(q_0 \cdot) = d_k(q_0) d_k \ast g$$
where $g$ is the multiplicative function
$$ g \coloneqq  \frac{1}{d_k(q_0)} d_k(q_0 \cdot) \ast \mu^{\ast k}$$ 
and $\mu^{\ast k}$ is the Dirichlet convolution of $k$ copies of $\mu$.  From multiplicativity we see that $g$ is $O_k(1)$-divisor-bounded, non-negative and supported on the multiplicative semigroup ${\mathcal G}$ generated by the primes dividing $q_0$.  We split $g = g_1 + g_2$, where $g_1(n) \coloneqq  g(n) 1_{n \leq X^{\eps/2}}$ and $g_2(n) \coloneqq  g(n) 1_{n > X^{\eps/2}}$, thus
$$ f(q_0 \cdot) = d_k(q_0) (d_k \ast g_1) 1_{(X/q_0,2X/q_0]} + d_k(q_0) (d_k \ast g_2) 1_{(X/q_0,2X/q_0]}.$$
The term $(d_k \ast g_1) 1_{(X/q_0,2X/q_0]}$ can be decomposed into terms of the form \eqref{pre} (but with $X$ replaced by $X/q_0$), exactly as in the proof of Lemma \ref{comb-decomp} (with the additional factor $g_1$ simply being an additional term $\gamma_i$), so by repeating the previous arguments (with $X$ replaced by $X/q_0$ as appropriate), we obtain the required decomposition of this term as a linear combination (with coefficients of size $O( d_k(q_0) ) = O(d_2(q_0)^{O_k(1)})$) of Type $d_j$ and Type II sums, using the second part of Lemma \ref{good-cancel} to ensure that convolution by $g_1$ does not destroy the good cancellation property.

It remains to handle the $(d_k \ast g_2) 1_{(X/q_0,2X/q_0]}$ term, which we will show to be small.  Indeed, we expand
$$ \|(d_k \ast g_2) 1_{(X/q_0,2X/q_0]}\|_{\ell^2}^2 = \sum_{n: X < q_0 n \leq 2X} |d_k \ast g_2(q_0 n)|^2.$$
We can expand out the square and bound this by
\begin{equation}\label{curd}
 \ll \sum_{m_1,m_2} g_2(m_1) g_2(m_2) \sum_{X < n \leq 2X: m_1,m_2,q_0|n} d_k\left(\frac{n}{m_1}\right) d_k\left(\frac{n}{m_2}\right),
\end{equation}
where $m_1,m_2$ range over the natural numbers.
We crudely drop the constraint $q_0|n$.  From \eqref{divisor-crude} (factoring $n = [m_1,m_2] n'$ and noting that $d_k\left(\frac{n}{m_1}\right) \leq d_2(m_2)^{O_k(1)} d_2(n')^{O_k(1)}$, and similarly for $d_k\left(\frac{n}{m_2}\right)$) we have
$$ \sum_{X < n \leq 2X: m_1,m_2|n} d_k\left(\frac{n}{m_1}\right) d_k\left(\frac{n}{m_2}\right) \ll_k \frac{X d_2(m_1)^{O_k(1)} d_2(m_2)^{O_k(1)} \log^{O_k(1)} X}{[m_1,m_2]} $$
and also crudely bounding $\frac{1}{[m_1,m_2]} \leq \frac{1}{m_1^{1/2} m_2^{1/2}}$,
we can bound \eqref{curd} by
$$
\ll_k \left( \sum_{m} \frac{g_2(m) d_2(m)^{O_k(1)}}{m^{1/2}}\right)^2 X \log^{O_k(1)} X
$$
which on bounding $g_2(m) \leq X^{-\eps/8} g(m) m^{1/4}$ becomes
$$
\ll_{k,\eps} X^{-\eps/8}  X \left(\sum_{m} \frac{g(m) d_2(m)^{O_k(1)}}{m^{1/4}}\right)^2 \log^{O_k(1)} X.
$$
From Euler products and the support and bounds on $g$ we have
$$ \sum_{m} \frac{g(m) d_2(m)^{O_k(1)}}{m^{1/4}} \ll_k d_2(q_0)^{O_k(1)}$$
and so by using \eqref{alpha-infty}, we conclude that $(d_k \ast g_2) 1_{(X/q_0,2X/q_0]}$ is small as required.
\end{proof}

\section{Applying the circle method}

Let $f,g \colon \Z \to \C$ be functions supported on a finite set, and let $h$ be an integer.  Following the Hardy-Littlewood circle method, we can express the correlation \eqref{co} as an integral
$$ \sum_n f(n) \overline{g}(n+h) = \int_\T S_f(\alpha) \overline{S_g(\alpha)} e(\alpha h)\ d\alpha
$$
where $S_f, S_g \colon \T \to \C$ are the exponential sums
\begin{align*}
S_f(\alpha) &\coloneqq  \sum_n f(n) e(\alpha n) \\
S_g(\alpha) &\coloneqq  \sum_n g(n) e(\alpha n).
\end{align*}
If we then designate some (measurable) portion ${\mathfrak M}$ of the unit circle $\T$ to be the ``major arcs'', we thus have
\begin{equation}\label{sam}
\sum_n f(n) \overline{g}(n+h) -  \operatorname{MT}_{{\mathfrak M},h} = \int_{\mathfrak m} S_f(\alpha) \overline{S_g(\alpha)} e(\alpha h)\ d\alpha
\end{equation}
where $ \operatorname{MT}_{{\mathfrak M},h} $ is the \emph{main term}
\begin{equation}\label{mt-def}
 \operatorname{MT}_{{\mathfrak M},h} \coloneqq  \int_{\mathfrak M} S_f(\alpha) \overline{S_g(\alpha)} e(\alpha h)\ d\alpha
\end{equation}
and ${\mathfrak m} \coloneqq  \T \backslash {\mathfrak M}$ denotes the complementary \emph{minor arcs}.

We will choose the major arcs so that the main term can be computed for any given $h$ by classical techniques (basically, the Siegel-Walfisz theorem, together with the analogous asymptotics for the divisor functions $d_k$). To control the minor arcs, we take advantage of the ability to average in $h$ to control this contribution by certain short $L^2$ integrals of the exponential sum $S_f(\alpha)$ (the factor $S_g(\alpha)$ will be treated by a trivial bound). 

\begin{proposition}[Circle method]\label{circle-method}  Let $H \geq 1$, and let $f, g, {\mathfrak M}, {\mathfrak m}, S_f, S_g, \operatorname{MT}_{{\mathfrak M},h}$ be as above.  
Then for any integer $h_0$, we have
\begin{equation}\label{bound}
\begin{split}
&\sum_{|h-h_0| \leq H} \left|\sum_n f(n) \overline{g(n+h)} -  \operatorname{MT}_{{\mathfrak M},h}\right|^2 \\
&\quad\ll H
 \int_{\mathfrak m} |S_f(\alpha)| |S_g(\alpha)|
\int_{\mathfrak m \cap [\alpha -1/2H, \alpha+1/2H]} |S_f(\beta)| |S_g(\beta)|\ d\beta d\alpha
\end{split}\end{equation}
\end{proposition}

\begin{proof}   
From \eqref{sam}, the left-hand side of \eqref{bound} may be written as
$$
\sum_{|h-h_0| \leq H} \left|\int_{\mathfrak m} S_f(\alpha) \overline{S_g(\alpha)} e(\alpha h)\ d\alpha\right|^2.$$
Next, we introduce an even non-negative Schwartz function $\Phi \colon \R \to \R^+$ with $\Phi(x) \geq 1$ for all $x \in [-1,1]$, such that the Fourier transform $\hat\Phi(\xi) \coloneqq  \int_\R \Phi(x) e(-x \xi)\ dx$ is supported in $[-1/2,1/2]$.  (Such a function may be constructed by starting with the inverse Fourier transform of an even test function supported on a small neighbourhood of the origin, and then squaring.)  Then we may bound the preceding expression by
$$ \sum_h \left|\int_{\mathfrak m} S_f(\alpha) \overline{S_g(\alpha)} e(\alpha h)\ d\alpha\right|^2 \Phi\left(\frac{h-h_0}{H}\right).$$
Expanding out the square, rearranging, and using the triangle inequality, we may bound this expression by
$$
 \int_{\mathfrak m} |S_f(\alpha)| |S_g(\alpha)|
\int_{\mathfrak m} |S_f(\beta)| |S_g(\beta)| \left|\sum_h e((\alpha-\beta) h) \Phi\left(\frac{h-h_0}{H}\right)\right|\ d\beta d\alpha.
$$
From the Poisson summation formula we have
$$ \sum_h e((\alpha-\beta) h) \Phi\left(\frac{h}{H}\right) = H \sum_k \hat \Phi( H (\tilde \alpha - \tilde \beta + k) )$$
where $\tilde \alpha, \tilde \beta$ are any lifts of $\alpha,\beta$ from $\T$ to $\R$.  In particular, this expression is of size $O(H)$, and vanishes unless $\beta$ lies in the interval $[\alpha-1/2H, \alpha+1/2H]$.  Shifting $h$ by $h_0$, the claim follows.
\end{proof}

From the Plancherel identities
\begin{align}
\int_\T |S_f(\alpha)|^2 d\alpha &= \|f\|_{\ell^2}^2 \label{planch-1}\\
\int_\T |S_g(\alpha)|^2 d\alpha &= \|g\|_{\ell^2}^2 \label{planch-2}
\end{align}
and Cauchy-Schwarz, we have
$$ \int_{\mathfrak m} |S_f(\alpha)| |S_g(\alpha)|\ d\alpha \leq \|f\|_{\ell^2} \|g\|_{\ell^2}$$
so we can bound the right-hand side of \eqref{bound} by
$$ \|f\|_{\ell^2} \|g\|_{\ell^2} \sup_{\alpha \in {\mathfrak m}} \int_{\mathfrak m \cap [\alpha -1/2H, \alpha+1/2H]} |S_f(\beta)| |S_g(\beta)|\ d\beta.$$
By \eqref{planch-2} and Cauchy-Schwarz, we may bound this expression in turn by
\begin{equation}\label{bound-2}
 \|f\|_{\ell^2} \|g\|_{\ell^2}^2 \sup_{\alpha \in {\mathfrak m}} \left(\int_{\mathfrak m \cap [\alpha -1/2H, \alpha+1/2H]} |S_f(\beta)|^2\ d\beta\right)^{1/2}.
\end{equation}

Note that from \eqref{planch-1} we have the trivial upper bound 
\begin{equation}\label{triv-planch}
\int_{{\mathfrak m} \cap [\alpha-1/2H, \alpha+1/2H]} |S_f(\alpha)|^2\ d\alpha \leq \|f\|_{\ell^2}^2
\end{equation}
and so the right-hand side of \eqref{bound} may be crudely upper bounded by $H \|f\|_{\ell^2}^2 \|g\|_{\ell^2}^2$, which is essentially the trivial bound on \eqref{bound} that one obtains from the Cauchy-Schwarz inequality.  Thus, any significant improvement (e.g. by a large power of $\log X$) over \eqref{triv-planch} for minor arc $\alpha$ will lead to an approximation of the form
$$ \sum_n f(n) \overline{g(n+h)} \approx  \operatorname{MT}_{{\mathcal M},h}$$
for most $h \in [h_0-H,h_0+H]$.  We formalize this argument as follows:

\begin{corollary}\label{chebyshev}  Let $H \geq 1$ and $\eta, F, G, X > 0$.  Let $f,g \colon \Z \to \C$ be functions supported on a finite set, let ${\mathfrak M}$ be a  measurable subset of $\T$, and let ${\mathfrak m} \coloneqq  \T \backslash {\mathfrak M}$.  For each $h$, let $\operatorname{MT}_h$ be a complex number.  Let $h_0$ be an integer. Assume the following axioms:
\begin{itemize}
\item[(i)] (Size bounds)  One has $\|f\|_{\ell^2}^2 \ll F^2 X$ and $\|g\|_{\ell^2}^2 \ll G^2 X$.
\item[(ii)]  (Major arc estimate)  For all but $O(\eta H)$ integers $h$ with $|h-h_0| \leq H$, one has
$$ \int_{\mathfrak M} S_f(\alpha) \overline{S_g(\alpha)} e(\alpha h)\ d\alpha = \operatorname{MT}_{h} + O( \eta F G X ).$$
\item[(iii)]   (Minor arc estimate)  For each $\alpha \in {\mathfrak m}$, one has
\begin{equation}\label{fF}
 \int_{{\mathfrak m} \cap [\alpha-1/2H, \alpha+1/2H]} |S_f(\alpha)|^2\ d\alpha \ll \eta^6 F^2 X.
\end{equation}
\end{itemize}
Then for all but $O(\eta H)$ integers $h$ with $|h-h_0| \leq H$, one has
\begin{equation}\label{fg}
 \sum_n f(n) \overline{g(n+h)} = \operatorname{MT}_{h} + O( \eta FG X ).
\end{equation}
\end{corollary}

In our applications, $F$ and $G$ will behave like a fixed power of $\log X$, and $\eta$ will be set $\log^{-A} X$ for some large $A$.  By symmetry one can replace $f,F$ in \eqref{fF} with $g,G$ if desired, but note that we only need a minor arc estimate for one of the two functions $f,g$.

\begin{proof}  From Proposition \ref{circle-method}, the upper bound \eqref{bound-2}, and axioms (i), (iii) we have
$$ \sum_{|h-h_0| \leq H} \left|\sum_n f(n) \overline{g(n+h)} -  \operatorname{MT}_{{\mathfrak M},h}\right|^2 \ll \eta^3 F^2 G^2 H X^2$$
and hence by Chebyshev's inequality we have
$$ \sum_n f(n) \overline{g(n+h)} -  \operatorname{MT}_{{\mathfrak M},h} = O( \eta F G X ) $$
for all but $O(\eta H)$ integers $h$ with $|h-h_0| \leq H$.  Applying axiom (ii), \eqref{mt-def} and the triangle inequality, we obtain the claim.
\end{proof}

In view of the above corollary, Theorem \ref{unav-corr} will be an easy consequence of major and minor arc estimates which we will soon present.  Given parameters $Q \geq 1$ and $\delta > 0$, define the major arcs
$$ {\mathfrak M}_{Q,\delta} \coloneqq  \bigcup_{1 \leq q \leq Q} \bigcup_{a: (a,q) = 1} \left[\frac{a}{q} - \delta, \frac{a}{q} + \delta\right],$$
where we identify intervals such as $[\frac{a}{q}-\delta, \frac{a}{q}+\delta]$ with subsets of the unit circle $\T$ in the usual fashion.
We will take $Q \coloneqq \log^B X$ and $\delta \coloneqq X^{-1} \log^{B'} X$ for some large $B' > B > 1$.  To handle the major arcs, we use the following estimate:

\begin{proposition}[Major arc estimate]\label{major}
Let $A >0$, $0 < \eps < 1/2$ and $k,l \geq 2$ be fixed, and suppose that $X \geq 2$, $B \geq 2A$ and $B' \geq 2B+A$.  Let $h$ be an integer with $0 < |h| \leq X^{1-\eps}$. Let $P_{k, l, h}, Q_{k, h}$ and $\mathfrak{S}(h)$ be as in Section \ref{se:intro}.
\begin{itemize}
\item[(i)]  (Major arcs for Hardy-Littlewood conjecture) We have
\begin{equation}\label{dust}
\begin{split}
\int_{{\mathfrak M}_{\log^B X,X^{-1} \log^{B'} X}} |S_{\Lambda 1_{(X,2X]}}(\alpha)|^2 e(\alpha h)\ d\alpha &= {\mathfrak G}(h) X \\
&\quad  + O_{\eps,A,B, B'}( d_2(h)^{O(1)} X \log^{-A} X ).
\end{split}
\end{equation}
\item[(ii)]  (Major arcs for divisor correlation conjecture) We have
\begin{align*}
\int_{{\mathfrak M}_{\log^B X,X^{-1} \log^{B'} X}} S_{d_k 1_{(X,2X]}}(\alpha) \overline{S_{d_l 1_{(X,2X]}}(\alpha)} e(\alpha h)\ d\alpha &= P_{k,l,h}(\log X) X \\
+ O_{\eps,k,l,A,B,B'}&( d_2(h)^{O_{k,l}(1)} X \log^{k+l-2-A} X ).
\end{align*}
\item[(iii)]  (Major arcs for higher order Titchmarsh problem)  We have
\begin{align*}
\int_{{\mathfrak M}_{\log^B X,X^{-1} \log^{B'} X}} S_{\Lambda 1_{(X,2X]}}(\alpha) \overline{S_{d_k 1_{(X,2X]}}(\alpha)} e(\alpha h)\ d\alpha &= Q_{k,h}(\log X) X \\
+ O_{\eps,k,A,B,B'}&( d_2(h)^{O_k(1)} X \log^{k-1-A} X ).
\end{align*}
\item[(iv)]  (Major arcs for Goldbach conjecture)  If $X$ is an integer, then
\begin{align*}
\int_{{\mathfrak M}_{\log^B X,X^{-1} \log^{B'} X}} S_{\Lambda 1_{(X,2X]}}(\alpha) S_{\Lambda 1_{[1,X)}}(\alpha) e(-\alpha X)\ d\alpha &= {\mathfrak G}(X) X \\
+ O_{\eps,A,B,B'}&( d_2(X)^{O(1)} X \log^{-A} X ).
\end{align*}
\end{itemize}
\end{proposition}

These bounds are quite standard and will be established in Section \ref{major-sec}.  It is likely that one can remove the factors of $d_2(h), d_2(X)$ from the error terms with a little more effort, but we will not need to do so here, as these factors will usually be dominated by the $\log^{-A} X$ savings.  In case (ii), it is also likely that we can improve the error term to a power saving in $X$ if one enlarges the major arcs accordingly, but we will again not do so here.

To handle the minor arcs, we use the following exponential sum estimate:

\begin{proposition}[Minor arc estimate]\label{minor}  Let $\eps>0$ be a sufficiently small absolute constant, and let $A,B,B'>0$.  Let $k \geq 2$ be fixed, let $X \geq 2$, and set $Q \coloneqq  \log^B X$.  Assume that $B$ is sufficiently large depending on $A,k$, and that $B'$ is sufficiently large depending on $A, B, k$.

Let $1 \leq q \leq Q$, let $a$ be coprime to $q$. Let $f \colon \N \to \C$ be either the function $f(n) \coloneqq  \Lambda(n) 1_{(X,2X]}$ or $f(n) \coloneqq  d_k(n) 1_{(X,2X]}$.  
\begin{itemize}
\item[(i)]  One has
\begin{equation}\label{ma0}
 \int_{X^{-1} \log^{B'} X \ll |\theta| \ll X^{-1/6-\eps}} \left|S_f\left(\frac{a}{q}+\theta\right)\right|^2\ d\theta \ll_{k,\eps,A,B,B'} X \log^{-A} X 
\end{equation}
\item[(ii)]  One has, for $\sigma \geq 8/33$, the bound
\begin{equation}\label{mae}
 \int_{|\theta - \beta| \ll X^{-\record-\eps}} \left|S_f\left(\frac{a}{q}+\theta\right)\right|^2\ d\theta \ll_{k,\eps,A,B} X \log^{-A} X
\end{equation}
for any real number $\beta$ with $X^{-1/6-\eps} \ll |\beta| \leq \frac{1}{qQ}$.
\end{itemize}
\end{proposition}

Note from \eqref{alpha-2} and \eqref{planch-1} that one already has the bound
$$  \int_\T \left|S_f\left(\frac{a}{q}+\theta\right)\right|^2\ d\theta  \ll_{k,\eps} X \log^{O_k(1)} X,$$
so the bounds \eqref{ma0}, \eqref{mae} only gain a logarithmic savings over the trivial bound.  We also remark that the $\sigma \geq 1/3$ case of Proposition \ref{minor}(ii) can be established from the estimates in \cite{mikawa} (in the case $f = \Lambda 1_{(X,2X]}$) or \cite{bbmz} (in the case $f = d_3 1_{(X,2X]}$).

Proposition \ref{minor} will be proven in Sections \ref{dirichlet-sec}-\ref{vdc-sec}.  Assuming it for now, let us see why it (and Proposition \ref{major}) imply Theorem \ref{unav-corr}.  The four cases are very similar and we will only describe the argument in detail for Theorem \ref{unav-corr}(i).  By subdividing the interval $[h_0-H,h_0+H]$ if necessary, we may assume that $H = X^{\sigma+\eps}$ with $\eps$ small.  Let $A>0$, let $B>0$ be sufficiently large depending on $A$, and let $B'>0$ be sufficiently large depending on $A,B$.  We apply Corollary \ref{chebyshev} with $f=g=\Lambda$, $\eta = \log^{-A-2} X$, and ${\mathfrak M} := {\mathfrak M}_{\log^B X, X^{-1} \log^{B'} X}$.  From crude bounds we can verify the hypothesis in Corollary \ref{chebyshev}(i) with $F,G = \log X$.  From the estimates in \cite{landreau}, we know that
$$ \sum_{h: |h-h_0| \leq H} d_2(h) \ll H \log^{O(1)} X$$
and hence by the Markov inequality we have $d_2(h) \ll \log^{O_A(1)} X$ for all but $O(\eta H)$ values of $h \in [h_0-H,h_0+H]$.  This fact and Proposition \ref{major}(i) then give the hypothesis in Corollary \ref{chebyshev}(ii).  It remains to verify the hypothesis in Corollary \ref{chebyshev}(iii) for any $\alpha \not \in {\mathfrak M}_{\log^B X, X^{-1} \log^{B'} X}$.  By the Dirichlet approximation theorem, we can write $\alpha = a/q + \beta$ for some $1 \leq q \leq \log^B X$, $(a,q)=1$, and $|\beta| \leq \frac{1}{qQ}$.  Since $\alpha \not \in {\mathfrak M}_{\log^B X, X^{-1} \log^{B'} X}$, we also have $|\beta| \geq X^{-1} \log^{B'} X$.  If $\beta \leq X^{-\frac{1}{6}-\eps}$, the claim then follows from Proposition \ref{minor}(ii), while for $\beta > X^{-\frac{1}{6}-\eps}$ the claim follows from Proposition \ref{minor}(i). Theorem~\ref{unav-corr}(ii)-(iv) follow similarly (with slightly larger choices of $F,G$).

\begin{remark} The bound \eqref{mae} is being used here to establish Theorem \ref{unav-corr}.  In the converse direction, it is possible to use Theorem \ref{unav-corr} to establish \eqref{mae}; we sketch the argument as follows.  The left-hand side of \eqref{mae} may be bounded by
$$
 \int_\R \left|S_f\left(\theta\right)\right|^2 \eta\left( X^{\record +\eps} \left(\theta - \frac{a}{q} - \beta\right) \right)\ d\theta$$
for some rapidly decreasing $\eta$ with compactly supported Fourier transform,and this can be rewritten as
$$
X^{-\record-\eps} \sum_h e\left(-h\left(\frac{a}{q}+\beta\right)\right) \hat \eta\left( \frac{h}{X^{\record+\eps}} \right) \sum_n f(n) \overline{f(n+h)}.$$
The inner sum can be controlled for most $h$ using Theorem \ref{unav-corr}, and the contribution of the exceptional values of $h$ can be controlled by upper bound sieves.  We leave the details to the interested reader.
\end{remark}

\section{Major arc estimates}\label{major-sec}

In this section we prove Proposition \ref{major}.  To do this we need some estimates on $S_{\Lambda 1_{(X,2X]}}(\alpha)$ and $S_{d_k 1_{(X,2X]}}(\alpha)$ for major arc $\alpha$.  The former is standard:

\begin{proposition}\label{kog}  Let $A, B,B' > 0$, $X \geq 2$, and let $\alpha = \frac{a}{q} + \beta$ for some $1 \leq q \leq \log^B X$, $(a,q)=1$, and $|\beta| \leq \frac{\log^{B'} X}{X}$.  Then we have
$$ S_{\Lambda 1_{[1,X]}}(\alpha) = \frac{\mu(q)}{\phi(q)} \int_1^{X} e(\beta x)\ dx + O_{A,B,B'}( X \log^{-A} X )$$
and hence also
$$ S_{\Lambda 1_{(X,2X]}}(\alpha) = \frac{\mu(q)}{\phi(q)} \int_X^{2X} e(\beta x)\ dx + O_{A,B,B'}( X \log^{-A} X )$$
\end{proposition}

\begin{proof}  See \cite[Lemma 8.3]{nathanson}.  We remark that this estimate requires Siegel's theorem and so the bounds are ineffective.
\end{proof}

For the $d_k$ exponential sum, we have

\begin{proposition}\label{klog} Let $A, B,B' > 0$, $k \geq 2$, $X \geq 2$, and let $\alpha = \frac{a}{q} + \beta$ for some $1 \leq q \leq \log^B X$, $(a,q)=1$, and $|\beta| \leq \frac{\log^{B'} X}{X}$.  Then we have
$$ S_{d_k 1_{(X,2X]}}(\alpha) = \int_X^{2X} p_{k,q}(x) e(\beta x)\ dx + O_{k,A,B,B'}( X \log^{-A} X )$$
where
$$ p_{k,q}(x) \coloneqq \sum_{q=q_0q_1} \frac{\mu(q_1)}{\phi(q_1) q_0} p_{k,q_0,q_1}\left(\frac{x}{q_0}\right)$$
$$ p_{k,q_0,q_1}(x) \coloneqq \frac{d}{dx} \mathrm{Res}_{s=1} \frac{x^s F_{k,q_0,q_1}(s)}{s} $$
$$ F_{k,q_0,q_1}(s) \coloneqq \sum_{n \geq 1: (n,q_1) = 1} \frac{d_k(q_0 n)}{n^s}.$$
\end{proposition}

Using Euler products we see that $F_{k,q_0,q_1}$ has a pole of order $k$ at $s=1$, and so $p_{k,q_0,q_1}$ (and hence $p_{k,q}$) will be a polynomial of degree at most $k-1$ in $\log x$.  One could improve the error term $X \log^{-A} X$ here to a power savings $X^{1-c/k}$ for some absolute constant $c>0$, and also allow $q$ and $X |\beta|$ to similarly be as large as $X^{c/k}$, but we will not exploit such improved estimates here.

\begin{proof} This is a variant of the computations in \cite[\S 6]{bbmz}.  Using Lemma \ref{sbp} (and increasing $A$ as necessary), it suffices to show that
$$ \sum_{n \leq X'} d_k(n) e(an/q) = \int_0^{X'} p_{k,q}(x)\ dx + O_{k,A,B}( X \log^{-A} X )$$
for all $X' \asymp X$. Writing $n = q_0 n_1$ where $q_0 = (q, n)$, we can expand the left-hand side as
$$ \sum_{q=q_0q_1} \sum_{n_1 \leq X'/q_0: (n_1,q_1)=1} d_k(q_0 n_1) e(a n_1 / q_1)$$
so (again by enlarging $A$ as necessary) it will suffice to show that
\begin{equation}\label{dope}
 \sum_{n_1 \leq X'/q_0: (n_1,q_1)=1} d_k(q_0 n_1) e(a n_1 / q_1) =
\frac{\mu(q_1)}{\phi(q_1)} \int_0^{X'/q_0} p_{k,q_0,q_1}(x)\ dx + O_{k,A,B}( X \log^{-A} X )
\end{equation}
for each factorization $q = q_0 q_1$.  By \eqref{usual}, the left-hand side of \eqref{dope} expands as
$$ \frac{1}{\phi(q_1)} \sum_{\chi\ (q_1)} \chi(a) \tau(\overline{\chi}) \sum_{n_1 \leq X'/q_0} \chi(n_1) d_k(q_0 n_1)$$
where the Gauss sum $\tau(\overline{\chi})$ is defined by \eqref{taug}.
For non-principal $\chi$, a routine application of the Dirichlet hyperbola method shows that
$$ \sum_{n_1 \leq X'/q_0: (n_1,q_1)=1} \chi(n_1) d_k(q_0 n_1) \ll_{k,A,B} X \log^{-A} X$$
for any $A>0$ (in fact one can easily extract a power savings of order $\ll_{\varepsilon} X^{-1/k + \varepsilon}$ from this argument).  Thus it suffices to handle the contribution of the principal character.  Here, the Gauss sum \eqref{taug} is just $\mu(q_1)$, so we reduce to showing that
\begin{equation} \label{eq:divsumdk}
\sum_{n_1 \leq X'/q_0: (n_1,q_1)=1} d_k(q_0 n_1) = \int_0^{X'/q_0} p_{k,q_0,q_1}(x)\ dx + O_{k,A,B}( X \log^{-A} X ).
\end{equation}
By the fundamental theorem of calculus one has
$$ \int_0^{X'/q_0} p_{k,q_0,q_1}(x)\ dx  = \mathrm{Res}_{s=1} \frac{(X'/q_0)^s F_{k,q_0,q_1}(s)}{s}.$$
Meanwhile, by Lemma \ref{tpf-lem}(i), we can write the left-hand side of \eqref{eq:divsumdk} as 
$$ \frac{1}{2\pi i} \int_{\sigma-iX^{\eps}}^{\sigma+iX^\eps} \frac{(X'/q_0)^s F_{k,q_0,q_1}(s)}{s}\ ds + O_{k,A,B,\eps}( X \log^{-A} X)$$
where $\sigma \coloneqq 1 + \frac{1}{\log X}$ and $\eps>0$ is arbitrary.  On the other hand, by modifying the arguments in \cite[Lemma 4.3]{bbmz} (and using the standard convexity bound for the $\zeta$ function) we have for all sufficiently small $\varepsilon > 0$, the bounds
$$ F_{k,q_0,q_1}(s) \ll_{k,B,\eps} X^{O_k(\eps^2)}$$
when $\mathrm{Re}(s) \geq 1-\eps$, $|\mathrm{Im}(s)| \leq X^\eps$, and $|s-1| \geq \eps$.  Shifting the contour to the rectangular path connecting $\sigma-iX^{\eps}$, $(1-\eps)-iX^{\eps}$, $(1-\eps)+iX^\eps$, and $\sigma+iX^\eps$ and using the residue theorem, we obtain the claim.
\end{proof}

Now we establish Proposition \ref{major}(i).  From Proposition \ref{kog} and the trivial bound $| \int_{X}^{2X} e(\beta x) dx | \leq X$, we have
$$ |S_{\Lambda 1_{(X,2X]}}(\alpha)|^2 = \frac{\mu^2(q)}{\phi^2(q)} \left|\int_X^{2X} e(\beta x)\ dx\right|^2 + O_{A',B,B'}( X^2 \log^{-A'} X )$$
for any $A'>0$ and major arc $\alpha$ as in that proposition.  On the other hand, the set ${\mathfrak M}_{\log^B X,X^{-1} \log^{B'} X}$ has measure $O( X^{-1} \log^{2B+B'} X )$.
Thus (on increasing $A'$ as necessary) to prove \eqref{dust}, it suffices to show that
\begin{align*}
& \sum_{q \leq \log^B X} \sum_{(a,q)=1} \int_{|\beta| \leq X^{-1} \log^{B'} X} \frac{\mu^2(q)}{\phi^2(q)} \left|\int_X^{2X} e(\beta x)\ dx\right|^2 e\left(\left(\frac{a}{q}+\beta\right) h\right)\ d\beta \\
&\quad = {\mathfrak G}(h) X + O_{\eps,A,B, B'}( d_2(h)^{O(1)} X \log^{-A} X ).
\end{align*}
By the Fourier inversion formula we have
\begin{align*}
\int_\R \left|\int_X^{2X} e(\beta x)\ dx\right|^2 e(\beta h)\ d\beta 
&= \int_\R 1_{[X,2X]}(x) 1_{[X,2X]}(x+h)\ dx \\
&= (1 + O(X^{-\eps})) X
\end{align*}
so from the elementary bound $\int_X^{2X} e(\beta x)\ dx \ll 1/|\beta|$ one has
\begin{equation}\label{tbf}
\int_{|\beta| \leq X^{-1} \log^{B'} X} \left|\int_X^{2X} e(\beta x)\ dx\right|^2 e(\beta h)\ d\beta 
= (1 + O_{\eps,B'}( \log^{-B'} X)) X.
\end{equation}
Since $B' \geq 2B+A$, it thus suffices to show that
$$ \sum_{q \leq \log^B X} \sum_{(a,q)=1} \frac{\mu^2(q)}{\phi^2(q)} e\left(\frac{ah}{q}\right)= {\mathfrak G}(h) + O_{\eps,A,B}( d_2(h)^{O(1)} \log^{-A} X ).$$
Introducing the Ramanujan sum
\begin{equation}\label{cqd}
 c_q(a) \coloneqq \sum_{1 \leq b \leq q: (b,q) = 1} e\left(\frac{ab}{q}\right),
\end{equation}
the left-hand side simplifies to
$$ \sum_{q \leq \log^B X} \frac{\mu^2(q) c_q(h)}{\phi^2(q)}.$$

Recall that, for fixed $h$, $c_q(h)$ is multiplicative in $q$ and that $c_p(h) = -1$ if $p \nmid h$ and $c_p(h) = \varphi(h)$ if $p \mid h$. Hence by Euler products one has
\begin{align*}
\sum_q \frac{\mu^2(q) c_q(h)}{\phi^2(q)} q^{1/2} &\ll \prod_{p \nmid h} \left(1+O\left(\frac{1}{p^{3/2}}\right)\right) \times \prod_{p |h} O(1) \\
&\ll d_2(h)^{O(1)}
\end{align*}
and hence
$$ \sum_{q > \log^B X} \frac{\mu^2(q) c_q(h)}{\phi^2(q)} \ll d_2(h)^{O(1)} \log^{-B/2} X.$$
Since $B \geq 2A$, it thus suffices to establish the identity
$$ \sum_q \frac{\mu^2(q) c_q(h)}{\phi^2(q)} = {\mathfrak G}(h)$$
but this follows from a standard Euler product calculation.

The proof of Theorem \ref{major}(iv) is similar to that of Theorem \ref{major}(i) and is left to the reader.  We now turn to Theorem \ref{major}(ii).  From \eqref{divisor-crude} one has
$$ S_{d_k 1_{(X,2X]}}(\alpha) \ll_k X \log^{k-1} X$$
and similarly
$$ S_{d_l 1_{(X,2X]}}(\alpha) \ll_l X \log^{l-1} X.$$
Write
$$ S_{k,q}(\beta) \coloneqq \int_X^{2X} p_{k,q}(x) e(\beta x)\ dx$$
and similarly for $S_{l,q}(\beta)$.
Then from Proposition \ref{klog} we have (on increasing $A$ as necessary) that
$$ S_{d_k 1_{(X,2X]}}(\alpha) \overline{S_{d_l 1_{(X,2X]}}(\alpha)} = S_{k,q}(\beta) \overline{S_{l,q}(\beta)} + O_{k,l,A',B,B'}( X^2 \log^{-A'} X )
$$
for any $A'>0$.
It thus suffices to show that
\begin{align*}
& \sum_{q \leq \log^B X} \sum_{(a,q)=1} \int_{|\beta| \leq X^{-1} \log^{B'} X} 
S_{k,q}(\beta) \overline{S_{l,q}(\beta)} e\left(\left(\frac{a}{q}+\beta\right)h\right)\ d\beta\\
&=
P_{k,l,h}(\log X) X + O_{\eps,k,l,A,B,B'}( d_2(h)^{O_{k,l}(1)} X \log^{-A} X ).
\end{align*}

Using Euler products one can obtain the crude bounds
\begin{equation}\label{caqg}
 p_{k,q}(x) \ll_k \frac{d_2(q)^{O_k(1)}}{q} \log^{k-1} X
\end{equation}
for $x \asymp X$; indeed, the coefficients of $p_{k,q}$ (viewed as a polynomial in $\log X$) are of size $O_k( \frac{d_2(q)^{O_k(1)}}{q})$.  By repeating the proof of \eqref{tbf}, we can then conclude that
\begin{align*}
&\int_{|\beta| \leq X^{-1} \log^{B'} X} 
S_{k,q}(\beta) \overline{S_{l,q}(\beta)} e(\beta h)\ d\beta\\
&\quad = \int_X^{2X} p_{k,q}(x) \overline{p_{l,q}}(x+h)\ dx + O_{k,l,B'}\left( \frac{d_2(q)^{O_{k,l}(1)}}{q^2} X \log^{k+l-2-B'} X\right).
\end{align*}
Since $\log(x+h) = \log x + O_\eps(X^{-\eps})$ for $|h| \leq X^{1-\eps}$ and $x \asymp X$, we have
$$ \int_X^{2X} p_{k,q}(x) \overline{p_{l,q}}(x+h)\ dx  = \int_X^{2X} p_{k,q}(x) \overline{p_{l,q}}(x)\ dx + O_{k,l,B'}\left( d_2(q)^{O_{k,l}(1)} X \log^{k+l-2-B'} X\right).$$
Using \eqref{divisor-crude} to control the error terms, using \eqref{cqd}, and recalling that $B' \geq 2B+A$, it therefore suffices to establish the bound
$$ \sum_{q \leq \log^B X} c_q(h) \int_X^{2X} p_{k,q}(x) \overline{p_{l,q}}(x)\ dx = 
P_{k,l,h}(\log X) X + O_{\eps,k,l,A,B}\left( d_2(h)^{O_{k,l}(1)} X \log^{k+l-2-A} X \right).$$
Using the bounds \eqref{caqg}, we can argue as before to show that
$$ \sum_{q} q^{1/2} c_q(h) \int_X^{2X} p_{k,q}(x) \overline{p_{l,q}}(x)\ dx \ll_{k,l} d_2(h)^{O_{k,l}(1)} X \log^{k+l-2} X $$
and so for $B \geq 2A$ it suffices to show that
$$ \sum_q c_q(h) \int_X^{2X} p_{k,q}(x) \overline{p_{l,q}}(x)\ dx = X P_{k,l,h}(\log X).$$
But as $p_{k,q}, p_{l,q}$ are polynomials in $\log X$ of degree at most $k-1,l-1$ respectively, this follows from direct calculation (using \eqref{caqg} to justify the convergence of the summation).  An explicit formula for the polynomial $P_{k,l,h}$ may be computed by using the calculations in \cite{conrey}, but we will not do so here.

To prove Theorem \ref{major}(iii), we repeat the arguments used to establish Theorem \ref{major}(ii) (replacing one of the invocations of Proposition \ref{klog} with Proposition \ref{kog}) and eventually reduce to showing that
$$ \sum_q c_q(h) \int_X^{2X} p_{k,q}(x) \frac{\mu(q)}{\varphi(q)}\ dx = X Q_{k,h}(\log X),$$
but this is again clear since $p_{k,q}(x)$ is a polynomial in $\log X$ of degree at most $k-1$. Again, the polynomial $Q_{k,h}$ is explicitly computable, but we will not write down such an explicit formula here.

\section{Reduction to a Dirichlet series mean value estimate}\label{dirichlet-sec}

We begin the proof of Proposition \ref{minor}.  As discussed in the introduction, we will estimate the expressions \eqref{mae}, \eqref{ma0}, which currently involve the additive frequency variable $\alpha$, by expressions involving the multiplicative frequency $t$, by performing a sequence of Fourier-analytic transformations and changes of variable.

The starting point will be the following treatment of the $q=1$ case:

\begin{proposition}[Bounding exponential sums by Dirichlet series mean values]\label{spe}  Let $1 \leq H \leq X/2$, and let $f \colon \N \to \C$ be a function supported on $(X,2X]$.  Let $\beta, \eta$ be real numbers with $|\beta| \ll \eta \ll 1$, and let $I$ denote the region
\begin{equation}\label{I-def}
 I \coloneqq \left\{ t \in \R: \eta |\beta| X \leq |t| \leq \frac{|\beta| X}{\eta} \right\}
\end{equation}
\begin{itemize}
\item[(i)] We have
\begin{equation}\label{spe-2}
\begin{split}
\int_{\beta-1/H}^{\beta+1/H} |S_f(\theta)|^2\ d\theta  &\ll
\frac{1}{|\beta|^2 H^2} \int_I \left( \int_{t-|\beta| H}^{t+|\beta| H} |{\mathcal D}[f](\frac{1}{2}+it')|\ dt'\right)^2\ dt  \\
&\quad + \frac{\left(\eta + \frac{1}{|\beta| H}\right)^2 }{H^2} \int_\R \left(\sum_{x \leq n \leq x+H} |f(n)|\right)^2\ dx.
\end{split}
\end{equation}
\item[(ii)]  If $\beta = 1/H$, then we have the variant
\begin{equation}\label{spe-4}
\begin{split}
\int_{\beta \leq |\theta| \leq 2\beta} |S_f(\theta)|^2\ d\theta  &\ll 
\int_I |{\mathcal D}[f](\frac{1}{2}+it)|^2\ dt  \\
&\quad +  \frac{\left(\eta + \frac{1}{|\beta| X}\right)^2}{H^2} \int_\R \left(\sum_{x \leq n \leq x+H} |f(n)|\right)^2\ dx
\end{split}
\end{equation}
\end{itemize}
\end{proposition}

Observe from the Cauchy-Schwarz inequality that
$$
\frac{1}{|\beta|^2 H^2} \int_I \left(\int_{t-|\beta| H}^{t+|\beta| H} \left|{\mathcal D}[f]\left(\frac{1}{2}+it'\right)\right|\ dt'\right)^2\ dt \ll \int_{-2|\beta| X/\eta}^{2|\beta| X/\eta} \left|{\mathcal D}[f]\left(\frac{1}{2}+it\right)\right|^2\ dt $$
and so from the $L^2$ mean value estimate (Lemma \ref{mvt-lem}) we see that for $f$ a $k$-divisor bounded function \eqref{spe-2} trivially implies the bound
$$ \int_{\beta-1/H}^{\beta+1/H} |S_f(\theta)|^2\ d\theta   \ll_k X \log^{O_k(1)} X.$$
Note that this bound also follows from the ``trivial'' bounds \eqref{planch-1} and \eqref{divisor-crude}.  Thus, ignoring powers of $\log X$, \eqref{spe-2} is efficient in the sense that trivial estimation of the right-hand side recovers the trivial bound on the left-hand side.  In particular, any significant improvement (such as a power savings) over the trivial bound on the right-hand side will lead to a corresponding non-trivial estimate on the left-hand side, of the type needed for Proposition \ref{minor}.  Similarly for \eqref{spe-4} (which roughly corresponds to the endpoint $|\beta| \asymp \frac{1}{H}$ of \eqref{spe-2}).

\begin{proof}  For brevity we adopt the notation
$$ F(t) \coloneqq {\mathcal D}[f](\frac{1}{2}+it).$$
We first prove \eqref{spe-2}.  It will be convenient for Fourier-analytic computations to work with smoothed sums.  Let $\varphi \colon \R \to \R$ be a smooth even function supported on $[-1,1]$, equal to one on $[-1/10,1/10]$, and whose Fourier transform $\hat \varphi(\theta) := \int_\R \varphi(y) e(-\theta y)\ dy$ obeys the bound $|\hat \varphi(\theta)| \gg 1$ for $\theta \in [-1,1]$. Notice that $\varphi$ is a Schwartz function since it is smooth and compactly supported, thus $\widehat{\varphi}$ is also a Schwarz function. Then we have 
\begin{align*}
\int_{\beta-1/H}^{\beta+1/H} |S_f(\theta)|^2\ d\theta &\ll
\int_\R |S_f(\theta)|^2 |\hat \varphi( H(\theta - \beta) )|^2 \ d\theta \\
&= \int_\R \left| \int_\R \sum_n f(n) \varphi(y) e( \beta H y ) e( \theta (n - H y) )\ dy \right|^2\ d\theta \\
&= H^{-2} \int_\R \left| \int_\R \sum_n f(n) \varphi\left(\frac{n-x}{H}\right) e( \beta (n-x) ) e( \theta x )\ dx \right|^2\ d\theta \\
&= H^{-2} \int_\R \left| \sum_n f(n) \varphi\left(\frac{n-x}{H}\right) e( \beta (n-x) ) \right|^2\ dx \\
&= H^{-2} \int_\R \left| \sum_n f(n) \varphi\left(\frac{n-x}{H}\right) e( \beta n ) \right|^2\ dx \\
&= H^{-2} \int_{X/2}^{4X} \left| \sum_n f(n) \varphi\left(\frac{n-x}{H}\right) e( \beta n ) \right|^2\ dx.
\end{align*}
where we have made the change of variables $x = n-Hy$, followed by the Plancherel identity, and then used the support of $f$ and $\varphi$.  (This can be viewed as a smoothed version of a lemma of Gallagher \cite[Lemma 1]{gallagher}).)

By the triangle inequality, we can bound the previous expression by
$$ \ll H^{-2} \int_\R \left(\sum_{x-H \leq n \leq x+H} |f(n)|\right)^2\ dx,$$
which is acceptable if $|\beta| H \ll 1$ or $\eta \geq 1/100$. Thus we may assume henceforth that $|\beta| \gg 1/H$ and $\eta < 1/100$.
 
By duality, it thus suffices to establish the bound
\begin{equation}\label{fad}
\begin{split}
\int_\R \sum_n f(n) \varphi\left(\frac{n-x}{H}\right) e( \beta n ) g(x)\ dx &\ll
\frac{1}{|\beta|} \left( \int_I \left( \int_{t-|\beta| H}^{t+|\beta| H} |F(t')|\ dt'\right)^2\ dt \right)^{1/2}  \\
&\quad + 
\left(\eta + \frac{1}{|\beta| H}\right) \left(\int_\R \left(\sum_{x \leq n \leq x+H} |f(n)|\right)^2\ dx\right)^{1/2}
\end{split}
\end{equation}
whenever $g \colon \R \to \C$ is a measurable function supported on $[X/2,4X]$ with the normalization
\begin{equation}\label{g-norm}
\int_\R |g(x)|^2\ dx = 1.
\end{equation}

Using the change of variables $u = \log n - \log X$ (or equivalently $n = X e^u$), as discussed in the introduction, we can write the left-hand side of \eqref{fad} as
\begin{equation}\label{ga}
 \sum_n \frac{f(n)}{n^{1/2}} G(\log n - \log X) 
\end{equation}
where $G \colon \R \to \R$ is the function
$$ G(u) \coloneqq X^{1/2} e^{u/2} e(\beta X e^u) \int_\R \varphi\left(\frac{Xe^u-x}{H}\right) g(x)\ dx.$$
From the support of $g$ and $\varphi$, we see that $G$ is supported on the interval $[-10,10]$ (say).  

At this stage we could use the Fourier inversion formula
\begin{equation}\label{G-invert}
 G(u) = \frac{1}{2\pi} \int_\R \hat G\left(-\frac{t}{2\pi}\right) e^{-itu}\ dt
\end{equation}
to rewrite \eqref{ga} in terms of the Dirichlet series ${\mathcal D}[f](\frac{1}{2}+it) = \sum_n \frac{f(n)}{n^{\frac{1}{2}+it}}$.  However, the main term in the right-hand side of \eqref{spe-2} only involves ``medium'' values of the frequency variable $t$, in the sense that $|t|$ is constrained to lie between $\eta \beta X$ and $\beta X/\eta$.  Fortunately, the phase $e(\beta X e^u)$ of $G(u)$ oscillates at frequencies comparable to $\beta X$ in the support $[-10,10]$ of $G$, so the contribution of ``high frequencies'' $|t| \gg \beta X/\eta$ and ``low frequencies'' $|t| \ll \eta \beta X$ will both be acceptable, in the sense that they will be controllable using the error term in \eqref{spe-2}.

To make this precise we will use the harmonic analysis technique of Littlewood-Paley decomposition.  Namely, we split the sum \eqref{ga} into three subsums
\begin{equation}\label{gaj}
 \sum_n \frac{f(n)}{n^{1/2}} G_i(\log n - \log X) 
\end{equation}
for $i=1,2,3$, where $G_1,G_2,G_3$ are Littlewood-Paley projections of $G$,
\begin{align*}
G_1(u) &\coloneqq \int_\R G\left( u - \frac{2\pi v}{10 \eta |\beta| X} \right) \hat \varphi(v)\ dv \\
G_2(u) &\coloneqq \int_\R G\left( u - \frac{2\pi \eta v}{|\beta| X} \right) \hat \varphi(v)\ dv - \int_\R G\left( u - \frac{2\pi v}{10 \eta |\beta| X} \right) \hat \varphi(v)\ dv \\
G_3(u) &\coloneqq G(u) - \int_\R G\left( u - \frac{2\pi \eta v}{|\beta| X} \right) \hat \varphi(v)\ dv
\end{align*}
and estimate each subsum separately.  

\begin{remark}\label{fourier}  Expanding out $\hat \varphi$ as a Fourier integral and performing some change of variables, one can compute that
\begin{align*}
G_1(u) &= \frac{1}{2\pi} \int_\R \hat G\left(-\frac{t}{2\pi}\right) \varphi \left( \frac{t}{10 \eta |\beta| X} \right) e^{-itu}\ dt \\
G_2(u) &= \frac{1}{2\pi} \int_\R \hat G\left(-\frac{t}{2\pi}\right) \left( \varphi\left( \frac{\eta t}{|\beta| X} \right) - \varphi\left( \frac{t}{10 \eta |\beta| X} \right)\right) e^{-itu}\ dt \\
G_3(u) &= \frac{1}{2\pi} \int_\R \hat G\left(-\frac{t}{2\pi}\right) \left( 1 - \varphi\left( \frac{\eta t}{|\beta| X} \right)\right) e^{-itu}\ dt.
\end{align*}
Comparing this with \eqref{G-invert}, we see that $G_1,G_2,G_3$ arise from $G$ by smoothly truncating the frequency variable $t$ to ``low frequencies'' $|t| \ll \eta |\beta| X$, ``medium frequencies'' $\eta |\beta| X \ll |t| \ll |\beta| X / \eta$, and ``high frequencies'' $|t| \gg |\beta| X / \eta$ respectively.  It will be the medium frequency term $G_2$ that will be the main term; the low frequency term $G_1$ and high frequency term $G_3$ can be shown to be small by using the oscillation properties of the phase $e(\beta X e^u)$.
\end{remark}

We first consider the contribution of \eqref{gaj} in the ``high frequency'' case $i=3$.  Since $\int_\R \hat \varphi(v)\ dv = \varphi(0) = 1$, we can use the fundamental theorem of calculus to write
\begin{equation}\label{g3-d}
 G_3(u) = \frac{2\pi \eta}{|\beta| X} \int_0^1 \int_\R v G'\left( u - a \frac{2\pi \eta v}{|\beta| X} \right) \hat \varphi(v)\ dv da.
\end{equation}
For $x \in [X/2,4X]$, the function $u \mapsto X^{1/2} e^{u/2} e(\beta X e^u) \varphi\left(\frac{Xe^u-x}{H}\right)$ is only non-zero when $x = X e^u + O(H)$, and has a derivative of $O( |\beta| X^{3/2} )$.  As a consequence, we have the derivative bound
$$ G'(u) \ll \beta X^{3/2} \int_{x = X e^u + O(H)} |g(x)|\ dx$$
for any $u$, and hence by the triangle inequality
$$ G_3(u) \ll \eta X^{1/2} \int_0^1 \int_\R \int_{x = X e^{-a \frac{2\pi \eta v}{|\beta| X}} e^u + O(H)} |g(x)| |v| |\hat \varphi(v)|\ dx dv da.$$
The expression \eqref{gaj} when $i=3$ may thus be bounded by
$$ \ll \eta X^{1/2} \int_0^1 \int_\R \int_\R \sum_{n: x = \lambda n + O(H)} \frac{|f(n)|}{n^{1/2}}
|g(x)| |v| |\hat \varphi(v)|\ dx dv da$$
where we abbreviate $\lambda \coloneqq e^{-a \frac{2\pi \eta v}{|\beta| X}}$.  From the support of $f$ and $g$ we see that the inner integral vanishes unless $\lambda \asymp 1$.  By the rapid decrease of $\hat \varphi$, we may then bound the previous expression by
$$ \ll \eta X^{1/2} \sup_{\lambda \asymp 1} \int_\R \sum_{n: x = \lambda n + O(H)} \frac{|f(n)|}{n^{1/2}}
|g(x)|\ dx.$$
Since $f$ is supported on $[X,2X]$, we see from \eqref{g-norm} and Cauchy-Schwarz that this quantity is bounded by
$$ \ll \eta \sup_{\lambda \asymp 1} \left(\int_\R \left(\sum_{n: x = \lambda n + O(H)} |f(n)|\right)^2\ dx\right)^{1/2}.$$
Rescaling $x$ by $\lambda$ and using the triangle inequality, we can bound this by
$$ \ll \eta \left(\int_\R (\sum_{x \leq n \leq x+H} |f(n)|)^2\ dx\right)^{1/2}$$
which is acceptable.

Now we consider the contribution of \eqref{gaj} in the ``low frequency'' case $i=1$.  We first make the change of variables $w \coloneqq u - \frac{2\pi v}{10 \eta |\beta| X}$ to write
\begin{align*}
G_1(u) &= \frac{-10 \eta |\beta| X}{2\pi} \int_\R G(w) \hat \varphi\left(\frac{10 \eta |\beta| X}{2\pi} (u-w)\right)\ dw \\
&= \frac{-10 \eta |\beta| X^{3/2}}{2\pi} \int_\R \left(\int_\R e( \beta X e^w ) \psi_{x,u}(w)\ dw\right)  g(x)\ dx
\end{align*}
where $\psi_{x,u} \colon \R \to \C$ is the amplitude function
\begin{equation}\label{psixu-def}
\psi_{x,u}(w) \coloneqq e^{w/2} \hat \varphi\left(\frac{10 \eta |\beta| X}{2\pi} (u-w)\right) \varphi\left( \frac{Xe^w - x}{H} \right).
\end{equation}
The function $\psi_{x,u}$ is supported on the region $w = \log \frac{x}{X} + O(\frac{H}{X})$ (in particular, $w = O(1)$); from the rapid decrease of $\hat \varphi$ and the hypothesis $|\beta| \gg 1/H$ we also have the bound
$$ \psi_{x,u}(w) \ll (1 + \eta |\beta| X |w-u|)^{-2}$$
(say).  Differentiating \eqref{psixu-def} in $w$, we conclude the bounds
$$ \psi'_{x,u}(w) \ll \left(\eta |\beta| X + \frac{X}{H}\right) (1 + \eta |\beta| X |w-u|)^{-2}.$$
Meanwhile, the phase $\beta X e^w$ has all $w$-derivatives comparable to $|\beta| X$ in magnitude.  Integrating by parts, we conclude the bound
\begin{align*}
\int_\R e( \beta X e^w ) \psi_{x,u}(w)\ dw &\ll \left( \eta + \frac{1}{|\beta| H} \right)
\int_{w = \log \frac{x}{X} + O(\frac{H}{X})} (1 + \eta |\beta| X |w-u|)^{-2}\ dw
\end{align*}
and hence \eqref{gaj} for $i=1$ may be bounded by
$$ \ll \left( \eta + \frac{1}{|\beta| H} \right) \eta |\beta| X^{3/2}
\int_\R \int_{w = \log \frac{x}{X} + O(\frac{H}{X})} \sum_n \frac{|f(n)|}{n^{1/2}} \frac{|g(x)|}{(1 + \eta |\beta| X \cdot |w-\log n+\log X|)^{2}}\ dw dx.$$
Making the change of variables $z \coloneqq w - \log n + \log X$, this becomes
$$ \ll \left( \eta + \frac{1}{|\beta| H} \right) \eta |\beta| X^{3/2}
\int_\R \int_\R \sum_{n: z = \log \frac{x}{n} + O(\frac{H}{X})}  \frac{|f(n)|}{n^{1/2}} \frac{|g(x)|}{(1 + \eta |\beta| X |z|)^{2}}\ dx dz.$$
The sum vanishes unless $z = O(1)$, in which case the condition $z = \log \frac{x}{n} + O(\frac{H}{X})$ can be rewritten as $n = e^{-z} x + O(H)$.  Since $\int_\R \eta |\beta| X (1 + \eta |\beta| X |z|)^{-2}\ dz \ll 1$, we can thus bound the previous expression by
$$ \ll \left( \eta + \frac{1}{|\beta| H} \right) X^{1/2} \sup_{z = O(1)} \int_\R \sum_{n = e^{-z} x + O(H)} \frac{|f(n)|}{n^{1/2}}  |g(x)|\ dx.$$
Arguing as in the high frequency case $i=3$ (with $e^{-z}$ now playing the role of $\lambda$), we can bound this by
$$ \left( \eta + \frac{1}{|\beta| H} \right) \left(\int_\R \left(\sum_{x \leq n \leq x+H} |f(n)|\right)^2\ dx\right)^{1/2}$$
which is acceptable.

Finally we consider the main term, which is \eqref{gaj} in the ``medium frequency'' case $i=2$.  For any $T>0$, the quantity
$$
\sum_n \frac{f(n)}{n^{1/2}} \int_\R G\left( \log n - \log X - \frac{2\pi v}{T} \right) \hat \varphi(v)\ dv $$
can be expanded by first opening $\widehat{\varphi}(v) = \int_{\mathbb{R}} \varphi(y) e(- v y) dy$ and then using the change of variables $t \coloneqq Ty$, $w \coloneqq \log n - \log X - \frac{2\pi v}{T}$ as
\begin{align*}
&\int_\R \sum_n \frac{f(n)}{n^{1/2}} \int_\R G\left( \log n - \log X - \frac{2\pi v}{T} \right) e(-vy) \varphi(y)\ dv dy \\
&\quad = 
\frac{1}{2\pi} \int_\R \sum_n \frac{f(n)}{n^{1/2}} \int_\R G( w ) n^{-it} e^{itw} X^{it} \varphi\left(\frac{t}{T}\right)\ dw dt \\
&\quad = 
\frac{1}{2\pi} \int_\R F(t) \int_\R G( w ) e^{itw} X^{it} \varphi\left(\frac{t}{T}\right)\ dw dt
\end{align*}
(compare with Remark \ref{fourier}).  Applying identity for $T \coloneqq \frac{|\beta| X}{\eta}$ and $T \coloneqq 10 \eta |\beta| X$ and subtracting, we may thus write \eqref{gaj} for $i=2$ as
\begin{equation}\label{g2-expand}
\int_\R \int_\R \tilde F(t) G(w) e^{itw}\ dw dt
\end{equation}
where $\tilde F$ is the function
$$ \tilde F(t) \coloneqq \frac{1}{2\pi} F(t) X^{it} \left(\varphi\left(\frac{\eta t}{|\beta| X}\right) - \varphi\left( \frac{t}{10 \eta |\beta| X}\right) \right).
$$
For future reference we observe that $\tilde F$ is supported on $I$ and enjoys the pointwise bound $\tilde F(t) = O( |F(t)| )$.  Expanding out $G$, we can write the preceding expression as
$$ X^{1/2} \int_\R \int_\R \tilde F(t) g(x) J_x(t)\ dt dx$$
where $J_x(t)$ is the oscillatory integral
\begin{equation}\label{jxt-def}
 J_x(t) \coloneqq \int_\R e(\phi_t(w)) a_x(w)\ dw
\end{equation}
with the phase function
$$ \phi_t(w) \coloneqq \beta X e^w + \frac{tw}{2\pi}$$
and the amplitude function
$$ a_x(w) \coloneqq  e^{w/2} \varphi\left(\frac{Xe^w-x}{H}\right) \varphi(w/100),$$
noting that $\varphi(w/100)$ will equal $1$ whenever $g(x)\varphi\left(\frac{Xe^w-x}{H}\right)$ is non-zero.  
By \eqref{g-norm} and Cauchy-Schwarz, the above expression may be bounded in magnitude by
$$ X^{1/2} \left( \int_\R \int_\R \tilde F(t) \overline{\tilde F(t')} \int_\R J_x(t) \overline{J_x(t')}\ dx dt dt' \right)^{1/2},$$
so by the triangle inequality and the pointwise bounds on $\tilde F$ it will suffice to establish the bound
\begin{equation}\label{ffjj} \int_I \int_I |F(t)| |F(t')| \left|\int_\R J_x(t) \overline{J_x(t')}\ dx\right| dt dt' \ll \frac{1}{|\beta|^2 X} \int_I \left(\int_{t-|\beta| H}^{t+|\beta|H} |F(t')|\ dt'\right)^2\ dt.
\end{equation}
We shall shortly establish the bound
\begin{equation}\label{jax}
\int_\R J_x(t) \overline{J_x(t')}\ dx \ll \frac{H}{|\beta| X \left(1 + \frac{|t-t'|}{|\beta| H}\right)^2}.
\end{equation}
Assuming this bound, we can bound the left-hand side of \eqref{ffjj} by
$$ \ll \frac{1}{(|\beta| H)^2} \int_0^{2|\beta| X/\eta} \int_I A(t) A(t+h) \frac{H}{|\beta| X \left(1 + \frac{h}{|\beta| H}\right)^2} \ dt dh $$
where $A(t) \coloneqq \int_{t-|\beta| H}^{t+|\beta| H} |F(t')| \ dt'$, and from Schur's test (see \cite[Theorem 5.2]{Halmos}) we conclude that this contribution is acceptable.  

It remains to obtain \eqref{jax}.  We first consider the regime where $|t-t'| = O( |\beta| H )$.  By Cauchy-Schwarz, this bound will follow if we can obtain the bound
\begin{equation}\label{jxxt}
\int_\R |J_x(t)|^2\ dx \ll \frac{H}{|\beta| X}
\end{equation}
for all $t \in I$.  To establish this bound, we divide into the cases $\beta H^2 \geq X$ and $\beta H^2 < X$.  First suppose that $\beta H^2 \geq X$.  Then one has $\phi''_t(w) \asymp |\beta| X$ on the support of $a_x$, and the cutoff $a_x$ has total variation $O(1)$.  Hence by van der Corput estimates (see e.g. \cite[Lemma 8.10]{ik}) we have the bound
$$ J_x(t) \ll (|\beta| X)^{-1/2}.$$
Furthermore, if $|\frac{t}{2\pi} + \beta x| \geq C |\beta| H$ for a large constant $C$, then on the support of $a_x$, then one has $\phi'_t(w) \asymp |\frac{t}{2\pi} + \beta x|$, so that $1/\phi'_t(w)$ is of size $O( \frac{1}{|\frac{t}{2\pi} + \beta x|} )$.  A calculation then shows that the $j^{\mathrm{th}}$ derivative of $1/\phi'_t(w)$ is of size $O( \frac{(X/H)^j}{|\frac{t}{2\pi} + \beta x|} )$ for $j=0,1,2$.  Similarly, the $j^{\mathrm{th}}$ derivative of $a_x$ has an $L^1$ norm of $O( (X/H)^{j-1} )$ for $j=0,1,2$.  Applying two integrations by parts, we then obtain the bound
\begin{equation}\label{jxb}
 J_x(t) \ll \frac{X/H}{|\frac{t}{2\pi} + \beta x|^2}
\end{equation}
in this regime.  Combining these bounds we obtain \eqref{jxxt} in the case $\beta H^2 \geq X$ after some calculation.

Now suppose that $\beta H^2 < X$.  On the one hand, from the triangle inequality we have the bound
$$ J_x(t) \ll \frac{H}{X}.$$
On the other hand, if $|\frac{t}{2\pi} + \beta x| \geq C \frac{X}{H}$ for a large constant $C$, then on the support of $a_x$, one can again calcualate that $j^{\mathrm{th}}$ derivative of $1/\phi'_t(w)$ is of size $O( \frac{(X/H)^j}{|\frac{t}{2\pi} + \beta x|} )$ for $j=0,1,2$, and that $j^{\mathrm{th}}$ derivative of $a_x$ has an $L^1$ norm of $O( (X/H)^{j-1} )$.  This again gives the bound \eqref{jxb} after two integrations by parts.  
Combining these bounds we obtain \eqref{jxxt} in the case $\beta H^2 < X$ after some calculation.

It remains to treat the case when $|t-t'| > C |\beta| H$ for some large constant $C>0$.  Here we write
$$
\int_\R J_x(t) \overline{J_x(t')}\ dx = H \int_\R \int_\R e( \phi_t(w) - \phi_{t'}(w') ) \tilde a( w, w' )\ dw dw'
$$
where
$$ \tilde a(w,w') \coloneqq e^{w/2} e^{w'/2} \varphi_2\left(\frac{Xe^w-Xe^{w'}}{H}\right) \varphi(w/100) \varphi(w'/100)$$
and $\varphi_2$ is the convolution of $\varphi$ with itself, thus
$$ \varphi_2(x) \coloneqq \int_\R \varphi(y) \varphi(x+y)\ dy.$$
We make the change of variables $w' = w+h$ to then write
$$
\int_\R J_x(t) \overline{J_x(t')}\ dx = H \int_\R \int_\R e( \phi_t(w) - \phi_{t'}(w+h) ) \tilde a( w, w+h )\ dw dh.
$$
Observe that $\tilde a(w,w+h)$ vanishes unless $h = O(H/X)$, so we may restrict to this range.  If $|t-t'| > C |\beta| H$ for a sufficiently large $C$, a calculation then reveals that on the support of $\tilde a(w,w+h)$, the $w$-derivative of $\phi_t(w) - \phi_{t'}(w+h)$ has magnitude comparable to $|t'-t|$, and that the $j^{\mathrm{th}}$ $w$-derivative is of size $O( |\beta| H ) = O( |t'-t| )$ for $j=2,3$.  Furthermore, the $j^{\mathrm{th}}$ $w$-derivative of $\tilde a(w,w+h)$ is of size $O(1)$ for $j=0,1,2$.  From two integrations by parts we conclude that
$$ \int_\R e( \phi_t(w) - \phi_{t'}(w+h) ) \tilde a( w, w+h )\ dw  \ll \frac{1}{|t'-t|^2}$$
for all $h = O(H/X)$, and hence
$$
\int_\R J_x(t) \overline{J_x(t')}\ dx \ll \frac{H^2}{X |t-t'|^2}.$$
This gives \eqref{jax} (with some room to spare), since $|\beta| H \gg 1$.  This concludes the proof of \eqref{spe-2}.

Now we prove \eqref{spe-4}.  Again, we use duality.  It suffices to show that
\begin{equation}
\label{eq:spe-4dual}
\int_\R S_f(\theta) g(\theta)\ d\theta \ll
\left(\int_I |F(t)|^2\ dt\right)^{1/2}  + \frac{\eta + \frac{1}{|\beta| X}}{H} \left(\int_\R \left (\sum_{x \leq n \leq x+H} |f(n)|\right )^2\ dx\right)^{1/2}
\end{equation}
whenever $g \colon \R \to \C$ is a measurable function supported on $\{ \theta: \beta \leq |\theta| \leq 2\beta \}$ with the normalization
\begin{equation}\label{gnorm}
 \int_\R |g(\theta)|^2\ d\theta = 1.
\end{equation}
The expression $\int_\R S_f(\theta) g(\theta)\ d\theta$ can be rearranged as
$$ \sum_n \frac{f(n)}{n^{1/2}} G(\log n - \log X)$$
where
\begin{equation}\label{gdef}
 G(u) \coloneqq \varphi( u / 10 ) X^{1/2} e^{u/2} \int_\R g(\theta) e(X e^u \theta)\ d\theta,
\end{equation}
noting that the cutoff $\varphi(u/10)$ will equal $1$ for $n \in [X,2X]$.  We again split this sum as the sum of three subsums \eqref{gaj} with $i=1,2,3$, where $G_1,G_2,G_3$ are defined as before.

We first control the sum \eqref{gaj} in the ``high frequency'' case $i=3$.  By \eqref{g3-d}, the triangle inequality, and the rapid decay of $\hat \varphi$, we may bound this sum by
$$ \ll \frac{\eta}{|\beta| X} \sup_{a,v} \sum_n \frac{|f(n)|}{n^{1/2}} \left|G'\left( \log n - \log X - a \frac{2\pi \eta v}{|\beta| X} \right)\right|.$$
Because of the cutoff $\varphi(u/10)$ in \eqref{gdef}, the sum vanishes unless $a \frac{2\pi \eta v}{|\beta| X} = O(1)$, so we may bound the preceding expression by
$$ \ll \frac{\eta}{|\beta| X} \sum_n \frac{|f(n)|}{n^{1/2}} |G'( \log n - \log X' )|$$
for some $X' \asymp X$.  Computing the derivative of $G$, we may bound this in turn by
$$ \ll \frac{\eta}{|\beta| X} \sum_n |f(n)| \left( \left|\int_\R g(\theta) e\left( \frac{X}{X'} n \theta \right)\ d\theta\right| + 
\left|\int_\R X \theta g(\theta) e\left( \frac{X}{X'} n \theta \right)\ d\theta\right| \right).$$
We shall just treat the second term here, as the first term is estimated analogously (with significantly better bounds).  We write this contribution as
$$ \ll \eta \sum_n |f(n)| \left|\int_\R \frac{\theta}{\beta} g(\theta) e\left( \frac{X}{X'} n \theta \right)\ d\theta\right|.$$
By partitioning the support $[X,2X]$ of $f$ into intervals of length $H$, and selecting on each such interval a number $n$ that maximizes the quantity $|\int_\R \frac{\theta}{\beta} g(\theta) e( \frac{X}{X'} n \theta )\ d\theta|$, we may bound this by
$$ \ll \eta \sum_{j=1}^J \left(\sum_{|n-n_j| \leq H} |f(n)|\right) \left|\int_\R \frac{\theta}{\beta} g(\theta) e\left( \frac{X}{X'} n_j \theta \right)\ d\theta\right|$$
for some $H$-separated subset $n_1,\dots,n_J$ of $[X,2X]$.  By \eqref{gnorm}, the support of $g$ (which in particular makes the factor $\frac{\theta}{\beta}$ bounded), the choice $\beta = 1/H$ and the large sieve inequality (e.g. the dual of \cite[Corollary 3]{monts}), we have
$$ \sum_{j=1}^J \left|\int_\R \frac{\theta}{\beta} g(\theta) e\left( \frac{X}{X'} n_j \theta \right)\ d\theta\right|^2 \ll \beta $$
and so by Cauchy-Schwarz, one can bound the preceding expression by
$$ \ll \eta \beta^{1/2} \left( \sum_{j=1}^J \left(\sum_{|n-n_j| \leq H} |f(n)|\right)^2 \right)^{1/2}.$$
But one has
$$ \left(\sum_{|n-n_j| \leq H} |f(n)|\right)^2 \ll \frac{1}{H} \int_{|x-n_j| \leq 2H} \left(\sum_{x \leq n \leq x+H} |f(n)|\right)^2\ dx $$
for each $j$, and so as $\beta = 1/H$ and the $n_j$ are $H$-separated, the high frequency case $i=3$ of \eqref{gaj} contributes to~\eqref{eq:spe-4dual}
$$ \ll \frac{\eta}{H} \left(\int_\R \left(\sum_{x \leq n \leq x+H} |f(n)|\right)^2\ dx\right)^{1/2}$$
which is acceptable.

Now we control the ``low frequency'' case $i=1$ of \eqref{gaj}.  Using the change of variables $h \coloneqq \frac{2\pi v}{10 \eta |\beta| X}$, we can write this as
$$
\frac{10 \eta |\beta| X}{2\pi} \int_\R \int_\R \sum_n f(n) \varphi\left(\frac{\log n - \log X - h}{10}\right) e^{-h/2} g(\theta) e( n e^{-h} \theta ) \hat \varphi\left(\frac{10 \eta |\beta| X h}{2\pi}\right)\ d\theta dh.$$
The integrand vanishes unless $h = O(1)$.  Writing 
$$ e(n e^{-h} \theta) = \frac{-1}{2\pi i n e^{-h} \theta} \frac{d}{dh} e(n e^{-h} \theta) $$
and then integrating by parts in the $h$ variable, we can write this expression as
$$
\frac{10 \eta |\beta| X}{4\pi^2 i} \int_\R \int_\R \sum_n \frac{f(n) g(\theta)e( n e^{-h} \theta )}{n \theta}  \frac{d}{dh} R_{n,\theta}(h)\ d\theta dh$$
where $R_{n,\theta}(h)$ is the quantity
$$ R_{n,\theta}(h) \coloneqq e^{h/2} \varphi\left(\frac{\log n - \log X - h}{10}\right) \hat \varphi\left(\frac{10 \eta |\beta| X h}{2\pi}\right).$$
By the Leibniz rule, and the fact that $\widehat{\phi}$ is a Schwartz function we see that $\frac{d}{dh} R_{n,\theta}(h)$ is supported on the region $h = O(1)$ with
$$ \int_\R \left|\frac{d}{dh} R_{n,\theta}(h)\right|\ dh \ll 1.$$
Thus we may bound the preceding expression using the triangle inequality and pigeonhole principle by
$$ \ll \eta \sum_n |f(n)| \left| \int_\R \frac{\beta}{\theta} g(\theta) e(n e^h \theta) d\theta \right |$$
for some $h=O(1)$.  But by the same large sieve inequality arguments used to control the high frequency case $i=3$ (with $e^h$ now playing the role of $\frac{X}{X'}$), we see that this contribution is acceptable.  This concludes the treatment of the low frequency case $i=1$.

Finally we consider the main term, which is the ``medium frequency'' case $i=2$ of \eqref{gaj}.  As in the proof of \eqref{spe-2}, we may bound this expression by \eqref{g2-expand}. By Cauchy-Schwarz and the Plancherel identity, one may bound this by
$$
\ll \left(\int_I |F(t)|^2\ dt\right)^{1/2} \left(\int_\R |G(w)|^2\ dw\right)^{1/2}.$$
By \eqref{gdef} and the change of variables $y \coloneqq X e^w$, we have
$$
\int_\R |G(w)|^2\ dw \ll \int_\R \left|\int_\R g(\theta) e(y\theta)\ d\theta\right|^2\ dy,$$
and by \eqref{gnorm} and the Plancherel identity again, the right-hand side is equal to $1$. Hence the $i=2$ case also gives an acceptable contribution to~\eqref{eq:spe-4dual}, and the claim \eqref{spe-4} follows.
\end{proof}

We can now use Lemma \ref{edc} to obtain an estimate for general $q$:

\begin{corollary}[Stationary phase estimate, minor arc case]\label{spem}  Let $1 \leq H \leq X$, and let $f \colon \N \to \C$ be a function supported on $(X,2X]$.  Let $q \geq 1$, let $a$ be coprime to $q$, and let $\beta, \eta$ be real numbers with $|\beta| \ll \eta \ll 1$.  Let $I$ denote the region in \eqref{I-def}.  Then we have
\begin{align*}
\int_{\beta-1/H}^{\beta+1/H} \left|S_f\left(\frac{a}{q}+\theta\right)\right|^2\ d\theta  &\ll\\
\frac{d_2(q)^4}{|\beta|^2 H^2 q} &\sup_{q=q_0q_1} \int_I \left(\sum_{\chi\ (q_1)} \int_{t-|\beta| H}^{t+|\beta| H} \left|{\mathcal D}[f]\left(\frac{1}{2}+it',\chi,q_0\right)\right|\ dt'\right)^2\ dt  \\
&\quad + \frac{\left(\eta + \frac{1}{|\beta| H}\right)^2 }{H^2} \int_\R \left(\sum_{x \leq n \leq x+H} |f(n)|\right)^2\ dx.
\end{align*}
If $\beta = 1/H$, one has the variant
\begin{align*}
\int_{\beta \leq |\theta| \leq 2\beta} \left|S_f\left(\frac{a}{q} + \theta\right)\right|^2\ d\theta  &\ll  \frac{d_2(q)^4}{q} 
\sup_{q=q_0 q_1} \int_I \left(\sum_{\chi\ (q_1)}  \left|{\mathcal D}[f]\left(\frac{1}{2}+it, \chi, q_0\right)\right|\right)^2\ dt  \\
&\quad + \frac{\left(\eta + \frac{1}{|\beta| X}\right)^2 }{H^2} \int_\R \left(\sum_{x \leq n \leq x+H} |f(n)|\right)^2\ dx.
\end{align*}
\end{corollary}

The factor $d_2(q)^4$ might be improvable, but is already negligible in our analysis, so we do not attempt to optimize it.  The presence of the $q_0$ variable is technical; the most important case is when $q_0=1$ and $q_1=q$, so the reader may wish to restrict to this case for a first reading.  It will be important that there are no $q$ factors in the error terms on the right-hand side; this is possible because we estimate the left-hand side in terms of Dirichlet series at moderate values of $t$ \emph{before} decomposing into Dirichlet characters.

\begin{proof} We just prove the first estimate, as the second is similar.
By applying Proposition \ref{spe} with $f$ replaced by $f e(a\cdot /q)$, we obtain the bound
\begin{align*}
\int_{\beta - 1/H}^{\beta + 1/H} |S_f(a/q + \theta)|^2 d \theta 
&\ll \frac{1}{|\beta|^2 H^2} \int_I \left(\int_{t-|\beta| H}^{t+|\beta| H} \left|{\mathcal D}[f e(a\cdot/q)]\left(\frac{1}{2}+it'\right)\right|\ dt'\right)^2\ dt  \\
&\quad + \frac{\left(\eta + \frac{1}{|\beta| H}\right)^2 }{H^2} \int_\R \left(\sum_{x \leq n \leq x+H} |f(n)|\right)^2\ dx
\end{align*}
From Lemma \ref{edc} we have
$$ \int_{t-|\beta| H}^{t+|\beta| H} \left|{\mathcal D}[fe(a\cdot/q)]\left(\frac{1}{2}+it'\right)\right|\ dt' \leq \frac{d_2(q)}{\sqrt{q}} \sum_{q = q_0 q_1} \sum_{\chi\ (q_1)} \int_{t-|\beta| H}^{t+|\beta| H} \left|{\mathcal D}[f]\left(\frac{1}{2}+it',\chi,q_0\right)\right|$$
and hence by Cauchy-Schwarz
\begin{align*}
& \left(\int_{t-|\beta| H}^{t+|\beta| H} \left|{\mathcal D}[fe(a\cdot/q)]\left(\frac{1}{2}+it'\right)\right|\ dt'\right)^2 \\
&\quad \leq \frac{d_2(q)^3}{q} \sum_{q = q_0 q_1} \left(\sum_{\chi\ (q_1)} \int_{t-|\beta| H}^{t+|\beta| H} \left|{\mathcal D}[f]\left(\frac{1}{2}+it',\chi,q_0\right)\right|\right)^2
\end{align*}
Inserting this bound and bounding the summands in the $q=q_0q_1$ summation by their supremum yields the claim.
\end{proof}

If the function $f$ in the above corollary is $k$-divisor-bounded, then by Cauchy-Schwarz and \eqref{divisor-crude} we have
\begin{align*}
\frac{1}{H^2} \int_\R \left(\sum_{x \leq n \leq x+H} |f(n)|\right)^2\ dx &\ll
\frac{1}{H} \int_\R \sum_{x \leq n \leq x+H} |f(n)|^2\ dx \\
&\ll \sum_n |f(n)|^2\ dx \\
&\ll_k X \log^{O_k(1)} X.
\end{align*}
Applying the above corollary  with $\eta \coloneqq  Q^{-1/2} = \log^{-B/2} X$, Proposition \ref{minor} is now an immediate consequence of the following mean value estimates for Dirichlet series.

\begin{proposition}[Mean value estimate]\label{mve}   Let $\eps > 0$ be a sufficiently small constant, and let $A>0$.  Let $k \geq 2$ be fixed, let $B > 0$ be sufficiently large depending on $k,A$, and let $X \geq 2$.  Set
\begin{equation}\label{H-def}
H \coloneqq X^{\sigma+\eps}
\end{equation}
and
\begin{equation}\label{Q-def}
Q \coloneqq \log^B X.
\end{equation}
Let $1 \leq q \leq Q$, and suppose that $q = q_0 q_1$.  Let $\lambda$ be a positive quantity such that 
\begin{equation}\label{l6}
X^{-1/6-\eps} \leq \lambda \ll \frac{1}{qQ}.  
\end{equation}
Let $f \colon \N \to \C$ be either the function $f \coloneqq \Lambda 1_{(X,2X]}$ or $f \coloneqq d_k 1_{(X,2X]}$.  
Then we have
\begin{equation}\label{mve-est}
 \int_{Q^{-1/2} \lambda X \ll |t| \ll Q^{1/2} \lambda X} \left(\sum_{\chi\ (q_1)} \int_{t-\lambda H}^{t+\lambda H} \left|{\mathcal D}[f]\left(\frac{1}{2}+it',\chi,q_0\right)\right|\ dt'\right)^2\ dt \ll_{k,\eps, A, B} q \lambda^2 H^2 X \log^{-A} X.
\end{equation}
\end{proposition}

\begin{proposition}\label{harm-prop}  Let $\eps > 0$ be a sufficiently small constant, and let $A, B > 0$.  Let $k \geq 2$ be fixed, let $X \geq 2$, and suppose that $q_0, q_1$ are natural numbers with $q_0, q_1 \leq \log^B X$.  Let $f \colon \N \to \C$ be either the function $f \coloneqq \Lambda 1_{(X,2X]}$ or $f \coloneqq d_k 1_{(X,2X]}$. Let $B'$ be sufficiently large depending on $k, A, B$.  Then one has
\begin{equation}\label{drip0}
 \int_{\log^{B'} X \ll |t'| \ll X^{5/6 - \eps} } \left( \sum_{\chi\ (q_1)} \left|{\mathcal D}[f]\left(\frac{1}{2}+it',\chi,q_0\right)\right|\right)^2\ dt' \ll_{k,\eps,A,B,B'} q X \log^{-A} X.
\end{equation}
\end{proposition}

Proposition \ref{harm-prop}  is comparable\footnote{See \cite[Lemma 9.3]{harman} for a precise connection between $L^2$ mean value theorems such as \eqref{drip0} and estimates for sums of $f \chi$ on short intervals.} in strength to the prime number theorem (in arithmetic progressions) in almost all intervals of the form $[x, x+x^{1/6+\eps}]$.  A popular approach to proving such theorems is via zero density estimates (see e.g. \cite[\S 10.5]{ik}); this works well in the case $f = \Lambda 1_{(X,2X]}$, but is not as suitable for treating the case $f = d_k 1_{(X,2X]}$.  We will instead adapt a slightly different approach from \cite{harman} using combinatorial decompositions and mean value theorems and large value theorems for Dirichlet polynomials; the details of the argument will be given in Appendix \ref{harman-sec}.  In the case $f = d_3 1_{(X,2X]}$, an estimate closely related to Proposition \ref{harm-prop} was established in \cite[Theorem 1.1]{bbmz}, relying primarily on sixth moment estimates for the Riemann zeta function.

It remains to prove Proposition \ref{mve}.  This will be done in the remaining sections of the paper.  

\section{Combinatorial decompositions}\label{hb-sec}

Let $\eps, k, A, B, H, X, Q, q_0, q_1, q, \lambda, f$ be as in Proposition \ref{mve}. 
 We may assume without loss of generality that $\eps$ is small, say $\eps < 1/100$; we may also assume that $X$ is sufficiently large depending on $k,\eps$.  

%From the bounds on $\eta,Q, \lambda$ we record the obvious estimates
%\begin{equation}\label{win}
 %X^{-c^2 \eps}, Q^{-c}, (\lambda H)^{-c} \ll \eta.
%\end{equation}
%As a practical matter, \eqref{win} means that we will ``win'' as soon as we save some power of $X^\eps$, $Q$, or $\lambda H$ over the trivial bound.

We first invoke Lemma \ref{comb-decomp-2} with $m=5$, and with $\eps$ and $H_0$ replaced by $\eps^2$ and $X^{-\eps^2} H$ respectively; this choice of $m$ is available thanks to \eqref{sigma-range}.  We conclude that the function $(t,\chi) \mapsto {\mathcal D}[f](\frac{1}{2}+it, \chi, q_0)$ can be decomposed as a linear combination (with coefficients of size $O_{k,\eps}( d_2(q_0)^{O_{k,\eps}(1)} )$) of $O_{k,\eps}( \log^{O_{k,\eps}(1)} X)$ functions of the form $(t,\chi) \mapsto {\mathcal D}[\tilde f]\left(\frac{1}{2}+it,\chi\right)$, where $\tilde f \colon \N \to \C$ is one of the following forms:

\begin{itemize}
\item[(Type $d_1$, $d_2$, $d_3$, $d_4$ sums)]  A function of the form
\begin{equation}\label{fst-yetagain}
 \tilde f = (\alpha \ast \beta_1 \ast \dots \ast \beta_j) 1_{(X/q_0,2X/q_0]}
\end{equation}
for some arithmetic functions $\alpha,\beta_1,\dots,\beta_j \colon \N \to \C$, where $j=1,2,3,4$, $\alpha$ is $O_{k,\eps}(1)$-divisor-bounded and supported on $[N,2N]$, and each $\beta_i$, $i=1,\dots,j$ is either of the form $\beta_i = 1_{(M_i,2M_i]}$ or $\beta_i = L 1_{(M_i,2M_i]}$ for some $N, M_1,\dots,M_j$ obeying the bounds 
$$ 1 \ll N \ll_{k,\eps} X^{\eps^2},$$
$$ N M_1 \dots M_j \asymp_{k,\eps} X/q_0,$$
and
$$ X^{-\eps^2} H \ll M_1 \ll \dots \ll M_j \ll X/q_0.$$ 
\item[(Type II sum)]  A function of the form
$$ \tilde f = (\alpha \ast \beta) 1_{(X/q_0,2X/q_0]}$$
for some $O_{k,\eps}(1)$-divisor-bounded arithmetic functions $\alpha,\beta \colon \N \to \C$ supported on $[N,2N]$ and $[M,2M]$ respectively, for some $N,M$ obeying the bounds 
$$ X^{\eps^2} \ll N \ll X^{-\eps^2} H$$
and
$$NM \asymp_{k,\eps} X / q_0.$$
\item[(Small sum)]  A function $\tilde f$ supported on $(X/q_0,2X/q_0]$ obeying the bound
\begin{equation}\label{smallsum}
 \|\tilde f \|_{\ell^2}^2 \ll_{k,\eps} X^{1-\eps^2/8}.
\end{equation}
\end{itemize}

We have omitted the conclusion of good cancellation in the Type II case as it is not required in the regime $\lambda \gg X^{-1/6-\eps}$ under consideration.

By the triangle inequality, it thus suffices to show that for $\tilde f$ being a sum of one of the above forms, that we have the bound
\begin{equation}\label{trivb}
\begin{split}
&\int_{Q^{-1/2} \lambda X \ll |t| \ll Q^{1/2} \lambda X} \left(\sum_{\chi\ (q_1)} \int_{t-\lambda H}^{t+\lambda H} \left|{\mathcal D}[\tilde f]\left(\frac{1}{2}+it',\chi\right)\right|\ dt'\right)^2\ dt \\
&\quad \ll_{k,\eps,A,B} d_2(q_1)^{O_k(1)} q_1 \lambda^2 X H^2 \log^{-A} X
\end{split}
\end{equation}
(noting from \eqref{alpha-infty} that the factors of $d_2(q_1)^{O_k(1)}$ can be easily absorbed into the $\log^{-A} X$ factor after increasing $A$ slightly).  

We can easily dispose of the small case.  From Cauchy-Schwarz one has
$$
\left(\sum_{\chi\ (q_1)} \int_{t-\lambda H}^{t+\lambda H} \left|{\mathcal D}[\tilde f]\left(\frac{1}{2}+it',\chi\right)\right|\ dt'\right)^2 \ll q_1 \lambda H
\sum_{\chi\ (q_1)} \int_{t-\lambda H}^{t+\lambda H} \left|{\mathcal D}[\tilde f]\left(\frac{1}{2}+it',\chi\right)\right|^2\ dt'$$
and hence after interchanging the integrals, the left-hand side of \eqref{trivb} can be bounded by
$$ \ll q_1 \lambda^2 H^2 \sum_{\chi\ (q_1)} \int_{|t| \ll Q^{1/2} \lambda X} \left|{\mathcal D}[\tilde f]\left(\frac{1}{2}+it,\chi\right)\right|^2\ dt.$$
Using Lemma \ref{mvt-lem-ch} we can bound this by
$$ \ll_{k,\eps} q_1 \lambda^2 H^2 \frac{X / q_0 + q_1 Q^{1/2} \lambda X}{X/q_0} \|\tilde f\|_{\ell^2}^2 \log^3 X.$$
Crudely bounding $q_0, q_1, Q, \lambda \leq \log^B X$, the claim \eqref{trivb} then follows in this case from \eqref{smallsum}.

It remains to consider $\tilde f$ that are of Type $d_1$, Type $d_2$, Type $d_3$, Type $d_4$, or Type II.  In all cases we can write $\tilde f = f' 1_{(X/q_0, 2X/q_0]}$, where $f'$ is a Dirichlet convolution of the form $\alpha \ast \beta_1 \dots \ast \beta_j$ (in the Type $d_j$ cases) or of the form $\alpha \ast \beta$ (in the Type II case).  It is now convenient to remove the $1_{(X/q_0,2X/q_0]}$ truncation.  Applying Corollary \ref{trunc-dir} with $T \coloneqq  \lambda X^{1-\varepsilon/10}$ and $f$ replaced by $\tilde f \chi$, and using the divisor bound \eqref{alpha-infty} to control the supremum norm, we see that
$$ {\mathcal D}[\tilde f]\left(\frac{1}{2}+it,\chi\right) \ll_{\eps} \int_{|u| \leq \lambda X^{1-\varepsilon/10}} |F(t+u)| \frac{du}{1+|u|} 
+ \frac{X^{-1/2+\eps/5}}{\lambda}$$
where
$$ F(t) \coloneqq {\mathcal D}[\tilde f]\left(\frac{1}{2}+it,\chi\right).$$
We can thus bound the left-hand side of \eqref{trivb} by
\begin{align*}
&\ll_{\eps,k} \int_{Q^{-1/2} \lambda X \ll |t| \ll Q^{1/2} \lambda X} \left(\int_{|u| \leq \lambda X^{1-\varepsilon/10}} \sum_{\chi\ (q_1)} \int_{t+u-\lambda H}^{t+u+\lambda H} |F(t')|\ dt' \frac{du}{1+|u|}\right)^2\ dt \\
&\quad + \left(Q^{1/2} \lambda X \right) (q_1 \lambda H)^2 \left(\frac{X^{-1/2+\eps/5}}{\lambda}\right)^2.
\end{align*}
The second term can be written as
$$ Q^{1/2} \frac{X^{2\eps/5} q_1}{\lambda X} q_1 \lambda^2 X H^2;$$
since $\lambda \geq X^{-1/6-\varepsilon}$ and $q_1 \leq Q \leq (\log X)^B$, we see that this contribution to \eqref{trivb} is acceptable.  

Meanwhile, as $\frac{1}{1+|u|}$ has an integral of $O( \log X )$ on the region $|u| \leq \lambda X^{1-\varepsilon/10}$, we see from the Minkowski integral inequality in $L^2$ and on shifting $t$ by $u$ that
\begin{align*}
&\int_{Q^{-1/2} \lambda X \ll |t| \ll Q^{1/2} \lambda X} \left(\int_{|u| \leq \lambda X^{1-\varepsilon/10}} \sum_{\chi\ (q_1)} \int_{t+u-\lambda H}^{t+u+\lambda H} |F(t')|\ dt' \frac{du}{1+|u|}\right)^2\ dt  \\
&\leq \left( \int_{|u| \leq \lambda X^{1-\varepsilon/10}} \left(
\int_{Q^{-1/2} \lambda X \ll |t| \ll Q^{1/2} \lambda X} \left(\sum_{\chi\ (q_1)} \int_{t+u-\lambda H}^{t+u+\lambda H} |F(t')|\ dt'\right)^2\ dt\right)^{1/2} \frac{du}{1+|u|}\right)^2 \\
&\ll \log^2 X  \int_{Q^{-1/2} \lambda X \ll |t| \ll Q^{1/2} \lambda X} \left(\sum_{\chi\ (q_1)} \int_{t-\lambda H}^{t+\lambda H} |F(t')|\ dt'\right)^2\ dt
\end{align*}
where we allow the implied constants in the region $\{ Q^{-1/2} \lambda X \ll |t| \ll Q^{1/2} \lambda X\}$ to vary from line to line.
Putting all this together, we now see that Proposition \ref{mve} will be a consequence of the following estimates.

\begin{proposition}[Estimates for Type $d_1,d_2,d_3,d_4$,II sums]\label{types}   Let $\eps>0$ be sufficiently small.  Let $k \geq 2$ and $A>0$ be fixed, and let $B>0$ be sufficiently large depending on $A,k$.  Let $X \geq 2$, and set $H \coloneqq X^{\record+\eps}$.  Set $Q \coloneqq  \log^B X$, and let $1 \leq q_1 \leq Q$.  Let $\lambda$ be a quantity such that $X^{-1/6-2\eps} \leq \lambda \ll \frac{1}{q_1 Q}$.  Let $f \colon \N \to \C$ be a function of one of the following forms:
\begin{itemize}
\item[(Type $d_1,d_2,d_3,d_4$ sums)] One has
\begin{equation}\label{fbb}
 f = \alpha \ast \beta_1 \ast \dots \ast \beta_j
\end{equation}
for some $O_{k,\eps}(1)$-divisor-bounded arithmetic functions $\alpha,\beta_1,\dots,\beta_j \colon \N \to \C$, where $j=1,2,3,4$, $\alpha$ is supported on $[N,2N]$, and each $\beta_i$, $i=1,\dots,j$ is supported on $[M_i, 2M_i]$ for some $N, M_1,\dots,M_j$ obeying the bounds 
$$ 1 \ll N \ll_{k,\eps} X^{\eps^2},$$
$$ X/Q \ll N M_1 \dots M_j \ll_{k,\eps} X,$$
and
$$ X^{-\eps^2} H \ll M_1 \ll \dots \ll M_j \ll X.$$ 
Furthermore, each $\beta_i$ is either of the form $\beta_i = 1_{(M_i,2M_i]}$ or $\beta_i = L 1_{(M_i,2M_i]}$.
\item[(Type II sum)]  One has
$$ f = \alpha \ast \beta$$
for some $O_{k,\eps}(1)$-divisor-bounded arithmetic functions $\alpha,\beta \colon \N \to \C$ supported on $[N,2N]$ and $[M,2M]$ respectively, for some $N,M$ obeying the bounds 
\begin{equation}\label{nib}
 X^{\eps^2} \ll N \ll X^{-\eps^2} H
\end{equation}
and
\begin{equation}\label{xq}
 X/Q \ll NM \ll_{k,\eps} X.
\end{equation}
\end{itemize}
Then
\begin{equation}\label{lk}
 \int_{Q^{-1/2} \lambda X \ll |t| \ll Q^{1/2} \lambda X} \left(\sum_{\chi\ (q_1)} \int_{t-\lambda H}^{t+\lambda H} \left|{\mathcal D}[f]\left(\frac{1}{2}+it',\chi\right)\right|\ dt'\right)^2\ dt \ll_{k,\eps,A,B} q_1 \lambda^2 H^2 X \log^{-A} X.
\end{equation}
\end{proposition}

It remains to prove Proposition \ref{types}.  We will deal with the Type $d_1$, Type $d_2$, Type $d_4$, and Type II cases in next section; the Type $d_3$ case is trickier and we will prove it in next section only assuming an average result for exponential sums which will in turn be proven in Section \ref{vdc-sec}.

\section{Proof of Proposition~\ref{types}}

We begin with the Type II case.  Since $f = \alpha \ast \beta$, we may factor
$$
{\mathcal D}[f]\left(\frac{1}{2}+it',\chi\right) = {\mathcal D}[\alpha]\left(\frac{1}{2}+it',\chi\right) {\mathcal D}[\beta]\left(\frac{1}{2}+it',\chi\right)
$$
and hence by Cauchy-Schwarz we have
\begin{align*}
&\left(\sum_{\chi\ (q_1)} \int_{t-\lambda H}^{t+\lambda H} \left|{\mathcal D}[f]\left(\frac{1}{2}+it',\chi\right)\right|\ dt'\right)^2 \\
&\ll \left(\sum_{\chi\ (q_1)} \int_{t-\lambda H}^{t+\lambda H}\left |{\mathcal D}[\alpha]\left(\frac{1}{2}+it',\chi\right)\right|^2\ dt'\right) \\
&\quad \times \left(\sum_{\chi\ (q_1)} \int_{t-\lambda H}^{t+\lambda H} \left|{\mathcal D}[\beta]\left(\frac{1}{2}+it',\chi\right)\right|^2\ dt'\right) 
\end{align*}
for any $t$.  From Lemma \ref{mvt-lem-ch} we have
$$
\sum_{\chi\ (q_1)} \int_{t-\lambda H}^{t+\lambda H} \left|{\mathcal D}[\alpha]\left(\frac{1}{2}+it',\chi\right)\right|^2\ dt'
\ll_{k,\eps} (q_1 \lambda H + N) \log^{O_{k,\eps}(1)} X
$$
while from Fubini's theorem and Lemma \ref{mvt-lem-ch} we have
$$
\int_{Q^{-1/2} \lambda X \ll |t| \ll Q^{1/2} \lambda X} \sum_{\chi\ (q_1)} \int_{t-\lambda H}^{t+\lambda H} \left|{\mathcal D}[\beta]\left(\frac{1}{2}+it',\chi\right)\right|^2\ dt'
\ll_{k,\eps} \lambda H (q_1 Q^{1/2} \lambda X + M) \log^{O_{k,\eps}(1)} X
$$
and so we can bound the left-hand side of \eqref{lk} by
$$ \ll_{k,\eps} (q_1 \lambda H + N)\left( q_1 Q^{1/2} \lambda X + M \right) \lambda H \log^{O_{k,\eps}(1)} X.$$
We rewrite this expression using  \eqref{xq} as
$$ \ll_{k,\eps} q_1 \left( Q^{1/2} q_1 \lambda + \frac{Q^{1/2} N}{H} + \frac{1}{N} + \frac{1}{q_1 \lambda H} \right) \lambda^2 H^2 X \log^{O_{k,\eps}(1)} X.$$
Using the hypotheses \eqref{l6}, \eqref{nib}, \eqref{H-def}, \eqref{Q-def}, we obtain \eqref{lk} in the Type II case as required.

Now we handle the Type $d_1$ and $d_2$ cases. Actually we may unify the $d_1$ case into the $d_2$ case by adding a dummy factor $\beta_2$ (i.e $\beta_2(n) = \mathbf{1}_{n = 1}$) so that in both cases we have
$$ f = \alpha \ast \beta_1 \ast \beta_2$$
where $\alpha$ is  supported on $[N,2N]$ and is $O_{k,\eps,B}(1)$-divisor-bounded, and each $\beta_i$ is either $1_{(M_i,2M_i]}$ or $L 1_{(M_i,2M_i]}$, where
\begin{equation}\label{nib2}
1 \ll N \ll X^{\eps^2}
\end{equation}
and
$$ X/Q \ll N M_1 M_2 \ll X.$$

We may factor
$$
{\mathcal D}[f]\left(\frac{1}{2}+it',\chi\right) = {\mathcal D}[\alpha]\left(\frac{1}{2}+it',\chi\right) {\mathcal D}[\beta_1]\left(\frac{1}{2}+it',\chi\right)
{\mathcal D}[\beta_2]\left(\frac{1}{2}+it',\chi\right).$$
By Cauchy-Schwarz we have
\begin{align*}
&\left(\sum_{\chi\ (q_1)} \int_{t-\lambda H}^{t+\lambda H} \left|{\mathcal D}[f]\left(\frac{1}{2}+it',\chi\right)\right|\ dt'\right)^2 \\
&\ll \left(\sum_{\chi\ (q_1)} \int_{t-\lambda H}^{t+\lambda H} \left|{\mathcal D}[\alpha]\left(\frac{1}{2}+it',\chi\right)\right|^2\ dt'\right) \\
&\quad \times \left(\sum_{\chi\ (q_1)} \int_{t-\lambda H}^{t+\lambda H} \left|{\mathcal D}[\beta_1]\left(\frac{1}{2}+it',\chi\right)\right|^2 \left|{\mathcal D}[\beta_2]\left(\frac{1}{2}+it',\chi\right)\right|^2\ dt'\right).
\end{align*}
From Lemma \ref{mvt-lem} we have
$$
\sum_{\chi\ (q_1)} \int_{t-\lambda H}^{t+\lambda H} \left|{\mathcal D}[\alpha]\left(\frac{1}{2}+it',\chi\right)\right|^2\ dt'
\ll_{k,\eps} (q_1 \lambda H + N) \log^{O_{k,\eps}(1)} X
$$
so from Fubini's theorem we can bound the left-hand side of \eqref{lk} by
\begin{align*}
& \ll_{k,\eps} (q_1 \lambda H + N) \lambda H \log^{O_{k,\eps}(1)} X \\
&\quad \times \sum_{\chi\ (q_1)} \int_{Q^{-1/2} \lambda X \ll |t| \ll Q^{1/2} \lambda X} \left|{\mathcal D}[\beta_1]\left(\frac{1}{2}+it,\chi\right)\right|^2 \left|{\mathcal D}[\beta_2]\left(\frac{1}{2}+it,\chi\right)\right|^2\ dt.
\end{align*}

By the pigeonhole principle, we can thus bound the left-hand side of \eqref{lk} by
\begin{align*}
& \ll_{k,\eps} (q_1 \lambda H + N) \lambda H \log^{O_{k,\eps}(1)} X \\
&\qquad \times \sum_{\chi\ (q_1)} \int_{T/2 \leq |t| \leq T} \left|{\mathcal D}[\beta_1]\left(\frac{1}{2}+it,\chi\right)\right|^2 \left|{\mathcal D}[\beta_2]\left(\frac{1}{2}+it,\chi\right)\right|^2\ dt 
\end{align*}
for some  $T$ with
\begin{equation}\label{tb}
Q^{-1/2} \lambda X \ll T \ll Q^{1/2} \lambda X.
\end{equation}

By Corollary \ref{fourth-integral} and the triangle inequality, we have
$$ \sum_{\chi\ (q_1)} \int_{T/2 \leq |t| \leq T} \left|{\mathcal D}[\beta_1]\left(\frac{1}{2}+it,\chi\right)\right|^4\ dt \ll q_1 T \left(1 + \frac{q_1^2}{T^2} + \frac{M_1^2}{T^4} \right) \log^{O(1)} X; 
$$  
since $T \gg \lambda X Q^{-1/2} \geq \max\{M_1^{1/2}, q_1\}$ by our assumptions, we get
$$ \sum_{\chi\ (q_1)} \int_{T/2 \leq |t| \leq T}\left |{\mathcal D}[\beta_1]\left(\frac{1}{2}+it,\chi\right)\right|^4\ dt \ll q_1 T \log^{O(1)} X.$$
Similarly with $\beta_1$ replaced by $\beta_2$.  By Cauchy-Schwarz, we thus have
$$ \sum_{\chi\ (q_1)} \int_{T/2 \leq |t| \leq T} \left|{\mathcal D}[\beta_1]\left(\frac{1}{2}+it,\chi\right)\right|^2 \left|{\mathcal D}[\beta_2]\left(\frac{1}{2}+it,\chi\right)\right|^2\ dt \ll q_1 T \log^{O(1)} X$$
and so we can bound the left-hand side of \eqref{lk} by
$$ \ll_{k,\eps} (q_1 \lambda H + N) \lambda H q_1 T \log^{O_{k,\eps}(1)} X.$$
Using \eqref{tb}, we can bound this by
$$ \ll_{k,\eps} q_1 \left(q_1 Q^{1/2}\lambda + \frac{N Q^{1/2}}{H} \right) \lambda^2 H^2 X \log^{O_{k,\eps}(1)} X.$$
Using \eqref{l6}, \eqref{nib2}, \eqref{H-def}, \eqref{Q-def}, we obtain \eqref{lk} as desired.

\begin{remark} The above arguments recover the results of Mikawa \cite{mikawa} and Baier, Browning, Marasingha, and Zhao \cite{bbmz}, in which $\sigma$ is now set equal to $\frac{1}{3}$ so that we can take $m = 3$, and the Type $d_3$ and Type $d_4$ sums do not appear.
\end{remark}

Now we turn to the Type $d_j$ cases for $j=3,4$.  Here we have
\begin{equation}\label{nam-d4}
 N M_1 \dots M_j \ll_{k,\eps} X
\end{equation}
and
\begin{equation}\label{nam2-d4}
 X^{-\eps^2} H \ll M_1 \ll \dots \ll M_j
\end{equation}
which implies that
\begin{equation}\label{mij}
 M_1 \ll X^{1/j}.
\end{equation}

We now factor $f = \beta_1 * g$ where $g \coloneqq  \alpha \ast \beta_2 \ast \dots \ast \beta_j$, so that
$$ {\mathcal D}[f]\left(\frac{1}{2}+it,\chi\right) = {\mathcal D}[\beta_1]\left(\frac{1}{2}+it,\chi\right) {\mathcal D}[g]\left(\frac{1}{2}+it,\chi\right).$$
The function $g$ is supported in the range $\{ n: n \asymp N M_2 \dots M_j \}$ and is $O_{k,\eps}(1)$-divisor bounded.  By Cauchy-Schwarz, the left-hand side of \eqref{lk} may be bounded by
\begin{align*}
& \int_{Q^{-1/2} \lambda X \ll |t| \ll Q^{1/2} \lambda X} \left(\sum_{\chi\ (q_1)} \int_{t-\lambda H}^{t+\lambda H} \left|{\mathcal D}[g]\left(\frac{1}{2}+it',\chi\right)\right|^2\ dt'\right) \\
&\quad \left(\sum_{\chi\ (q_1)} \int_{t-\lambda H}^{t+\lambda H} \left|{\mathcal D}[\beta_1]\left(\frac{1}{2}+it'',\chi\right)\right|^2\ dt''\right)\ dt.
\end{align*}
Using Fubini's theorem to perform the $t$ integral first, we can estimate this by
\begin{align*}
& \ll \lambda H \sum_{\chi\ (q_1)} \int_{Q^{-1/2} \lambda X \ll |t'| \ll Q^{1/2} \lambda X} \left|{\mathcal D}[g]\left(\frac{1}{2}+it',\chi\right)\right|^2 \\ 
&\quad \left(\sum_{\chi'\ (q_1)} \int_{t'-2\lambda H}^{t'+2\lambda H} \left|{\mathcal D}[\beta_1]\left(\frac{1}{2}+it'',\chi'\right)\right|^2\ dt''\right)\ dt'.
\end{align*}

We fix a smooth non-negative Schwartz function $\eta \colon \R \to \R$, positive on $[-2,2]$ and whose Fourier transform $\hat \eta(u) \coloneqq \int_\R \eta(t) e(-tu)\ du$ is supported on $[-1,1]$ with $\hat \eta(0)=1$.  Since $\lambda \leq \frac{1}{q_1 Q}$, we can bound
\begin{align*}
&\sum_{\chi'\ (q_1)} \int_{t'-2\lambda H}^{t'+2\lambda H} \left|{\mathcal D}[\beta_1]\left(\frac{1}{2}+it'',\chi'\right)\right|^2\ dt'' \\
&\ll
\sum_{\chi'\ (q_1)} \int_\R \left|{\mathcal D}[\beta_1]\left(\frac{1}{2}+it'',\chi'\right)\right|^2 \eta\left(\frac{(t''-t') q_1 Q}{H}\right) \ dt'' \\
&= \int_\R \eta\left(\frac{(t''-t') q_1 Q}{H}\right) \sum_{\chi'\ (q_1)} \sum_{m_1,m_2} \frac{\beta_1(m_1) \overline{\beta_1(m_2)} \chi'(m_1) \overline{\chi'}(m_2)}{m_1^{1/2+it''} m_2^{1/2-it''}}\ dt'' \\
&= \phi(q_1) \int_\R \eta\left(\frac{(t''-t') q_1 Q}{H}\right) \sum_{1 \leq a \leq q_1: (a,q_1)=1} \left|\sum_{m = a\ (q_1)} \frac{\beta_1(m)}{m^{1/2+it''}}\right|^2\ dt'' \\
&\ll q_1 \int_\R \eta\left(\frac{(t''-t') q_1 Q}{H}\right) \sum_{1 \leq a \leq 2q_1} \left|\sum_{m = a\ (2q_1)} \frac{\beta_1(m)}{m^{1/2+it''}}\right|^2\ dt'' \\
&= \frac{H}{Q} \sum_{m_1 = m_2\ (2q_1)} \frac{\beta_1(m_1) \beta_1(m_2)}{m_1^{1/2+it'} m_2^{1/2-it'}} \hat \eta\left( \frac{H}{2\pi q_1 Q} \log \frac{m_1}{m_2} \right) \\
&= \frac{H}{Q} \sum_{m,\ell} \frac{\beta_1(m+q_1 \ell) \beta_1(m-q_1 \ell)}{(m+q_1 \ell)^{1/2+it'} (m-q_1 \ell)^{1/2-it'}} \hat \eta\left( \frac{H}{2\pi q_1 Q} \log \frac{m+q_1 \ell}{m-q_1 \ell} \right) 
\end{align*}
where we have used the balanced change of variables $(m_1,m_2) = (m+q_1\ell, m-q_1\ell)$ to obtain some cancellation in a Taylor expansion that will be performed in the next section.  Observe that the $\ell=0$ contribution to the above expression is $O(\frac{H}{Q}\log X (\log M_1)^2)$; also, the quantity
$\beta_1(m+q_1 \ell) \beta_1(m-q_1 \ell) \hat{\eta}$ is only non-vanishing when $m - q_1 \ell \asymp M_1$ and $\log \frac{m+q_1 \ell}{m-q_1 \ell} \ll \frac{q_1 Q}{H} \ll \frac{Q^2}{H}$, which implies that $q_1 \ell \ll \frac{Q^2 M_1}{H}$ and $m \asymp M_1$.  By symmetry and the triangle inequality (and crudely summing over $q_1 \ell$ instead of over $\ell$) we have
$$
\sum_{\chi\ (q_1)} \int_{t'-2\lambda H}^{t'+2\lambda H} \left|{\mathcal D}[\beta_1]\left(\frac{1}{2}+it'',\chi\right)\right|^2\ dt'' \ll Y(t') + \frac{H}{Q} \log X (\log M_1)^2$$
where $Y(t')$ denotes the quantity
$$
 Y(t') \coloneqq
 \frac{H}{Q} \sum_{1 \leq \ell \ll \frac{Q^2 M_1}{H}} \left|\sum_m \frac{\beta_1(m+\ell) \beta_1(m-\ell)}{(m+\ell)^{1/2+it'} (m-\ell)^{1/2-it'}} \hat \eta\left( \frac{H}{2\pi q_1 Q} \log \frac{m+ \ell}{m- \ell}\right)\right|.
$$
The function $m \mapsto \frac{\beta_1(m+\ell) \beta_1(m-\ell)}{(m+\ell)^{1/2} (m-\ell)^{1/2}} \hat \eta( \frac{H}{2\pi q_1 Q} \log \frac{m+ \ell}{m- \ell})$ is supported on the interval $[M_1+\ell, 2M_1-\ell]$, is of size $O_\eps( X^{O(\eps^2)} / M_1 )$ on this interval, and has derivative of size $O_\eps( X^{O(\eps^2)} / M_1^2 )$.  Thus by Lemma \ref{sbp-2}, one has $Y(t') \ll_\eps X^{O(\eps^2)} \tilde Y(t')$, where $\tilde Y(t')$ denotes the quantity
\begin{equation}\label{yat}
 \tilde Y(t') \coloneqq
\frac{H}{M_1} \sum_{1 \leq \ell \ll \frac{Q^2 M_1}{H}} \left|\sum_{M_1 \leq m \leq 2M_1} e\left( \frac{t'}{2\pi} \log \frac{m+ \ell}{m- \ell}  \right) \right|^*.
\end{equation}

We may thus bound the left-hand side of \eqref{lk} by $O(Z_1+Z_2)$, where
$$ Z_1 \coloneqq \lambda H \sum_{\chi\ (q_1)} \int_{\lambda X/Q^{1/2} \ll |t'| \ll Q^{1/2} \lambda X} \left|{\mathcal D}[g]\left(\frac{1}{2}+it',\chi\right)\right|^2 \tilde Y(t')\ dt'$$
and
$$ Z_2 \coloneqq \frac{\lambda H^2}{Q} \log X (\log M_1)^2 \sum_{\chi\ (q_1)} \int_{\lambda X/Q^{1/2} \ll |t'| \ll Q^{1/2} \lambda X} \left|{\mathcal D}[g]\left(\frac{1}{2}+it',\chi\right)\right|^2\ dt'.$$
From Lemma \ref{mvt-lem-ch} we have
$$ Z_2 \ll_k \frac{\lambda H^2}{Q} (q_1 Q^{1/2} \lambda X + N M_2 \dots M_j) \log^{O_k(1)} X.$$
From \eqref{nam-d4}, \eqref{nam2-d4} we have
$$ N M_2 \dots M_j \ll_{k,\eps} \frac{X}{M_1} \ll X^{\eps^2} \frac{X}{H}$$
and hence by \eqref{Q-def}, \eqref{H-def} and \eqref{l6}
$$ N M_2 \dots M_j \ll q_1 Q^{1/2} \lambda X.$$
We thus have
$$ Z_2 \ll_{k,\eps} Q^{-1/2} q_1 \lambda^2 H^2 X \log^{O_k(1)} X;$$
by \eqref{Q-def}, this contribution is acceptable for $B$ large enough.

Now we turn to $Z_1$.  At this point we will begin conceding factors of $X^{O(\eps^2)}$, in particular we can essentially ignore the role of the parameters $q_1$ and $Q$ thanks to \eqref{Q-def}.

We begin with the easier case $j=4$ of Type $d_4$ sums.  To deal with ${\mathcal D}[g]$ in this case, we simply invoke Lemma \ref{mvt-lem-ch} to obtain the bound
$$ \sum_{\chi\ (q_1)} \int_{Q^{-1/2} \lambda X \ll |t'| \ll Q^{1/2} \lambda X} \left|{\mathcal D}[g]\left(\frac{1}{2}+it',\chi\right)\right|^2\ dt' \ll_\eps X^{O(\eps^2)} \lambda X.$$
To show the contribution of $Z_1$ is acceptable in the $j=4$ case, it thus suffices to show the following lemma.

\begin{lemma}\label{pt}  Let the notation be as above with $j = 4$.  Then for any $Q^{-1/2} \lambda X \ll |t'| \ll Q^{1/2} \lambda X$, one has
$$ \tilde Y(t') \ll_\eps X^{-\eps + O(\eps^2)} H.$$
\end{lemma}

\begin{proof}  From \eqref{yat} and the triangle inequality it suffices to show that
$$
\left|\sum_{M_1 \leq m \leq 2M_1} e\left( \frac{t'}{2\pi} \log \frac{m+ \ell}{m- \ell} \right )\right|^* \ll_\eps X^{-c\eps + O(\eps^2)} H$$
for all $1 \leq \ell \ll Q^2 M_1 / H$. The phase $m \mapsto \frac{t'}{2\pi} \log \frac{m+ \ell}{m- \ell}$ has $j^{\operatorname{th}}$ derivative $\asymp_j \frac{|t'| \ell}{M_1^{j+1}}$ for all $j \geq 1$.  Using the classical van der Corput exponent pair $(1/14,2/7)$ (see \cite[\S 8.4]{ik}) we have
$$
\left|\sum_{M_1 \leq m \leq 2M_1} e\left( \frac{t'}{2\pi} \log \frac{m+ \ell}{m- \ell}  \right)\right|^* \ll_\eps X^{O(\eps^2)} \left(\frac{|t'| \ell}{M_1^2}\right)^{\frac{1}{14}} M_1^{\frac{2}{7} + \frac{1}{2}}$$
(noting from \eqref{l6}, \eqref{mij} that $|t'| \ell \geq |t'| \geq M_1^2$).  Using \eqref{mij}, \eqref{Q-def}, the right-hand side is
$$ \ll_\eps X^{O(\eps^2)} (X/H)^{\frac{1}{14}} (X^{1/4})^{\frac{2}{7} + \frac{1}{2} - \frac{1}{14}}.$$
From \eqref{sigma-range}, \eqref{H-def} we have $H \geq X^{\frac{7}{30}+\eps}$, and the claim then follows after some arithmetic.
\end{proof}

\begin{remark} If one uses the recent improvements of Robert \cite{robert} to the classical $(1/14,2/7)$ exponent pair, one can establish Lemma \ref{pt} for $\sigma$ as small as $\frac{3}{13} = 0.2307\dots$, improving slightly upon the exponent $\frac{7}{30} = 0.2333\dots$ provided by the classical pair.  Unfortunately, due to the need to also treat the $d_3$ sums, this does not improve the final exponent \eqref{sigmadef} in Theorem \ref{unav-corr}.
\end{remark}

Now we turn to estimating $Z_1$ in the $j=3$ case of Type $d_3$ sums.  To deal with ${\mathcal D}[g]$ in this case, we apply Jutila's estimate (Corollary \ref{jutila-cor}) to conclude

\begin{proposition}\label{qwe}  Let the notation and assumptions be as above.  Cover the region $\{ t': Q^{-1/2} \lambda X \ll |t'| \ll Q^{1/2} \lambda X\}$ by a collection ${\mathcal J}$ of disjoint half-open intervals $J$ of length $X^{\eps^2} \sqrt{\lambda X}$.  Then
$$ 
\sum_{J \in {\mathcal J}} \left( \sum_{\chi\ (q_1)} \int_{J} \left|{\mathcal D}[g]\left(\frac{1}{2}+it',\chi\right)\right|^2\ dt'\right)^{3} \ll_{k,\eps,B} X^{O(\eps^2)} (\lambda X)^2.
$$
\end{proposition}

\begin{proof}  For each $R > 0$, let ${\mathcal J}_R$ denote the set of those intervals $J \in {\mathcal J}$ such that
$$
R \leq \sum_{\chi\ (q_1)} \int_J \left|{\mathcal D}[g]\left(\frac{1}{2}+it',\chi\right)\right|^2\ dt' \leq 2R.$$
Applying Corollary \ref{jutila-cor} with $Q^{-1/2} \lambda X \ll T \ll Q^{1/2} \lambda X$ and $T_0 \coloneqq X^{\eps^2} \sqrt{\lambda X}$ (and $\eps$ replaced by $\eps^2$), together with the triangle inequality and conjugation symmetry, we have
$$
\sum_{J \in {\mathcal J}_R} \sum_{\chi \ (q_1)} \int_J  |{\mathcal D}[\beta_j](\tfrac 12 + it, \chi)|^4\ dt \ll_{\eps,B} X^{O(\eps^2)} ((\# {\mathcal J}_R) \sqrt{\lambda X} + ((\# {\mathcal J}_R) \lambda X)^{2/3}) $$
for $j=2,3$, where we recall that $\# {\mathcal J}_R$ denotes the cardinality of ${\mathcal J}_R$.  Note that the hypothesis $M_j \ll T^2$ required for Corollary \ref{jutila-cor} will follow from \eqref{nam-d4} and \eqref{l6}.  From Cauchy-Schwarz and the crude estimate
$$ {\mathcal D}[\alpha](\tfrac 12 + it, \chi) \ll_{k,\eps} X^{O(\eps^2)}$$
we thus have
$$
\sum_{J \in {\mathcal J}_R} \sum_{\chi \ (q_1)} \int_{J} |{\mathcal D}[g](\tfrac 12 + it, \chi)|^2\ dt \ll_{k,\eps,B} X^{O(\eps^2)} ((\# {\mathcal J}_R) \sqrt{\lambda X} + ((\# {\mathcal J}_R) \lambda X)^{2/3}).$$
By definition of ${\mathcal J}_R$, we conclude that
$$ R \# {\mathcal J}_R \ll_{k,\eps,B} X^{O(\eps^2)} ((\# {\mathcal J}_R) \sqrt{\lambda X} + ((\# {\mathcal J}_R) \lambda X)^{2/3})$$
and thus either $R \ll_{k,\eps,B} X^{O(\eps^2)} \sqrt{\lambda X}$ or $\# {\mathcal J}_R \ll_{k,\eps,B} X^{O(\eps^2)} (\lambda X)^2 / R^3$.  Using the trivial bound $\# {\mathcal J}_R \ll \sqrt{\lambda X}$ in the former case, we thus have
$$ \# {\mathcal J}_R \ll_{k, \eps, B} X^{O(\eps^2)}  \min\left( \frac{(\lambda X)^2}{R^3}, \sqrt{\lambda X} \right)$$
for all $R>0$.  The claim then follows from dyadic decomposition (noting that ${\mathcal J}_R$ is only non-empty when $R \ll X^{O(1)}$).
\end{proof}

In the next section, we will establish a discrete fourth moment estimate for the $Y(t)$:

\begin{proposition}\label{pq}  Let the notation and assumptions be as above.  Let $t_1 < \dots < t_r$ be elements of $\{ t': Q^{-1/2} \lambda X \ll |t'| \ll Q^{1/2} \lambda X\}$ such that $|t_{j+1}-t_j| \geq \sqrt{\lambda X}$ for all $1 \leq j < r$.  Then
$$ \sum_{j=1}^r \tilde Y(t_j)^4 \ll_{k,\eps,B} X^{-\eps+O(\eps^2)} H^4 \sqrt{\lambda X}.$$
\end{proposition}

Assume this proposition for the moment.  Cover the set $\{ t': Q^{-1/2} \lambda X \ll |t'| \ll Q^{1/2} \lambda X\}$ by a family ${\mathcal J}$ of
disjoint half-open intervals $J$ of length $X^{\eps^2} \sqrt{\lambda X}$ for $i=1,\dots,r$.  On each such $J$, let $t_J$ be a point in $J$ that maximizes the quantity $\tilde Y(t_J)$.  One can partition the $t_J$ into $O(1)$ subsequences that are $\sqrt{\lambda X}$-separated in the sense of Proposition \ref{pq}.  From the triangle inequality, we thus have
$$ \sum_{J \in {\mathcal J}} \tilde Y(t_J)^4 \ll_{k,\eps,B} X^{-\eps+O(\eps^2)} H^4 \sqrt{\lambda X}$$
and hence by H\"older's inequality and the cardinality bound $|{\mathcal J}| \ll_{\eps,B} X^{O(\eps^2)} \sqrt{\lambda X}$
$$ \sum_{J \in {\mathcal J}} \tilde Y(t_J)^{3/2} \ll_{k,\eps,B} X^{-3\eps/8} X^{O(\eps^2)} H^{3/2} \sqrt{\lambda X}.$$
On the other hand, we can bound 
$$ Z_1 \leq
\lambda H \sum_{J \in {\mathcal J}} \tilde Y(t_J) \sum_{\chi\ (q_1)} \int_{J} \left|{\mathcal D}[g]\left(\frac{1}{2}+it',\chi\right)\right|^2\ dt'$$
and hence by H\"older's inequality and Proposition~\ref{qwe} we have
$$ Z_1 \ll_{k,\eps,B} \lambda H (X^{-3\eps/8+O(\eps^2)} H^{3/2} \sqrt{\lambda X})^{2/3}
((\lambda X)^2)^{1/3}$$
which simplifies to
$$ Z_1 \ll_{k,\eps,B} \lambda^2 H^2 X^{-\eps/4 + O(\eps^2)} X$$
which is an acceptable contribution to \eqref{lk} for $\eps$ small enough.  This completes the proof of \eqref{lk}.

Thus it remains only to establish Proposition \ref{pq}.  This will be the objective of the next section.

%%% Local Variables: 
%%% mode: latex
%%% TeX-master: "correlations"
%%% End: 

\section{Averaged exponential sum estimates}\label{vdc-sec}

We now prove Proposition \ref{pq}.  We will now freely lose factors of $X^{O(\eps^2)}$ in our analysis, for instance we see from the hypothesis $Q \leq \log^B X$ that
\begin{equation}\label{qsmall}
 1\leq q_1 \leq Q \ll_{B,\eps} X^{\eps^2}.
\end{equation}
By partitioning the $t_j$ based on their sign, and applying a conjugation if necessary, we may assume that the $t_j$ are all positive.  By covering the positive portion of $\{ Q^{-1/2} \lambda X \ll |t| \ll \lambda X Q^{1/2}\}$ into dyadic intervals $[T,2T]$ (and giving up an acceptable loss of $O(\log X)$), we may assume that there exists
$$ Q^{-1/2} \lambda X \ll T \ll \lambda X Q^{1/2}$$
such that $t_1,\dots,t_r \in  [T,2T]$; from \eqref{qsmall} we see in particular that
\begin{equation}\label{tbound}
T = X^{O(\eps^2)} \lambda X.
\end{equation}
From \eqref{l6} we also note that
\begin{equation}\label{tbound-2}
X^{5/6-2\eps} \ll T \ll X.
\end{equation}
Since the $t_1,\dots,t_r$ are $\sqrt{\lambda X}$-separated, we have
\begin{equation}\label{qor}
 r \ll_\eps X^{O(\eps^2)} \sqrt{\lambda X}.
\end{equation}
Finally, from \eqref{mij} we have
\begin{equation}\label{mi3}
M_1 \ll X^{1/3}.
\end{equation}

Now we need to control the maximal exponential sums $\tilde Y(t_j)$ defined in \eqref{yat}.  If one uses exponent pairs such as $(1/6,1/6)$ here as in Lemma \ref{pt} to obtain uniform control on the $\tilde Y(t_j)$, one obtains inferior results (indeed, the use of $(1/6,1/6)$ only gives Theorem \ref{unav-corr} for $\sigma = \frac{2}{7} = 0.2857\dots$).   Instead, we will exploit the averaging in $j$.  We first use H\"older's inequality to note that
$$ \tilde Y(t_j)^4 \ll_\eps X^{O(\eps^2)} \frac{H}{M_1}
\sum_{1 \leq \ell \ll \frac{Q^2 M_1}{H}} 
 \left(\left|\sum_{M_1 \leq m \leq 2M_1} e\left( \frac{t_j}{2\pi} \log \frac{m+ \ell}{m- \ell} \right)\right|^*\right)^4.$$
Next, we observe that $\log \frac{m+ \ell}{m- \ell}  = \log \frac{1+\ell/m}{1-\ell/m}$ has the Taylor expansion
\begin{align*}
\log \frac{m+ \ell}{m- \ell} &= \sum_{j=0}^\infty \frac{2}{2j+1} \left(\frac{\ell}{m}\right)^{2j+1} \\
&= 2 \frac{\ell}{m} + \frac{2}{3} \frac{\ell^3}{m^3} + \frac{2}{5} \frac{\ell^5}{m^5} + \dots.
\end{align*}
Note that there are no terms in the Taylor expansion with even powers of $\frac{\ell}{m}$.
Thus we can write
$$ e\left(\frac{t_j}{2\pi} \log \frac{m+ \ell}{m- \ell}\right) = e\left(\frac{1}{\pi} \frac{t_j \ell}{m} + \frac{1}{3\pi} \frac{t_j \ell^3}{m^3}\right) e(R_{j,\ell}(m))$$
where for $m \in [M_1,2M_1]$, the remainder $R_{j,\ell}(m)$ is of size
$$ R_{j,\ell}(m) \ll T \left(\frac{Q^2 M_1/H}{M_1}\right)^5 \ll_\eps X^{-\frac{7}{33}+O(\eps^2)} $$
and has derivative estimates
$$ R_{j,\ell}'(m) \ll \frac{1}{M_1} T \left(\frac{Q^2 M_1/H}{M_1}\right)^5 \ll_\eps \frac{1}{M_1} X^{-\frac{7}{33}+O(\eps^2)}.$$
Thus by Lemma \ref{sbp-2} again, we have
$$ \tilde Y(t_j)^4 \ll_\eps X^{O(\eps^2)} \frac{H}{M_1}
\sum_{1 \leq \ell \ll \frac{Q^2 M_1}{H}} 
 \left(\left|\sum_{M_1 \leq m \leq 2M_1} e\left( \frac{1}{\pi} \frac{t_j \ell}{m} + \frac{1}{3\pi} \frac{t_j \ell^3}{m^3} \right)\right|^*\right)^4.
$$
We write this bound as
$$ \tilde Y(t_j)^4 \ll_\eps X^{O(\eps^2)} \frac{H}{M_1}
\sum_{1 \leq \ell \ll \frac{Q^2 M_1}{H}} f\left( \frac{t_j \ell}{M_1}, \frac{t_j \ell^3}{M_1^3} \right)^4 $$
$f(\alpha,\beta)$ denotes the maximal exponential sum
\begin{equation}\label{fab}
 f(\alpha,\beta) \coloneqq \left|\sum_{M_1 \leq m \leq 2M_1} e\left( \frac{\alpha}{\pi} \frac{M_1}{m} + \frac{\beta}{3\pi} \frac{M_1^3}{m^3} \right)\right|^*.
\end{equation}
By a further application of Lemma \ref{sbp-2}, we see that
\begin{equation}\label{f-stable}
f(\alpha+u,\beta+v) \asymp f(\alpha,\beta)
\end{equation}
whenever $\alpha,\beta,u,v$ are real numbers with $u,v = O(1)$.  Thus
$$ \tilde Y(t_j)^4 \ll_\eps X^{O(\eps^2)} \frac{H}{M_1}
\sum_{1 \leq \ell \ll \frac{Q^2 M_1}{H}} \int_{t_j \ell/M_1}^{t_j \ell/M_1 + 1} f\left( t, \frac{\ell^2}{M_1^2} t \right)^4\ dt.$$
As the $t_j$ are $\sqrt{\lambda X}$-separated and lie in $[T,2T]$, and by \eqref{l6}, \eqref{mi3} we have
$$ (\lambda X)^{1/2} \gg X^{5/12 - \eps/2} \gg X^{1/3} \gg M_1,$$
we see that for fixed $\ell$, the intervals $[t_j \ell/M_1, t_j \ell/M_1 + 1]$ are disjoint and lie in the region $\{ t: t \asymp T \ell / M_1 \}$.  Thus we have
$$ \sum_{j=1}^r \tilde Y(t_j)^4 \ll_\eps X^{O(\eps^2)} \frac{H}{M_1}
\sum_{1 \leq \ell \ll \frac{Q^2 M_1}{H}} \int_{|t| \asymp T \ell / M_1} f\left( t, \frac{\ell^2}{M_1^2} t\right )^4\ dt.$$
By the pigeonhole principle, we thus have
$$ \sum_{j=1}^r \tilde{Y}(t_j)^4 \ll_\eps X^{O(\eps^2)} \frac{H}{M_1}
\sum_{L \leq \ell < 2L} \int_{|t| \asymp T L / M_1} f\left( t, \frac{\ell^2}{M_1^2} t\right )^4\ dt$$
for some
\begin{equation}\label{lime}
 1 \leq L \ll \frac{Q^2 M_1}{H} \ll X^{O(\eps^2)} \frac{M_1}{H}.
\end{equation}

To obtain the best bounds, it becomes convenient to reduce the range of integration of $t$.  Let $S$ be a parameter in the range
\begin{equation}\label{srange}
 M_1 \ll S \ll \min \Big (M_1^2, \frac{TL}{M_1} \Big )
\end{equation}
to be chosen later.  By Lemma \ref{rs1}(i), we then have
$$ \sum_{j=1}^r \tilde{Y}(t_j)^4 \ll_\eps X^{O(\eps^2)} \frac{H TL}{S M_1^2}
\sum_{L \leq \ell < 2L} \int_{0 < \alpha \ll S} f\left( \alpha, \frac{\ell^2}{M_1^2} \alpha\right )^4\ d\alpha$$
Applying \eqref{f-stable}, we then have
$$ \sum_{j=1}^r \tilde{Y}(t_j)^4 \ll_\eps X^{O(\eps^2)} \frac{H TL}{S M_1^2}
\int_{0 < \alpha \ll S}\int_{0 < \beta \ll 1 + \frac{L^2 \alpha}{M_1^2}} f(\alpha,\beta)^4 \mu(\alpha,\beta)\ d\beta d\alpha $$
where the multiplicity $\mu(\alpha,\beta)$ is defined as the number of integers $\ell \in [L, 2L)$ such that
$$ \left|\beta - \frac{\ell^2}{M_1^2} \alpha\right| \leq 1.$$
Clearly we have the trivial bound $\mu(\alpha,\beta) \ll L$.  On the other hand, for fixed $\alpha$, the numbers $\frac{\ell^2}{M_1^2} \alpha$ are $\asymp \frac{L \alpha}{M_1^2}$-separated, so for $\beta \gg 1$ we also have the bound
$$ \mu(\alpha,\beta) \ll 1 + \frac{M_1^2}{L \alpha}.$$
We thus have
$$ \sum_{j=1}^r \tilde{Y}(t_j)^4 \ll_\eps X^{O(\eps^2)} \frac{H TL}{S M_1^2} (W_1 + W_2 + W_3)$$
where
\begin{align*}
 W_1 &\coloneqq L \int_{0 < \alpha \ll S} \int_{0 < \beta \ll 1} f(\alpha,\beta)^4\ d\beta d\alpha \\
 W_2 &\coloneqq \int_{0 < \alpha \ll S} \frac{M_1^2}{L \alpha} \int_{1 \ll \beta \ll \frac{L^2 \alpha}{M_1^2}} f(\alpha,\beta)^4\ d\beta d\alpha \\
 W_3 &\coloneqq \int_{0 < \alpha \ll S} \int_{1 \ll \beta \ll \frac{L^2 \alpha}{M_1^2}} f(\alpha,\beta)^4\ d\beta d\alpha.
\end{align*}
From \eqref{f-stable} and Lemma \ref{rs1}(ii) (with $\theta \coloneqq -1$ and $a_m \coloneqq e \Big ( \frac{\beta}{3 \pi} \Big ( \frac{M_1}{m} \Big )^3 \Big )$) we have
$$ W_1 \ll_\eps X^{O(\eps^2)} L (M_1^4 + M_1^2 S) \ll X^{O(\eps^2)} L M_1^4$$
thanks to \eqref{srange}.  Now we treat $W_2$.  We may assume that $\alpha \gg M_1^2/L^2$ since the inner integral vanishes otherwise.  By the pigeonhole principle, we thus have
$$ W_2 \ll \frac{M_1^2 \log X}{LA} \int_{\alpha \asymp A} \int_{\beta \ll \frac{L^2 A}{M_1^2}}  f(\alpha,\beta)^4\ d\beta d\alpha$$
for some $\frac{M_1^2}{L^2} \ll A \ll S$.  Applying Lemma \ref{rs1}(ii) again (now treating the $e(\frac{\beta}{3\pi} \frac{M_1^3}{m^3})$ term in \eqref{fab} as a bounded coefficient $a_m$) we have
$$ \int_{\alpha \asymp A} f(\alpha,\beta)^4\ d\alpha \ll_\eps X^{O(\eps^2)} (M_1^4 + M_1^2 A) \ll X^{O(\eps^2)} M_1^4$$
and hence
$$ W_2 \ll X^{O(\eps^2)} L M_1^4.$$
Finally we turn to $W_3$.  The contribution of the region $\alpha \ll \frac{M_1^2}{L}$ is $O(W_2)$.  Thus by the pigeonhole principle we have
\begin{equation}\label{trad}
 W_3 \ll W_2 + \log X \int_{\alpha \asymp A} \int_{\beta \ll \frac{L^2 A}{M_1^2}} f(\alpha,\beta)^4\ d\beta d\alpha
\end{equation}
for some $\frac{M_1^2}{L} \ll A \ll S$.  In particular (from \eqref{lime}, \eqref{srange}) one has
\begin{equation}\label{mam}
M_1 \ll A \ll M_1^2.
\end{equation}
One could estimate the integral here using Lemma \ref{rs1}(ii) once again, but this turns out to lead to an inferior estimate if used immediately, given that the length $M_1$ of the exponential sum and the dominant frequency scale $A$ lie in the range \eqref{mam}; indeed, this only lets one establish Theorem \ref{unav-corr} for $\sigma = 1/4$.  Instead, we will first apply the van der Corput $B$-process (Lemma \ref{rs1}(iii)), which morally speaking will shorten the length from $M_1$ to $A/M_1$, at the cost of applying a Legendre transform to the phase in the exponential sum.  

We turn to the details.  For a fixed $\alpha,\beta$ with 
\begin{equation}\label{abs}
\alpha \asymp A;\quad \beta \ll \frac{L^2 A}{M_1^2}
\end{equation}
(so in particular $\beta$ is much smaller in magnitude than $\alpha$, thanks to \eqref{lime}), let 
$$ \phi(x) \coloneqq \frac{\alpha}{\pi} \frac{M_1}{x} + \frac{\beta}{3\pi} \frac{M_1^3}{x^3}$$
denote the phase appearing in \eqref{fab}. The first derivative is given by
$$ \phi'(x) = - \frac{\alpha}{\pi} \frac{M_1}{x^2} - \frac{\beta}{\pi} \frac{M_1^3}{x^4}.$$
This maps the region $\{ x: x \asymp M_1\}$ diffeomorphically to a region of the form $\{ t: -t \asymp \frac{A}{M_1} \}$.  Denoting the inverse map by $u$, we thus have
$$ t = - \frac{\alpha}{\pi} \frac{M_1}{u(t)^2} - \frac{\beta}{\pi} \frac{M_1^3}{u(t)^4} $$
for $-t \asymp \frac{A}{M_1}$.  

One can solve explicitly for $u(t)$ using the quadratic formula as
\begin{align*} u(t)^2 &= \frac{1}{2} \left( \frac{\alpha M_1}{\pi |t|} + \sqrt{\left(\frac{\alpha M_1}{\pi |t|}\right)^2 + 4 \frac{\beta M_1^3}{\pi |t|}} \right).\\
&= \left(\frac{\alpha M_1}{\pi |t|}\right) \frac{1}{2} \left( 1 + \left(1 + \frac{4\beta}{\alpha} \frac{\pi |t| M_1}{\alpha}\right)^{1/2} \right).
\end{align*}
A routine Taylor expansion then gives the asymptotic
$$ u(t) = M_1 \left(\frac{\alpha}{\pi |t| M_1}\right)^{1/2} + \frac{\beta}{2\alpha} M_1 \left(\frac{\alpha}{\pi |t| M_1}\right)^{-1/2} + R_{\alpha,\beta}(t)$$
where the remainder term $R_{\alpha,\beta}(t)$ obeys the estimates
$$ R_{\alpha,\beta}(t) \ll \frac{\beta^2}{\alpha^2} M_1; \quad R'_{\alpha,\beta}(t) \ll \frac{\beta^2}{\alpha^2} \frac{M_1^2}{A}$$
for $-t \asymp \frac{A}{M}$.  The (negative) Legendre transform $\phi^*(t) \coloneqq \phi(u(t)) - t u(t)$ can then be similarly expanded as
$$ \phi^*(t) = \frac{2\alpha}{\pi} \left(\frac{\alpha}{\pi |t| M_1}\right)^{-1/2} + \frac{\beta}{3\pi} \left(\frac{\alpha}{\pi |t| M_1}\right)^{-3/2} + E_{\alpha,\beta}(t)$$
where the error term $E_{\alpha,\beta}$ obeys the estimates
$$ E_{\alpha,\beta}(t) \ll \frac{\beta^2}{\alpha^2} A; \quad E'_{\alpha,\beta}(t) \ll \frac{\beta^2}{\alpha^2} M_1.$$
From \eqref{abs}, \eqref{lime}, \eqref{mam} we have
$$ \frac{\beta^2}{\alpha^2} A \ll \frac{L^4}{M_1^4} A \ll \frac{X^{O(\varepsilon^2)}}{H^4} M_1^2 \ll 1 $$
where the last bound follows from \eqref{sigma-range} since $H = X^{\sigma + \eps}$ and $M_1 \ll X^{1/3}$.  Applying Lemma \ref{rs1}(iii) followed by Lemma \ref{sbp-2}, we conclude that
$$ f(\alpha,\beta) \ll \frac{M_1}{A^{1/2}} g(A^{1/2} \alpha^{1/2}, A^{3/2} \alpha^{-3/2} \beta) + M_1^{1/2}$$
where $g$ is the maximal exponential sum
$$ g(\alpha',\beta') \coloneqq \left| \sum_{-\ell \asymp A/M_1} e\left( \frac{2\alpha'}{\pi^{1/2}} \left( \frac{|\ell|}{A/M_1} \right)^{1/2} + \frac{\beta' \pi^{1/2}}{3} \left( \frac{|\ell|}{A/M_1} \right)^{3/2} \right)\right|^*.$$
Inserting this back into \eqref{trad} and performing a change of variables, we conclude that
$$
 W_3 \ll W_2 + \log X \left( L^2 A^2 + \frac{M_1^4}{A^2} \int_{\alpha \asymp A} \int_{\beta \ll \frac{L^2 A}{M_1^2}} g(\alpha,\beta)^4\ d\beta d\alpha \right).$$
On the other hand, by applying Lemma \ref{rs1}(ii) as before we have
$$
\int_{\alpha \asymp A} g(\alpha,\beta)^4\ d\alpha \ll_\eps X^{O(\eps^2)} ( (A/M_1)^4 + (A/M_1)^2 A ) \ll X^{O(\eps^2)} (A/M_1)^2 A $$
for any $\beta$, where the last inequality follows from \eqref{mam}.  We thus arrive at the bound
$$ W_3 \ll_\eps W_2 + X^{O(\eps^2)} L^2 A^2.$$
Since $A \ll S$, we thus have
$$ W_3 \ll_\eps W_2 + X^{O(\eps^2)} L^2 S^2.$$
Combining all the above bounds for $W_1,W_2,W_3$, we have
\begin{equation}\label{yt4}
 \sum_{j=1}^r \tilde{Y}(t_j)^4 \ll_\eps X^{O(\eps^2)} \frac{H TL}{S M_1^2} (L M_1^4 + L^2 S^2).
\end{equation}
To optimize this bound we select 
$$S \coloneqq \min\left( \frac{M_1^2}{L^{1/2}}, \frac{TL}{M_1} \right).$$
It is easy to see (using \eqref{l6}, \eqref{lime}, and \eqref{mi3}) that $S$ obeys the bounds \eqref{srange}.  From \eqref{yt4} we have
\begin{align*}
\sum_{j=1}^r \tilde{Y}(t_j)^4 &\ll_\eps X^{O(\eps^2)} \left( \frac{H TL^2 M_1^2}{S} + \frac{HTL^3 S}{M_1^2} \right) \\
&\ll_\eps X^{O(\eps^2)} ( H L M_1^3 + HTL^{5/2} ).
\end{align*}
Applying \eqref{lime}, \eqref{tbound} we thus have
$$ \sum_{j=1}^r \tilde{Y}(t_j)^4 \ll_\eps X^{O(\eps^2)} ( M_1^4 + H^{-3/2} \lambda X M_1^{5/2} ) $$
and hence by \eqref{mi3}
$$ \sum_{j=1}^r \tilde{Y}(t_j)^4 \ll_\eps X^{O(\eps^2)} ( X^{4/3} + H^{-3/2} \lambda X^{11/6} ) $$
From \eqref{H-def}, \eqref{sigma-range} we have $H \geq X^{\frac{11}{48}+\eps}$, which implies after some arithmetic and \eqref{l6} that the $X^{4/3}$ term here gives an acceptable contribution to Proposition \ref{pq}.  The $H^{-3/2} \lambda X^{11/6}$ term is similarly acceptable thanks to \eqref{l6}, \eqref{H-def}, and \eqref{sigmadef}.

%\begin{remark}\label{29-rem} The above arguments show that all cases of Theorem \ref{types} other than the Type $d_4$ case continue to hold if $\record$ is replaced by $2/9$.  In particular, this allows one to replace $\record$ by $2/9$ in the $k=l=3$ case of Theorem \ref{unav-corr}, since one can verify in this case that the Type $d_4$ sums do not actually appear in the combinatorial decompositions.
%\end{remark}

\appendix

\section{Mean value estimate}\label{harman-sec}

In this section we prove Proposition \ref{harm-prop}.  This estimate is fairly standard, for instance following from the methods in \cite[Chapter 9]{harman}; for the convenience of the reader we sketch a full proof here.

Let $\eps, A, B, X, q_0, q_1, f, B'$ be as in Proposition \ref{harm-prop}.  
We first invoke Lemma \ref{comb-decomp-2} with $m=3$, and with $\eps$ and $H$ replaced by $\eps/10$ and $X^{1/3+\eps/10}$ respectively.  We conclude that the function $(\chi,t) \mapsto {\mathcal D}[f](\frac{1}{2}+it, \chi, q_0)$ can be decomposed as a linear combination (with coefficients of size $O_{k,\eps}( d_2(q_0)^{O_{k,\eps}(1)} )$) of $O_{k,\eps}( \log^{O_{k,\eps}(1)} X)$ functions of the form $(\chi,t) \mapsto {\mathcal D}[\tilde f](\frac{1}{2}+it,\chi)$, where $\tilde f \colon \N \to \C$ is one of the following forms:

\begin{itemize}
\item[(Type $d_1$, $d_2$ sum)]  A function of the form
\begin{equation}\label{fst-again	}
 \tilde f = (\alpha \ast \beta_1 \ast \dots \ast \beta_j) 1_{(X/q_0,2X/q_0]}
\end{equation}
for some arithmetic functions $\alpha,\beta_1,\dots,\beta_j \colon \N \to \C$, where $j=1,2$, $\alpha$ is $O_{k,\eps}(1)$-divisor-bounded and supported on $[N,2N]$, and each $\beta_i$, $i=1,\dots,j$ is either of the form $\beta_i = 1_{(M_i,2M_i]}$ or $\beta_i = L 1_{(M_i,2M_i]}$ for some $N, M_1,\dots,M_j$ obeying the bounds 
$$ 1 \ll N \ll_{k,\eps} X^{\eps/10},$$
$$ N M_1 \dots M_j \asymp_{k,\eps} X/q_0,$$
and
$$ X^{1/3+\eps/10} \ll M_1 \ll \dots \ll M_j \ll X/q_0.$$ 
\item[(Type II sum)]  A function of the form
$$ \tilde f = (\alpha \ast \beta) 1_{(X/q_0,2X/q_0]}$$
for some $O_{k,\eps}(1)$-divisor-bounded arithmetic functions $\alpha,\beta \colon \N \to \C$ with good cancellation supported on $[N,2N]$ and $[M,2M]$ respectively, for some $N,M$ obeying the bounds 
$$ X^{\eps/10} \ll N \ll X^{1/3+\eps/10}$$
and
$$NM \asymp_{k,\eps} X / q_0.$$
The good cancellation bounds \eqref{alb} are permitted to depend on the parameter $B$ appearing in the bound $q_0 \leq \log^B X$.
\item[(Small sum)]  A function $\tilde f$ supported on $(X/q_0,2X/q_0]$ obeying the bound
\begin{equation}\label{smallsum-0}
 \|\tilde f \|_{\ell^2}^2 \ll_{k,\eps} X^{1-\eps/80}.
\end{equation}
\end{itemize}

By the $L^2$ triangle inequality (and enlarging $A$ as necessary), it thus suffices to establish the bound 
\begin{equation}\label{drip}
\int_{\log^{B'} X \leq |t| \leq X^{5/6 - \eps} } \left|{\mathcal D}[\tilde f]\left(\frac{1}{2}+it',\chi,q_0\right)\right|^2\ dt' \ll_{k,\eps,A,B,B'} X \log^{-A} X.
\end{equation}
for each individual character $\chi$ and $\tilde f$ one of the above forms.  % {\bf to make $Q$ larger we have to adapt the argument below to allow summation in $\chi$ efficiently.}

We first dispose of the small sum case.  From Lemma \ref{mvt-lem} and \eqref{smallsum-0} we have
$$ 
\int_{|t| \leq X^{5/6 - \eps}} \left|{\mathcal D}[\tilde f] \left (\frac{1}{2}+it,\chi \right )\right|^2\ dt \ll_k X^{1-\eps/80} \log^{O_k(1)} X.$$
which gives \eqref{drip} (with a power savings).

In the remaining Type $d_1$, Type $d_2$, and Type II cases, $\tilde f$ is of the form $\tilde f = f' 1_{(X/q_0,2X/q_0]}$, where $f'$ is of the form $\alpha \ast \beta_1$, $\alpha \ast \beta_1 \ast \beta_2$, or $\alpha \ast \beta$ in the Type $d_1$, Type $d_2$, and Type II cases respectively.  From Corollary \ref{trunc-dir} one has
$$ |{\mathcal D}[\tilde f](\frac{1}{2}+it,\chi)| \ll \int_{|u| \leq X^{5/6-\eps}} \left |{\mathcal D}[f'] \left (\frac{1}{2}+it+iu,\chi \right )\right| \frac{du}{1+|u|} + 1.$$
Meanwhile, from Lemma \ref{mvt-lem} we have
$$ \int_{|t'| \leq \frac{1}{2} \log^{B'} X} \left|{\mathcal D}[\tilde f]\left (\frac{1}{2}+it',\chi \right )\right|^2\ dt' \ll_{k,\eps} X \log^{O_{k,\eps}(1)} X $$
and hence by Cauchy-Schwarz
$$ \int_{|t'| \leq \frac{1}{2} \log^{B'} X} \left|{\mathcal D}[\tilde f]\left (\frac{1}{2}+it',\chi \right )\right|\ dt' \ll_{k,\eps} X^{1/2} \log^{O_{k,\eps}(1) + B'/2} X.$$
We conclude that
$$ {\mathcal D}[\tilde f]\left (\frac{1}{2}+it,\chi \right ) \ll_{k,\eps} \int_{\frac{1}{2} \log^{B'} X \leq |t'| \leq 2X^{5/6-\eps}} \left|{\mathcal D}[f'] \left (\frac{1}{2}+it',\chi \right )\right| \frac{dt'}{1+|t-t'|} + 1 + \frac{X^{1/2} \log^{O_{k,\eps}(1) + B'/2} X}{|t|} $$
for $\log^{B'} X \leq |t| \leq X^{5/6-\eps}$; by Cauchy-Schwarz, one thus has
$$ \left|{\mathcal D}[\tilde f]\left (\frac{1}{2}+it,\chi \right )\right|^2 \ll_{k,\eps} \int_{\frac{1}{2} \log^{B'} X \leq |t'| \leq 2X^{5/6-\eps}} \left |{\mathcal D}[f']\left (\frac{1}{2}+it',\chi\right )\right|^2 \frac{dt'}{1+|t-t'|} \log X + 1 + \frac{X \log^{O_{k,\eps}(1) + B'} X}{|t|^2}.$$
Integrating in $t$, we can bound the left-hand side of \eqref{drip} for $\tilde f$ by
$$ 
\ll_{k,\eps}
\int_{\frac{1}{2} \log^{B'} X \leq |t| \leq 2 X^{5/6 - \eps}} \left|{\mathcal D}[f']\left (\frac{1}{2}+it,\chi\right )\right|^2\ dt  \log^2 X
+ X \log^{O_{k,\eps}(1) - B'} X,$$
so (by taking $B'$ large enough) it suffices to establish the bounds
\begin{equation}\label{das}
\int_{\frac{1}{2} \log^{B'} X \leq |t| \leq 2X^{5/6 - \eps}} \left|{\mathcal D}[f'] \left (\frac{1}{2}+it,\chi \right )\right|^2\ dt \ll_{k,\eps, A, B, B'} X \log^{-A} X
\end{equation}
in the Type $d_1$, Type $d_2$, and Type II cases.

We first treat the Type $d_2$ case.  
By dyadic decomposition it suffices to show that
$$
\int_{T/2 \leq |t| \leq T} \left|{\mathcal D}[f']\left (\frac{1}{2}+it,\chi\right)\right|^2\ dt \ll_{k,\eps,A,B,B'} d_2(q_1)^{O_k(1)} X \log^{O_{k,\eps,B}(1) - 2A} X $$
for all $\log^{B'} X \leq T \ll X^{5/6-\eps}$.  From Corollary \ref{fourth-integral} we have
$$ \int_{T/2 \leq |t| \leq 2T} \left|{\mathcal D}[\beta_j]\left (\frac{1}{2}+it,\chi\right)\right|^4\ dt \ll_B
T \left( 1 + \frac{M_j^2}{T^4} \right) \log^{O_B(1)} X$$
for $j=1,2$; also from \eqref{dirt} we have the crude bound
$$ {\mathcal D}[\alpha]\left (\frac{1}{2}+it,\chi\right ) \ll_{k,\eps} N^{1/2} \log^{O_{k,\eps}(1)} X.$$
Thus by Cauchy-Schwarz we may bound
$$
\int_{T/2 \leq |t| \leq T} \left|{\mathcal D}[f']\left (\frac{1}{2}+it\right )\right|^2\ dt 
\ll_{k,\eps} N T \left(1 + \frac{M_1}{T^2} \right) \left(1 + \frac{M_2}{T^2} \right)
\log^{O_{k,\eps,B}(1)} X.$$
We can bound
$$ \left(1 + \frac{M_1}{T^2} \right) \left(1 + \frac{M_2}{T^2} \right) \ll 1 + \frac{M_1 M_2}{T^2}$$
and use $N M_1 M_2 \ll X$ to conclude
$$
\int_{T/2 \leq |t| \leq T} \left |{\mathcal D}[f']\left (\frac{1}{2}+it,\chi\right) \right |^2\ dt 
\ll_{k,\eps,B} \left(NT + \frac{X}{T} \right) \log^{O_{k,\eps,B}(1)} X$$
which is acceptable since $\log^{B'} X \leq T \ll X^{5/6-\eps}$ and $N \ll X^{\eps/10}$.

The Type $d_1$ case can be treated similarly to the Type $d_2$ case (with the role of $\beta_2(n)$ now played by the Kronecker delta function $\delta_{n=1}$).  It thus remains to handle the Type II case.  Here we factor
$$ {\mathcal D}[f']\left(\frac{1}{2}+it,\chi\right) = {\mathcal D}[\alpha]\left(\frac{1}{2}+it,\chi\right) {\mathcal D}[\beta]\left(\frac{1}{2}+it,\chi\right).$$
and hence
$$ {\mathcal D}[f']\left(\frac{1}{2}+it,\chi\right)^2 = {\mathcal D}[\beta]\left(\frac{1}{2}+it,\chi\right) {\mathcal D}[\beta]\left(\frac{1}{2}+it,\chi\right) {\mathcal D}[\alpha \ast \alpha]\left(\frac{1}{2}+it,\chi\right).$$
At this point it is convenient to invoke an estimate of Harman (which in turn is largely a consequence of Huxley's large values estimate and standard mean value theorems for Dirichlet polynomials), translated into the notation of this paper:

\begin{lemma}  Let $X \geq 2$ and $\eps>0$, and let $M,N,R \geq 1$ be such that $M = X^{2\alpha_1}$, $N = X^{2\alpha_2}$, and $MNR \asymp X$ for some $\alpha_1, \alpha_2 > 0$ obeying the bounds
$$ |\alpha_1 - \alpha_2| < \frac{1}{6}  + \eps$$
and
$$ \alpha_1 + \alpha_2 > \frac{2}{3} - \eps.$$
Let $a,b,c \colon \N \to \C$ be $O_{k,\eps}(1)$-divisor-bounded arithmetic functions supported on $[M/2,2M]$, $[N/2,2N]$, $[R/2,2R]$ respectively obeying the bounds
$$ a(n), b(n), c(n) \ll_{k,\eps} d_2(n)^{O_{k,\eps}(1)} \log^{O_{k,\eps}(1)} X$$
for all $n$.  Suppose also that $c$ has good cancellation.  Then we have
$$ \int_{\log^B X \leq |t| \leq X^{\frac{5}{6}-\eps}} \left|{\mathcal D}[a]\left(\frac{1}{2}+it\right) {\mathcal D}[b]\left(\frac{1}{2}+it\right) {\mathcal D}[c]\left(\frac{1}{2}+it\right)\right|\ dt \ll_{k,\eps,A,B} X \log^{-A} X$$
whenever $A>0$ and $B$ is sufficiently large depending on $A$.
\end{lemma}

\begin{proof}  Apply \cite[Lemma 7.3]{harman} with $x \coloneqq  X^2$ and $\theta \coloneqq  \frac{7}{12} + \frac{\eps}{2}$ (so that the quantity $\gamma(\theta)$ defined in \cite[Lemma 7.3]{harman} is at least as large as $\frac{1}{3} + 2\eps$).  Strictly speaking, the hypotheses in \cite[Lemma 7.3]{harman} restricted $|t|$ to be at least $\exp(\log^{1/3} X)$ rather than $\log^B X$, but one can check that the argument is easily modified to adapt to this new lower bound on $|t|$.
\end{proof}

If we apply this lemma with $a \coloneqq  \beta \chi$, $b \coloneqq  \beta \chi$, $c \coloneqq  (\alpha \ast \alpha)\chi$ (with $\alpha_1 = \alpha_2 \geq \frac{1}{3} - \frac{\eps}{20} + o(1)$) using Lemma \ref{good-cancel} to preserve the good cancellation property, we conclude that
$$ \int_{\log^{B'} X \leq |t| \leq X^{\frac{5}{6}-\eps}} \left|{\mathcal D}[f']\left(\frac{1}{2}+it\right)\right|^2\ dt \ll_{k,\eps,A,B} X \log^{-A} X$$
giving \eqref{das} in the Type II case.

\end{document}